\documentclass[11pt]{amsart}

\usepackage{AFBoix2015}

\begin{document}

\author[J. \`Alvarez Montaner]{Josep \`Alvarez Montaner}

\address{Departament de Matem\`atiques, Universitat Polit\`ecnica de Catalunya, Avinguda Diagonal 647,
Barcelona 08028, SPAIN}
\email{Josep.Alvarez@upc.edu}

\author[A.\,F.\,Boix]{Alberto F.\,Boix}

\address{Department of Mathematics, Ben Gurion University of the Negev, P.O.B. 653 Beer-Sheva 8410501, ISRAEL.}
\email{fernanal@post.bgu.ac.il}

\author[S. Zarzuela]{Santiago Zarzuela}

\address{Departament d'\`Algebra i Geometria, Universitat de Barcelona, Gran Via de les Corts Catalanes 585, Barcelona 08007, SPAIN} \email{szarzuela@ub.edu}

\thanks{J.A.M. is supported by Generalitat de Catalunya 2014SGR-634 project
and Spanish Ministerio de Econom\'ia y Competitividad MTM2015-69135-P.
A.F.B. is supported by Israel Science Foundation (grant No. 844/14) and Spanish Ministerio de Econom\'ia y Competitividad MTM2016-7881-P. S.Z. is supported by Spanish
Ministerio de Econom\'ia y Competitividad MTM2016-7881-P}

\keywords{Local cohomology; Spectral sequences; Hochster's formula; Regularity of linear varieties; Toric face rings}

\subjclass[2010]{Primary 13D45; Secondary 13A35, 13F55, 14N20, 18G40}

\begin{abstract}
We introduce a formalism to produce several families of spectral sequences involving the derived functors of
the limit and colimit functors over a finite partially ordered set. 
The first type of spectral sequences involves the left derived functors of the colimit of the direct system that
we obtain applying a family of functors to a single module. For the second type  we follow a completely
different strategy as we start with the inverse system that we obtain by applying a covariant functor to an inverse
system. The spectral sequences involve the right derived functors of the corresponding limit. We also have a version
for contravariant functors. In all the introduced spectral sequences we provide sufficient conditions to ensure their degeneration at their second page. As a consequence we obtain some decomposition theorems that greatly generalize
the well-known decomposition formula for local cohomology modules of Stanley--Reisner rings given by Hochster.
\end{abstract}

\title{On some local cohomology spectral sequences}

\maketitle

\dedicatory{To Jan-Erik Roos (1935-2017).}

\tableofcontents

\section{Introduction}

Let $A$ be a commutative Noetherian ring, let $J\subseteq A$ be an ideal, and let $H_J^r (A)$ be the $r$th local cohomology group of $A$ supported on $J$. When $A$ is $\Z^n$-graded (for some $n\in\N$) and $J$ is homogeneous, it is of interest to understand the graded pieces of $H_J^r (A)$, because of its connection with the cohomology of coherent sheaves of algebraic varieties given by the Deligne-Grothendieck-Serre correspondence \cite[20.3.16 (iv) and 20.4.4]{BroSha}. In case $A$ is a Stanley-Reisner ring with coefficients on a field $\K$, and $\mathfrak{m}$ is its irrelevant ideal, these graded pieces are encoded inside the celebrated Hochster's formula for Hilbert series of local cohomology modules $H_{\mathfrak{m}}^r (A)$ of Stanley-Reisner rings (see \cite[Theorem II.4.1]{Stanley1996}, \cite[Theorem 5.3.8]{BrunsHerzog1993} and \cite[Theorem 13.13]{MillerSturmfels2005}); it turns out that this decomposition preserves the grading of 
these local cohomology modules, by means of a formula proved by Gr\"abe \cite[Theorem 2]{Grabe1984}. On the other hand, when $A$ is a polynomial ring with coefficients on a field $\K$, and $J$ is a squarefree monomial ideal, the graded pieces of $H_J^r (A)$ are explicitly described by means of results obtained by Musta{\c{t}}{\v{a}} \cite[Theorem 2.1 and Corollary 2.2]{Mustata2000symbolic} (see also \cite{EisenbudMustataStillman2000}) and the corresponding formula for the Hilbert series was
calculated by Terai \cite{Teraiunpublished} (see also \cite[Corollary 13.16]{MillerSturmfels2005}); both decompositions preserve the grading, as proved by Musta{\c{t}}{\v{a}} in \cite{Mustata2000symbolic}. It turns out that the results by Hochster (respectively, Terai) and Gr\"abe (respectively, Musta{\c{t}}{\v{a}}) are equivalent, by means of the duality established by Miller \cite[Corollary 6.7]{Miller2000}. It is also worth noting that, in the last decade, several generalizations of Hochster's decomposition have appeared in the literature, being \cite[Theorem 1]{Takayama2005}, \cite[Theorem 5.1]{
EnescuHochster2008}, \cite[Theorems 1.1 and 1.3]{BrunBrunsRomer2007} and \cite[Theorem 4.5 and Lemma 4.6]{IchimRomer2007} the most relevant for the purposes of this paper (see also \cite[Proposition 3.2 and Corollary 3.4]{MaclaganSmith2004}, \cite[Theorem 13.14]{MillerSturmfels2005}, \cite[Theorem 1.3]{Rahimi2007}, \cite[Corollary 6.3]{OkazakiYanagawa2009},
\cite[Theorems 4.37 and 4.40]{Alberellithesis} and
\cite[Theorems 4.24 and 4.26]{AlberelliBrenner}).

Another different approach was adopted by \`Alvarez Montaner, Garc\'ia L\'opez and Zarzuela Armengou in \cite{AlvarezGarciaZarzuela2003}; indeed, building upon a Mayer-Vietoris spectral sequence
\begin{equation}\label{el objetivo inicial}
E_2^{-i,j}=\Lf_i \colim_{p\in P} H_{I_p}^j (A)\xymatrix{\hbox{}\ar@{=>}[r]_-{i}& }H_I^{j-i}(A),
\end{equation}
(where $A$ is a polynomial ring with coefficients on a field, $I$ is the defining ideal of an arrangement of linear varieties, $I=I_1\cap\ldots\cap I_n$ is a primary decomposition, and $P$ is the partially ordered set made up by all the possible different sums of the ideals $I_1,\ldots ,I_n$ ordered by reverse inclusion, so each element $p\in P$ is identified with one of these sums, which one denotes by $I_p$) they show that, under these assumptions, the spectral sequence  \eqref{el objetivo inicial} degenerates at its second page; along the way, they provide \cite[Corollary 1.3]{AlvarezGarciaZarzuela2003} a closed formula for topological Betti numbers of complements of arrangements of linear varieties, which may be regarded as a local cohomology interpretation of the celebrated Goresky-MacPherson formula \cite[III.1.3.Theorem A]{GoreskyMacPhersonbook}. Finally, by studying 
the extension problems of the filtration produced by this degeneration when $I$ is a squarefree monomial ideal, they show that these extensions are, in general, non--trivial, and
are determined by the $\Z^n$--graded structure of $H_I^j(A).$

Assume for a while that $\K$ is any field, $A=\K[\![x_1,\ldots ,x_n]\!]$, $\mathfrak{m}=(x_1,\ldots ,x_n),$ and $J\subseteq A$ is any ideal;
one might ask about the existence of a spectral sequence
\begin{equation}\label{el objetivo final}
E_2^{i,j}=\R^i \lim_{p\in P} H_{\mathfrak{m}}^j\left(A/I_p\right)\xymatrix{\hbox{}\ar@{=>}[r]_-{i}&}H_{\mathfrak{m}}^{i+j}\left(A/I\right).
\end{equation}
Unfortunately, such a spectral sequence can not exist; indeed, otherwise it would collapse by a result of Jensen \cite[7.1]{Jensen1972}, and therefore it would produce a decomposition of $H_{\mathfrak{m}}^r (A/I)$ which would contradict Hochster's one. It is worth noting that, under the additional assumption
that $\K$ has prime characteristic,   Lyubeznik introduced in \cite[Section 4]{Lyu1997} a functor $\mathcal{H}:=\mathcal{H}_{A,A/J}$ that has the following behaviour with respect to local cohomology modules:
\[
\mathcal{H}\left(H_{\mathfrak{m}}^{n-r}\left(A/J\right)\right)\cong H_J^r \left(A\right),
\]
so it was natural to ask about the existence of a spectral sequence that would correspond to \eqref{el objetivo inicial} under $\mathcal{H}.$

The first goal of this paper is to build a formalism to construct several spectral sequences that, on the one hand, recover and extend the Mayer-Vietoris spectral sequence \eqref{el objetivo inicial} (see also \cite[Theorem 2.1]{Lyu2006}) for other types of functors 
and, on the other hand, to fix and generalize the false spectral sequence \eqref{el objetivo final}. Broadly speaking, we produce two types of spectral sequences; the first ones are associated to a single module; given this module, one applies a family of functors 
that produce a direct system. Since the second page of these spectral sequences involves the left derived functors of the colimit, we will refer to them as \emph{homological spectral sequences}.
In contrast, the second type of spectral sequences we produce 
are attached to an inverse system of modules; indeed, given such an inverse system, we apply a single functor to produce another inverse system. 
Since the second page of these spectral sequences involves the right derived functors of the (inverse) limit, we refer to them as \emph{cohomological spectral sequences}.

The second goal is to study the degeneration of these spectral sequences at their corresponding second pages
and the filtration that this degeneration provides. From this filtration we get a collection of
short exact sequences for which we study its extension problems. As a consequence of this study we produce
some general Hochster type decompositions not only for local cohomology modules but also for 
more general functors.

The organization of this article is as follows.
Since all the spectral sequences which appear in this manuscript
involve the left (respectively, the right) derived functors of the
colimit (respectively, limit) functors, we review in Section
\ref{preliminares de sistemas directos y inversos} the facts we need
later on about the categories of direct and inverse systems over finite partially ordered sets; albeit
most of the material presented in Section \ref{preliminares de sistemas
directos y inversos} is known, we present, as far as possible, a
self contained treatment for the convenience of the reader. The
derived functors of the colimit and limit can be described by means
of the Roos complexes that we recall in Section \ref{Roos}.

In Section \ref{construccion de libros} we introduce our first family of homological spectral sequences. 
Given an $A$-module $M$ and a set of functors satisfying some technical conditions 
(see Setup \ref{las hipotesis del caso homologico}) we construct a direct system 
$\{T_p(M) \}_{p\in P}$ over a finite poset $P$ with a distinguished final element $T(M)$. 
The following generalization 
of the Mayer-Vietoris spectral sequence of local cohomology modules 
given in \cite{AlvarezGarciaZarzuela2003} and \cite{Lyu2006} is the main result 
of this section.

{\bf Theorem \ref{una primera construccion para calentar}.}
Under the previous assumptions we have the following spectral
sequence
\[
E_2^{-i,j}=\Lf_i \colim_{p\in P} \R^j T_{[*]}
(M)\xymatrix{\hbox{}\ar@{=>}[r]_-{i}& \R^{j-i} T (M).}
\]

This is a spectral sequence in the category of $A$-modules but it can be 
enlarged with an extra structure as graded modules, $D$-modules, $F$-modules
or modules over the group ring (see {\bf Theorem} \ref{enhanced1} and the examples
listed afterwards). Several examples 
are shown in Subsection \ref{examples_typeI} including 
generalized local cohomology modules, ideal transforms and local cohomology associated
to pairs of ideals. In particular, this is the first time, to the best of our knowledge, that a Mayer Vietoris sequence 
for this later example appears in the literature. We point out that the technical conditions
imposed on the functors are essential in the construction of the spectral sequence and, for example,
it does not apply to the contravariant Hom functor.

Moreover, in the spirit of \cite{AlvarezGarciaZarzuela2003} we provide sufficient conditions to guarantee the degeneration  of the spectral sequences at their corresponding second page (see \textbf{Theorem} \ref{extension del trabajo de master}). 
The associated filtration of $\R^{r} T (M)$ that we obtain is described in {\bf Corollary} \ref{habemus filtracion}, as well
as an enriched version (see \textbf{Corollary} \ref{habemus filtracion enriquecida}). In Subsection \ref{adding structure in the homological case} we use this filtration to recover and extend  the formula for the characteristic cycles of local cohomology modules obtained in \cite{AlvarezGarciaZarzuela2003}; 
we also produce a closed formula for certain generalized Lyubeznik numbers 
introduced in \cite{BetancourtWitt2014} and we provide a closed formula for the length of local cohomology modules  
in the category of Lyubeznik's $F$-modules \cite{Lyu1997}.

In Section \ref{de aqui sacaremos la formula de Hochster} we introduce our second family of spectral sequences.
The setup is completely different from the previous one as we start with an inverse system $V:=\{V_p\}_{p\in P}$ over a finite poset $P$ and a single covariant functor $T$ satisfying some technical conditions 
(see Setup \ref{fucntores covariantes hacificados}). Then we produce an inverse system $\mathcal{T}:=\{T(V_p)\}_{p\in P}$. The following result has to be understood as a generalization of the long exact sequence of local cohomology modules.

{\bf Theorem \ref{la sucesion espectral para rematar}.}
Under the previous assumptions we have the following spectral
sequence
\[
E_2^{i,j}= \R^i \lim_{p\in P}\xymatrix{\R^j\mathcal{T}
\left(V\right)\ar@{=>}[r]_-{i}& }\R^{i+j}\left(\lim_{p\in
P}\circ\mathcal{T}\right) \left(V\right).
\]

If, in addition, there is a natural equivalence of functors
$\lim_{p\in P}\circ\mathcal{T}\cong T\circ\lim_{p\in P},$
and $V$ is acyclic with respect to the limit, then the previous spectral sequence can be arranged
in the more accessible manner:
\[
E_2^{i,j}= \R^i \lim_{p\in P}\xymatrix{\R^j\mathcal{T}
\left(V\right)\ar@{=>}[r]_-{i}& }\R^{i+j}T\left(\lim_{p\in
P}V_p\right).
\]

In this case we may also associate an extra structure to this spectral sequence (see {\bf Theorem} \ref{additional structure in the cohomological case}). The examples that we present in Subsection \ref{examples cohomological type II}
include the covariant Hom, generalized local cohomology, generalized ideal transform local cohomology with respect to 
pairs of ideals and inverse systems of ideals. Moreover, changing our initial poset, we may consider local
cohomology of toric face rings. In \textbf{Theorem}  \ref{sucesion espectral y colapso todo en uno} we also
provide sufficient conditions to guarantee the degeneration at the  second page  in this case.

In Section \ref{homological_typeII} we play a similar game with an inverse system $V:=\{V_p\}_{p\in P}$ over a finite poset $P$ and a single contravariant functor $T$. In this way we  produce a direct system $\mathcal{T}:=\{T(V_p)\}_{p\in P}$ and a spectral sequence involving the derived functors of the colimit (see {\bf Theorem} \ref{otra construccion Josep es cansino}). The main example for this case is the contravariant Hom.

In Section \ref{section of Hochster decompositions} we provide some decomposition theorems that
follow naturally in the case that the spectral sequences introduced in the previous sections 
degenerate at the second page. For homological spectral sequences that
were constructed along Section \ref{construccion de libros} we obtain the following result.

\textbf{Theorem} \ref{formula de Terai: no me digas: segunda parte}. 
There is a $\K$-vector space isomorphism
\[
\R^j T(M)\cong\bigoplus_{q\in P} \R^{h_q}
T_q (M)^{\oplus m_{j,q}},
\]
where $h_q$ is the single value where $\R^j T_q (M)$ does not vanish and  $m_{j,q}=\dim_{\K}\widetilde{H}_{h_q-j-1} \left((q,1_{\widehat{P}});\K\right)$ with $(q,1_{\widehat{P}})$ being the order complexes associated to elements of the poset.

Indeed, this decomposition preserves the grading as shown in \textbf{Theorem} \ref{formula de Terai: no me digas: tercera parte}. In the case that $A=\K[x_1,\cdots,x_n]$ is a polynomial ring,
$T=\Gamma_{I},$ and $I$ is a monomial ideal we recover the formula given by Musta{\c{t}}{\v{a}} and Terai 
. We also want to single out here that this
result produces, to the best of our knowledge, a new decomposition (see \textbf{Theorem} \ref{Terai graduado}) for the
defining ideal of Stanley toric face rings over a field of {prime characteristic}.

In the case of cohomological spectral sequences (see 
{\bf Theorem} \ref{formula de Hochster: no me digas: segunda parte} for the case of homological spectral sequences associated to an inverse system, including a decomposition
for the so--called \textbf{deficiency modules}) that
were constructed in Section \ref{de aqui sacaremos la formula de Hochster}, we obtain a Hochster type decomposition.

{\bf Theorem} \ref{formula de Hochster: no me digas}. 
 There is a $\K$-vector space isomorphism
$$
\R^j T\left(\lim_{p\in P} V_p\right)\cong\bigoplus_{q\in P} \R^{d_q}
T (V_q)^{\oplus M_{j,q}},
$$
where $d_q$ is the single value where $\R^j T (V_p)$ does not vanish and $M_{j,q}=\dim_{\K} \widetilde{H}^{j-d_q-1} ((q,1_{\widehat{P}});\K)$. 

This isomorphism preserves the grading as well by \textbf{Theorem} \ref{en categorias semisimples hay formula}.
When $V=\{A/I_p\}$ is the inverse system associated to a Stanley-Reisner ring $A/I$ and $T=\Gamma_{\mathfrak{m}}$ we recover Hochster's formula; more generally, when $V=\{A/I_p\}$ is the inverse system associated to $A/I,$ where $I$ is any monomial ideal we recover Takayama's formula as formulated by Brun
and R\"omer in \cite[Corollary 2.3]{BrunRomer2008}. In addition, this decomposition also recovers the Brun-Bruns-R\"omer decomposition \cite[Theorem 1.3]{BrunBrunsRomer2007} for the so-called \emph{Stanley toric face rings} (see \textbf{Theorem} \ref{BBR graduado}). 
We point out that our result produces, to the best of our knowledge, a new decomposition for local cohomology modules of \textbf{certain} central arrangements of linear varieties over a field (see \textbf{Theorem \ref{Hochster decomposition for subspace arrangements}}). The issue in this case is that, in general, we do not have an isomorphism $\lim_{p\in P} A/I_p \cong A/I$ 
(see \textbf{Remarks \ref{it doesn't work for arrangements} and \ref{it doesn't work for arrangements 2}}).
Whenever the isomorphism holds we produce a closed formula for the
{\it Castelnuovo--Mumford regularity} of these arrangements (see Theorem \ref{thank you, Jack}); as immediate consequence, we obtain an alternative proof, {specific for this kind
of arrangements}, of the so--called {\it Subspace Arrangements Theorem} (see \textbf{Remark} \ref{subspace arrangements theorem}), originally proved
by Derksen and Sidman \cite[Theorem 2.1]{DerksenSidman2002} for {arbitrary central arrangements}.

When $A$ contains a field of positive characteristic, Enescu and Hochster show in \cite[Theorem 5.1]{EnescuHochster2008} that Hochster's decomposition of Stanley-Reisner rings is compatible with the natural Frobenius action
on the local cohomology modules, obtaining in particular that local cohomology modules of Stanley-Reisner rings have only a finite number of {\it F-stable submodules}, i.e. $A$-submodules $N\subseteq H_{\mathfrak{m}}^r (A)$ such that the action of Frobenius maps $N$ into itself. The interest for studying these F-stable submodules arises naturally in the study of singularities of algebraic varieties in prime characteristic, due to its connection with test ideals.
In \textbf{Theorem} \ref{formula de Enescu-Hochster}, we show that Brun-Bruns-R\"omer decomposition of Stanley toric face rings of prime characteristic is compatible with this natural Frobenius action, and therefore 
they only have a finite number of F-stable submodules; in particular, we recover and extend \cite[Theorem 5.1]{EnescuHochster2008}, because Stanley--Reisner rings are a particular case of Stanley toric face rings \cite[Example 2.2 (i)]{HopNguyen2012}.

In Section \ref{section of Grabe formulae}, following the spirit of \cite[Section 3]{AlvarezGarciaZarzuela2003},
we study the extension problems attached to the corresponding filtrations produced by the
degeneration of our  spectral sequences. Our focus is on the cohomological spectral sequence associated
to $T=\Gamma_{\mathfrak{m}}$. The main result is \textbf{Theorem} \ref{generalizamos la formula de Grabe}
that may be regarded as a Gr\"abe's type description formula.

\section{The categories of inverse and direct systems}\label{preliminares de sistemas directos y inversos}
The purpose of this section is to review the facts we shall need
later on about posets, and the categories of inverse and direct systems; we try to present, as far as possible, a self
contained exposition of this topic for the reader's profit.

\vskip 2mm


Let $(P,\leq)$ be a partially ordered set (from now on, \emph{poset} for the sake of brevity). We shall regard $P$ as a
small category which has as objects the elements of $P$ and, given $p,q\in P$, there is one morphism $p\rightarrow q$ if
$p\leq q$. If $P$ contains a unique minimal (respectively, maximal) element then this is called the \emph{initial}
(respectively, \emph{terminal}) element of $P$ and it will be denoted by $0_P$ (respectively, $1_P$). Adding an initial
and a terminal element to $P$ (even in case $P$ have them) we may consider the poset $(\widehat{P},\leq)$, where
$\widehat{P}:=P\cup\{0_{\widehat{P}}, 1_{\widehat{P}}\}$.

Throughout this paper, we mainly consider the following examples of posets.

\begin{ex}\label{el poset de Mayer-Vietoris}
Let $A$ be a commutative Noetherian ring, let $I\subseteq A$ be an ideal such that $I=I_1\cap\ldots\cap I_n$ for
certain ideals $I_1,\ldots ,I_n$ of $A$. In this case, we can produce a poset $P$ in the following way; $P$ will
be the poset given by all the possible different sums of the ideals $I_1,\ldots ,I_n$ ordered by reverse inclusion.
So, any element $p\in P$ is identified with a certain sum of the ideals $I_1,\ldots ,I_n$, which we denote by $I_p$.
\end{ex}

\begin{ex}\label{el poset de BBR}
Given $\Sigma\subseteq\R^d$ a rational pointed fan \cite[page 538]{HopNguyen2012}, we define $P$ as the poset given by all the possible faces of
$\Sigma$ ordered by inclusion; in other words, we identify each face $C\subseteq\Sigma$ with a point $p_C\in P$;
moreover, given faces $C\subseteq C'$, one has $p_C\leq p_{C'}$.
\end{ex}

\begin{ex}\label{el ejemplo de Hochster}
Let $\Delta$ be a simplicial complex. In this case, we define $P$ as the poset made up by the faces of $\Delta$ ordered
by inclusion; indeed, we identify a face $F\subseteq\Delta$ with a point $p_F\in P$ and, given faces $F\subseteq F'$ we
have $p_F\leq p_{F'}$.
\end{ex}
Notice that in all the above examples, the poset $P$ is finite; this will be an assumption that we preserve during the
rest of this paper. Indeed, for this reason we want to single out the following:

\begin{ass}\label{hipotesis de finitud sobre el poset}
Hereafter, we shall always assume that $P$ is a finite poset.
\end{ass}

\subsection{The categories of inverse and direct systems and some equivalences}\label{una de grafos orientados pintados}

Throughout this paper, let $\cA$ be the category of $A$-modules, where $A$ denotes a commutative Noetherian ring

\begin{enumerate}[(a)]

\item A \emph{direct system over $P$ valued on $\cA$}\index{direct system} is a covariant functor
$\xymatrix@1{P\ar[r]^-{F}& \cA}$; in what follows, we will denote either by $F_p$
or $F(p)$ the value of $F$ at $p\in P.$

\item An \emph{inverse system over $P$ valued on $\cA$}\index{inverse system} is a contravariant functor
$\xymatrix@1{P\ar[r]^-{G}& \cA}$; hereafter, we will denote either by $G_p$
or $G(p)$ the value of $G$ at $p\in P.$

\end{enumerate}

From now on, $\Dir (P,\cA)$ (respectively, $\Inv (P,\cA)$) will denote the category of direct (respectively, inverse)
systems valued on $\cA$; it is known that both are abelian categories \cite[5.94]{Rotman2009}.

Before going on, we would like to exhibit some examples of direct and inverse systems for the
convenience of the reader; in order to do so, we need to review the following:

\begin{df}[Hasse, Voght]
Let $P$ be a finite poset. The \emph{Hasse-Voght diagram} of $P$ is obtained by drawing the elements of $P$ as dots, with $x$ drawn lower than $y$ if $x<y$, and with an edge between $x$ and $y$ whenever $y$ \emph{covers} $x$; that is, if $x<y$ and no $z\in P$ satisfies $x<z<y$.
\end{df}

Now, we are ready for our examples.

\begin{ex}\label{unos grafos para ilustrar}
We go back to \textbf{Example} \ref{el poset de Mayer-Vietoris}; in the below picture, we take $n=2$, and we draw on the left the poset $P$ and on the right the direct system $F$ given by $F(p)=\Gamma_{I_p} (M),$ where $M$ is any $A$--module. Notice that, in all the
following pictures, the non--dotted lines and arrows are edges
of the poset $P,$ and the dotted ones are edges of the poset
$\widehat{P}.$  
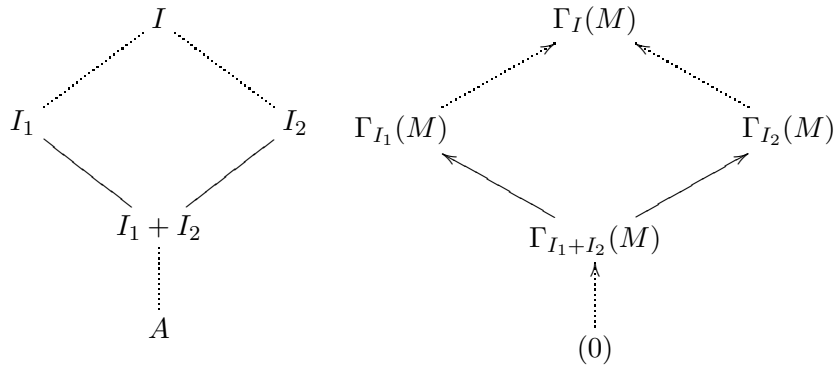
\[
\xymatrix{ & I& \\ I_1\ar@{.}[ur]& & I_2\ar@{.}[ul]\\ & I_1+I_2\ar@{-}[ul]\ar@{-}[ur] & \\ & A\ar@{.}[u]& }\quad\xymatrix{& \Gamma_{I} (M)& \\ \Gamma_{I_1} (M)\ar@{.>}[ur]& & \Gamma_{I_2} (M)\ar@{.>}[ul]\\ & \Gamma_{I_1+I_2} (M)\ar@{->}[ul]\ar@{->}[ur] & \\ & (0)\ar@{.>}[u]& }
\]
Again in \textbf{Example} \ref{el poset de Mayer-Vietoris}, we take $n=2$, but now on the right we draw the inverse system $G$ given by the assignment $G(p)=M/I_p M$.
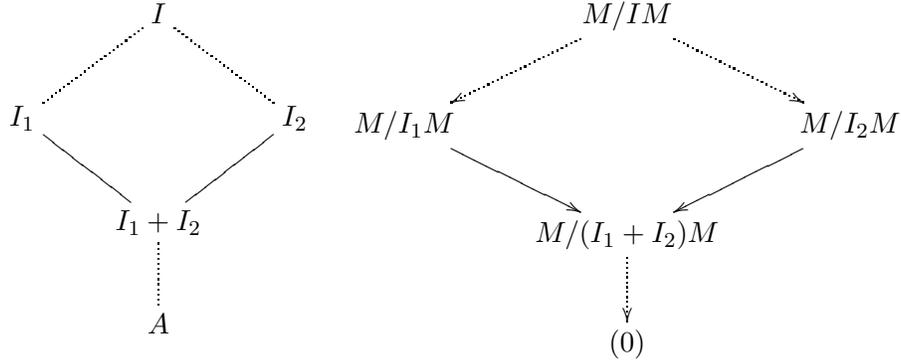
\[
\xymatrix{ & I& \\ I_1\ar@{.}[ur]& & I_2\ar@{.}[ul]\\ & I_1+I_2\ar@{-}[ul]\ar@{-}[ur] & \\ & A\ar@{.}[u]& }\quad\xymatrix{& M/IM\ar@{.>}[dl]\ar@{.>}[dr]& \\ M/I_1 M\ar@{->}[dr]& & M/I_2 M\ar@{->}[dl]\\ & M/(I_1+I_2)M\ar@{.>}[d] & \\ & (0)& }
\]
From now onward, we omit the initial element of the poset $\widehat{P}$ in our Hasse-Voght diagrams. Now, we want to illustrate \textbf{Example} \ref{el poset de Mayer-Vietoris} in case $n=3$ in a couple of specific situations; indeed, take the ring $A=\K [x_1,\ldots ,x_6]$, where $\K$ is any field. Firstly, if $I=(x_1, x_2)\cap(x_3, x_4)\cap(x_5, x_6)$ then our poset can be drawn in the following way:
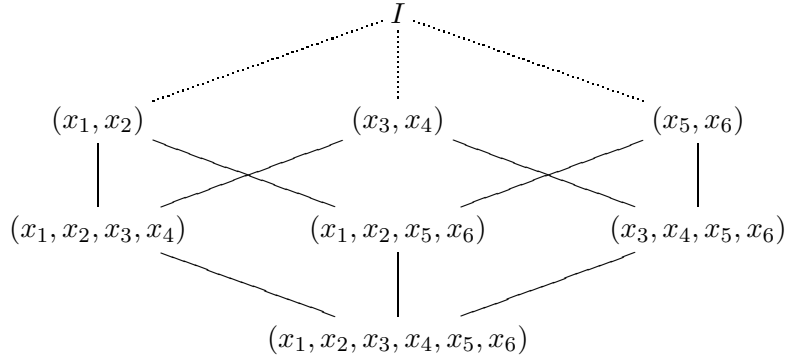
\[
\xymatrix{\hbox{}& I\\ (x_1, x_2)\ar@{.}[ur]& (x_3, x_4)\ar@{.}[u]& (x_5, x_6)\ar@{.}[ul]\\ (x_1, x_2,x_3,x_4)\ar@{-}@{-}[u]\ar@{-}[ur]& (x_1, x_2,x_5,x_6)\ar@{-}[ul]\ar@{-}[ur]& (x_3, x_4,x_5,x_6)\ar@{-}[ul]\ar@{-}[u]\\ & (x_1, x_2,x_3,x_4,x_5,x_6)\ar@{-}[ul]\ar@{-}[u]\ar@{-}[ur]& }
\]
In the case that $I=(x_1, x_2)\cap(x_1, x_3)\cap(x_2, x_3)$ our picture becomes more simple because
\textbf{we identify these sums that describe the same ideal of $A$}:
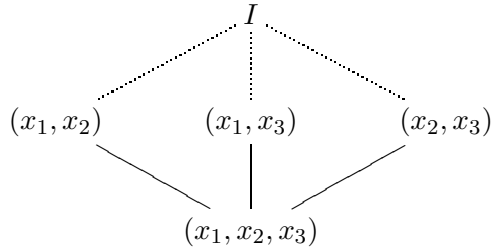
\[
\xymatrix{\hbox{}& I\\ (x_1, x_2)\ar@{.}[ur]& (x_1, x_3)\ar@{.}[u]& (x_2, x_3)\ar@{.}[ul]\\ & (x_1, x_2,x_3)\ar@{-}[ul]\ar@{-}[u]\ar@{-}[ur]& }
\]
\end{ex}

\subsection{Equivalent approaches}

We decided to choose this approach to study derived functors of
limits and colimits. Nevertheless, there are other equivalent
approaches to this subject. Some of them are reviewed in what
follows. The reader is encouraged to follow his/her own preferences.

\subsubsection{The category of sheafs on posets}
Given a poset $P$ and given $p,q\in P$, the closed interval of elements between $p$ and $q$ will be denoted
by $[p, q]:=\{z\in P\mid\quad p\leq z\leq q\}$ and form a sub-poset of $P$. In a similar way, we can also construct
the intervals $(p,q]$, $[p,q)$ and $(p,q)$.

The \emph{Alexandrov topology}\index{Alexandrov topology} on $P$ is the topology where the open sets are the subsets
$U$ of $P$ such that $p\in U$ and $p\leq q$ implies $q\in U$. In fact, this is the unique topology which one can attach
to $P$ verifying this property. Moreover, the subsets of the form $[p,1_{\widehat{P}})$ form a open basis for this topology. On the other hand, the \emph{dual Alexandrov topology} on $P$ is the topology where the open sets are the subsets $U$ of $P$ such that $p\in U$ and $q\leq p$ implies $q\in U$. Once more, this is the unique topology which one can attach to $P$ verifying this property and the subsets of the form $(0_{\widehat{P}}, p]$ form a open basis for this topology. We underline that this can be viewed as the Alexandrov topology on the opposite poset $P^{op}=(P,\preceq)$, where $p\preceq q$ if and only if $q\leq p$.

\begin{ex}\label{el poset de topologos natos}
Given $X$ a set with $X=X_1\cup\ldots\cup X_n$ for certain subsets $X_j\subseteq X$, we define $P$
as the poset given by all the possible different intersections of $X_1,\ldots ,X_n$ ordered by inclusion. Thus, any
element $p\in P$ is identified with a certain intersection of the $X_j$'s, which we denote by $X_p$. The reader will
easily note that this example recovers and extends \textbf{Example} \ref{el poset de Mayer-Vietoris}.
\end{ex}

We shall fix some additional notation before going on.

\begin{nt}
In the sequel, we shall denote by $\Sh (P,\cA)$ (respectively, $\Sh (P^{op},\cA)$) the category of sheaves on
$P$ (respectively, $P^{op}$) valued on $\cA$.
\end{nt}

We conclude this part with the following elementary statement, which will be useful in what follows; we omit
the details. 

\begin{lm}\label{contractibilidad de abiertos en la topologia de Alexandrov}
Let $(P,\leq)$ be a poset regarded as a topological space with the Alexandrov topology. Then, for any $p\in P$,
the basic open set $[p,1_{\widehat{P}})$ is contractible.
\end{lm}

We also want to establish now the analogous of \textbf{Lemma} \ref{contractibilidad de abiertos en la topologia de Alexandrov}
regarding $P$ as a topological space with the dual Alexandrov topology; once again, we skip the details.

\begin{lm}\label{contractibilidad de abiertos en la topologia de Alexandrov: segunda parte}
Let $\left(P,\leq\right)$ be a poset regarded as a topological space with the dual Alexandrov topology. Then,
for any $p\in P$, the basic open set $(0_{\widehat{P}}, p]$ is contractible.
\end{lm}

\subsubsection{The category of $AP$-modules}
The purpose of this part is to review the category of left $AP$-modules, as presented in
\cite[Section 6]{BrunBrunsRomer2007}; in fact, whereas in op.\,cit. the authors established this notion
in the context of inverse systems, here our definition deals with direct systems. However, it is clear that
both notions can be mutually recovered just by taking the opposite order on the poset $P$.

\begin{df}\label{la nocion de AP modulo}
A left $AP$-\emph{module} $M$ is a system $(M_p)_{p\in P}$ of left $A$-modules and, for $p\leq q$, homomorphisms
$\xymatrix@1{M_p\ar[r]^-{M_{pq}}& M_q}$ with the property that, for all $p\leq q\leq z$, $M_{pp}=\mathbbm{1}_{M_p}$
and $M_{pq}\circ M_{qz}=M_{pz}$; moreover, a \emph{homomorphism} $\xymatrix@1{M\ar[r]^-{f}& N}$ of left $AP$-modules
consists of, for $p\in P$, homomorphisms $\xymatrix@1{M_p\ar[r]^-{f_p}& N_p}$ of left $A$-modules such that,
for any $p\leq q$ in $P$, the following square commutes.
\[
\xymatrix{M_p\ar[d]_-{M_{pq}}\ar[r]^-{f_p}& N_p\ar[d]^-{N_{pq}}\\ M_q\ar[r]_-{f_q}& N_q}
\]
We shall denote the group of homomorphisms from $M$ to $N$ by $\Hom_{AP} (M,N)$. Moreover, we shall denote
by $AP-\Mod$ the category of left $AP$-modules.

In a similar way, just reversing the convenient morphisms, we can also construct the category $AP^{op}-\Mod$ of
left $AP^{op}$-modules.
\end{df}

\subsubsection{Modules over the incidence algebra}

Recall that $P$ is a finite poset (see \textbf{Assumption} \ref{hipotesis de finitud sobre el poset}) and that $A$ is a
commutative ring.

\begin{df}[Incidence algebra]\label{definicion del algebra de incidencia}
We define the \emph{incidence algebra} $I(P,A)$ of $P$ over $A$ as follows; $I(P,A)$ is the $A$-algebra with
underlying $A$-module
\[
I(P,A):=\bigoplus_{p\leq q} A\cdot\mathbf{e}_{p\leq q}
\]
endowed with the following multiplication rule:
\[
(\mathbf{e}_{p\leq q})\cdot (\mathbf{e}_{p'\leq q'}):=\begin{cases} \mathbf{e}_{p\leq q'},\text{ if }q=p',\\ 0,
\text{ otherwise.}\end{cases}
\]
Notice that the elements $\mathbf{e}_{p\leq p}$ are idempotent in $I(P,A)$ and that, since $P$
is finite, the element $\sum_{p\in P} \mathbf{e}_{p\leq p}$ is the multiplicative unit in $I(P,A)$.
\end{df}

\subsubsection{The four previous categories are equivalent}

In this part, we establish the fact that all the foregoing approaches are equivalent; more precisely:

\begin{prop}\label{equivalencias de categorias para colimites}
The following categories are equivalent:

\begin{enumerate}[(a)]

\item $\Dir (P,\cA)$, the category of direct systems on $P$ valued in $\cA$.

\item $\Sh (P,\cA)$, the category of sheaves on $P$ valued in $\cA$ regarding $P$ as a topological space with
the Alexandrov topology.

\item $AP-\Mod$, the category of left $AP$-modules.

\item The category of left modules over the incidence algebra $I(P^{op},A)$.

\end{enumerate}

On the other hand, the below four categories are equivalent:

\begin{enumerate}[(i)]

\item $\Inv (P,\cA)$, the category of inverse systems on $P$ valued in $\cA$.

\item $\Sh (P^{op}, \cA)$, the category of sheaves on $P^{op}$ valued in $\cA$ regarding $P$ as a topological
space with the Alexandrov topology.

\item $AP^{op}-\Mod$, the category of left $AP^{op}$-modules.

\item The category of left modules over the incidence algebra $I(P,A)$.

\end{enumerate}

\end{prop}

\begin{proof}[Sketch of proof]
A detailed proof of the equivalence between categories (i)--(iv) can be
found in \cite[6.5 and 6.6]{BrunBrunsRomer2007}; the equivalence between
categories (a)--(d) can be proved in a similar way. What we want to do here is
to review how to construct the equivalence between $\Dir (P,\cA)$
and $\Sh (P,\cA)$, the category of sheaves on $P$ valued in $\cA$ regarding $P$ as a topological space with
the Alexandrov topology, because we will use this equivalence several times
throughout this paper.

On the one hand, given an object $\mathcal{F}\in\Sh (P,\cA)$, one can cook up
a direct system $\Gamma (P,\cF)$ in the following way: for each $p\in P,$ set
$\Gamma (P,\cF)_p:= \cF ([p,1_{\widehat{P}})).$ Moreover, for any $p\leq q$
the homomorphism $\xymatrix@1{\Gamma (P,\cF)_p\ar[r]& \Gamma (P,\cF)_q}$
equals the restriction map induced by the inclusion of open sets
$[q,1_{\widehat{P}})\subseteq [p,1_{\widehat{P}}).$ On the other hand, given
a direct system $M$ we can produce a presheaf $\widetilde{M}$
as follows; for each open set $U\subseteq P$ set $\widetilde{M}(U):=\colim_{p\in U} M_p.$ In
addition, given open sets $V\subseteq U\subseteq P,$ the map
$\xymatrix@1{\widetilde{M}(U)\ar[r]& \widetilde{M}(V)}$ is just the
natural restriction of colimits. Finally, since $[p,1_{\widehat{P}})$ is
contained in any open neighborhood of $P$ (this is by definition of the
Alexandrov topology), the presheaf $\widetilde{M}$ is isomorphic to
the associated sheaf of $A$--modules on $P;$ in this way, the symbol
$\widetilde{(-)}$ defines a well--defined functor from $\Dir (P,\cA)$
to $\Sh (P,\cA).$
\end{proof}


\subsection{Injective, projective, and flasque inverse systems}

The goal of this subsection is to review, not only the well-known fact that the category $\Inv (P,\cA)$ has enough
injective and projective objects \cite[Satz 1 and Satz 2]{Nobeling1962}, but also to describe explicitly certain
classes of injective (respectively, projective) objects which will play a key role later on.

We start with the case of injective inverse systems, as described in the following result; the interested reader may like
to consult either \cite[Satz 1 and Beweis]{Nobeling1962}
or \cite[Lemma A.4.3 and Remark A.4.4]{NeemanTriangulatedCategories} for details. The reader will easily note that, since we are working with
finite posets, in the below result we only need to deal with direct sums instead of
direct products; namely:

\begin{teo}\label{sistemas inversos inyectivos especiales}
Any injective object of $\Inv (P,\cA)$ is a direct summand of a direct sum of injectives of the form $E_{\geq q}$
for some $q\in P$, where $E=E_A (A/\mathfrak{p})$ is an indecomposable injective $A$-module, and
\[
\left(E_{\geq q}\right)_p:=\begin{cases} E,\text{ if }p\in [q,1_{\widehat{P}}),\\ 0,\text{ otherwise.}\end{cases}
\]
Here, given $p\leq r$ the map $\xymatrix@1{\left(E_{\geq q}\right)_r\ar[r]& \left(E_{\geq q}\right)_p}$
is zero if either source or target is zero, otherwise it equals the identity on
$E.$ In addition, any inverse system $G$ admits an injective (not necessarily minimal) resolution
$\xymatrix@1{0\ar[r]& G\ar[r]& E^0\ar[r]& E^1\ar[r]& \ldots,}$ such that, for any $i\in\N,$ $E^i$ can be expressed as direct sum of injectives of the form
$E_{\geq q}.$
\end{teo}

\begin{rk}
The statement of \textbf{Theorem} \ref{sistemas inversos inyectivos especiales} is not fully satisfactory in the sense that
one would wish to have an structure theorem of injectives in $\Inv (P,\cA)$; in other words, we would like to have a description of indecomposable injective objects in $\Inv (P,\cA)$. We do not know whether such description exists.
\end{rk}
We also want to single out the multigraded version of \textbf{Theorem} \ref{sistemas inversos inyectivos especiales}, because
it will be useful later on; notice that \cite[Satz 1 and Beweis]{Nobeling1962} also works
in this multigraded setting because, in particular, there are enough
multigraded (a.k.a. $\hbox{}^*$)injectives objects, and these $\hbox{}^*$injectives
satisfy an analogous Matlis type structure theorem. The interested reader
may like to consult \cite[Chapter 13]{BroSha} and the references given
therein for additional information.

\begin{teo}\label{sistemas inversos graduados inyectivos especiales}
Let $\hbox{}^*\cA$ be the category of $\Z^n$-graded $A$-modules, where $A$ is a commutative Noetherian $\Z^n$-graded
ring (for some $n\geq 1$). Then, any $\hbox{}^*$injective object of $\Inv (P,\hbox{}^*\cA)$ is a direct summand of a
direct sum of $\hbox{}^*$injectives of the form $\hbox{}^*E_{\geq q}$ for some $q\in P$, where
$\hbox{}^*E=\hbox{}^*E_A (A/\mathfrak{p})$ is an indecomposable $\hbox{}^*$injective $A$-module, and
\[
\left(\hbox{}^*E_{\geq q}\right)_p:=\begin{cases} \hbox{}^*E,\text{ if }p\in [q,1_{\widehat{P}}),\\ 0,
\text{ otherwise.}\end{cases}
\]
Here, given $p\leq r$ the map $\xymatrix@1{\left(\hbox{}^*E_{\geq q}\right)_r\ar[r]& \left(\hbox{}^*E_{\geq q}\right)_p}$
is zero if either source or target is zero, otherwise it equals the identity on
$\hbox{}^*E.$ In addition, any inverse system $G$ admits an $\hbox{}^*$injective (not necessarily minimal) resolution
$\xymatrix@1{0\ar[r]& G\ar[r]& \hbox{}^*E^0\ar[r]& \hbox{}^*E^1\ar[r]& \ldots,}$ such that, for any $i\in\N,$ $\hbox{}^*E^i$ can be expressed as direct sum of $\hbox{}^*$injectives of the form
$\hbox{}^*E_{\geq q}.$
\end{teo}

Now, the case of projective objects; we refer to \cite[Satz 2 and Beweis]{Nobeling1962} for details.

\begin{teo}\label{proyectivos que interesan}
Given a projective $A$-module $F$ and given $p\in P$, the inverse system $F_{\leq p}$ defined by
\[
\left(F_{\leq p}\right)_q:=\begin{cases} F,\text{ if }q\in (0_{\widehat{P}},p],\\ 0,\text{otherwise}\end{cases}
\]
is a projective object of $\Inv (P,\cA)$; in particular, $A_{\leq p}$ is so. Here, given $q\leq r$ the homomorphism $\xymatrix@1{\left(F_{\leq p}\right)_r\ar[r]& \left(F_{\leq p}\right)_q}$
is zero if either source or target is zero, otherwise it equals the identity on
$F.$ Furthermore, any inverse system $G$ admits a projective (not necessarily minimal) resolution
\[
\xymatrix{\ldots\ar[r] & F_1\ar[r]& F_0\ar[r]& G\ar[r]& 0}
\]
such that, for any $i\in\N$, $F_i$ can be expressed as a direct sum of projectives of the form $A_{\leq p}$.
\end{teo}

\vskip 2mm

Since the category of inverse systems has enough injectives
\cite[Lemma A.4.3]{NeemanTriangulatedCategories}, one has
that the right derived functors of the limit can be defined in the
usual way through injective resolutions \cite[Corollary A.4.5]{NeemanTriangulatedCategories}; however, in this case, it
turns out that one needs to compute explicitly these derived
functors and, with this purpose in mind, we have to restrict our
attention to the following subclass of objects of $\Inv (P,\cA)$.

\begin{df}\label{sistemas inversos flasque}
We say that an object $G$ of $\Inv (P,\cA)$ is
\emph{flasque} if, for any open set
$U\subseteq P$ (regarding $P$ as a
topological space with the dual Alexandrov topology), the natural
restriction map $\lim_{p\in P} G(p)\xymatrix@1{ \ar[r]& }\lim_{u\in
U} G(u)$ is surjective.
\end{df}

Now, we want to single out the following result because it will play
some role in \textbf{Proposition} \ref{para que quiero a los
flasque}. Its proof follows from the fact that any retract
of a surjection is also a surjection, and it will be omitted.

\begin{lm}\label{los flasque son thick}
Any direct summand of a flasque inverse system is also flasque.
\end{lm}


Flasque inverse systems are useful for our purposes because of the following:

\begin{prop}\label{para que quiero a los flasque}
The following statements hold.

\begin{enumerate}[(i)]

\item Any inverse system can be embedded into a flasque one; whence $\Inv (P,\cA)$ has enough flasque
objects.

\item Any injective inverse system is flasque.

\item Any inverse system admits a flasque resolution. 



\item The right derived functors of the limit can be computed through a flasque resolution.



\end{enumerate}

\end{prop}

\begin{proof}
Let $G$ be an object of $\Inv (P,\cA)$ and fix $p\in P$. We consider the following inverse system:
\[
\Pi^0 (p):=\Pi^0 (G) (p):=\prod_{p_0\leq p} G(p_0).
\]
Moreover, given $p\leq q$ we set $\xymatrix@1{\Pi^0 (q)\ar[r]& \Pi^0 (p)}$ as the natural projection. In this way,
$\Pi^0 (G)$ defines an inverse system which we claim is flasque.

Before showing so, we check that, for any open subset $W$ of $P^{op}$, one has that
\[
\lim_{w\in W} \Pi^0 (w)\cong\prod_{w\in W} G(w).
\]
Indeed, given $w_0\in W$ consider the natural projection
$\xymatrix@1{\prod_{w\in W}G(w)\ar[r]^-{\pi_{w_0}}& \prod_{w\leq w_0} G(w);}$ on the other hand,
we also consider the natural restriction map:
\[
\lim_{w\in W} \Pi^0 (w)\xymatrix{\ar[r]^-{p_{w_0}}& \Pi^0 (w_0).}
\]
Thus, for any $x_0\leq w_0,$ if $\xymatrix@1{\Pi^0 (w_0)\ar[r]^-{p_{w_0 x_0}}& \Pi^ 0 (x_0)}$ denotes
the corresponding structural map in the inverse system $\Pi^0,$ then one has the
following commutative triangle:
\[
\xymatrix{ & \lim_{w\in W} \Pi^0 (w)\ar[dl]_-{p_{w_0}}\ar[dr]^-{p_{x_0}}& \\
\Pi^0 (w_0)\ar[rr]^-{p_{w_0 x_0}}& & \Pi^ 0 (x_0).}
\]
In this way, the universal property of the limit guarantees the existence of a unique homomorphism of $A$-modules
\[
\prod_{w\in W} G(w)\xymatrix{\ar[r]^-{\varphi=\varphi_W}& }\lim_{w\in W} \Pi^0 (w)
\]
such that $p_{w_0}\varphi =\pi_{w_0}$; on the other hand, the universal property of the direct product ensures
the existence of a unique homomorphism of $A$-modules
\[
\lim_{w\in W} \Pi^0 (w)\xymatrix{\ar[r]^-{\psi=\psi_W}& }\prod_{w\in W} G(w)
\]
such that $\pi_{w_0}\psi =p_{w_0}$. We claim that $\psi$ and $\varphi$ are mutually inverses; indeed, we only
have to point out that
\[
\pi_{w_0}\left(\psi\varphi\right)=\left(\pi_{w_0}\psi\right)\varphi=p_{w_0}\varphi =\pi_{w_0}
\]
and
\[
p_{w_0} \left(\varphi\psi\right)=\left(p_{w_0}\varphi\right)\psi=\pi_{w_0}\psi =p_{w_0},
\]
whence $\psi\varphi$ and $\varphi\psi$ satisfy respectively the same universal problem as the identity on
\[
\prod_{w\in W} G(w)\text{ and }\lim_{w\in W} \Pi^0 (w),
\]
whence they are mutually inverses; in particular, one has a canonical isomorphism
\[
\lim_{w\in W} \Pi^0 (w)\cong\prod_{w\in W} G(w).
\]
Now, we want to show that $\Pi^0$ is flasque; let $U\subseteq V$ be open subsets of $P$. Using the previous isomorphisms,
the natural restriction map
\[
\lim_{v\in V} \Pi^0 (v)\xymatrix{\ar[r]^-{r_{VU}}& }\lim_{u\in U} \Pi^0 (u)
\]
can be written as $r_{VU}=\varphi_U\circ\pi_{VU}\circ\psi_V,$ where $\pi_{VU}$ denotes the natural projection
\[
\prod_{v\in V} G(v)\xymatrix{\ar[r]& }\prod_{u\in U} G(u),
\]
and, since $\pi_{VU}$ is surjective (because $U\subseteq V$), $r_{VU}$
is also surjective. Summing up, we have checked that $\Pi^0$ is a flasque inverse system;
finally, the natural map $\xymatrix@1{G\ar[r]& \Pi^0}$ clearly defines a monomorphism of $A$-modules; in this way,
part (i) holds.

Now, we prove part (ii). Let $I$ be an injective inverse system; by part (i), $I$ can be embedded into a flasque
inverse system $\Pi^0 (I)$. Regardless, since $I$ is injective one has that $I$ is a direct factor of $\Pi^0 (I)$,
whence it is also flasque according to Lemma \ref{los flasque son thick}. In this way, part (ii) also holds.

We go on proving part (iii). Given an object $G$ of $\Inv (P,\cA)$, we have
constructed in part (i) a flasque inverse system $\Pi^0 (G)$. Then, setting
$Q^0:=\Coker\left(\xymatrix{G\ar[r]& \Pi^0 (G)}\right)$ we have the following short exact sequence of inverse systems:
\[
\xymatrix{0\ar[r]& G\ar[r]& \Pi^0 (G)\ar[r]& Q^0\ar[r]& 0.}
\]
Now, replacing $G$ by $Q^0$ and setting $\Pi^1 (G):=\Pi^0 (Q^0)$ one obtains another short exact sequence:
\[
\xymatrix{0\ar[r]& Q^0\ar[r]& \Pi^1 (G)\ar[r]& Q^1\ar[r]& 0.}
\]
Iterating this process, we build a (possibly infinite) flasque resolution $\xymatrix@1{0\ar[r]& G\ar[r]& \Pi^* (G);}$
whence part (iii) also holds.

Finally, part (iv) follows directly combining parts (i), (ii), and (iii) jointly
with the fact that flasque inverse systems are
acyclic with respect to the limit functor
\cite[Lemma A.3.9]{NeemanTriangulatedCategories}; the proof is therefore completed.
\end{proof}

We conclude this part with the following:

\begin{rk}
Under the equivalence between inverse systems
and sheaves of $A$-modules over $P^{op}$, the resolution of flasque inverse systems
(see \textbf{Definition \ref{sistemas inversos flasque}}) used in the proof of
\textbf{Proposition} \ref{para que quiero a los flasque} (namely, $\Pi^*$) corresponds to the so-called
\emph{Godement resolution} \cite[Proposition 6.73 and Definition of page 381]{Rotman2009}
for computing sheaf cohomology on the topological space $P^{op}$; on the other hand, it was already
pointed out in \cite[7.2]{BrunBrunsRomer2007} that the equivalence between sheaves and left $AP^{op}$-modules
transforms sheaf cohomology on the topological space $P^{op}$ into the right derived functors of
$\Hom_{AP^{op}} \left(A,-\right)$ (namely, $\Ext_{AP^{op}}^* \left(A,-\right))$. 
\end{rk}

\subsection{Injective, projective, and coflasque direct systems}

It is known that any category of modules over an arbitrary ring has enough injective and projective objects;
on the other hand, we have shown in \textbf{Proposition} \ref{equivalencias de categorias para colimites} that, whenever $P$ is
a finite poset, the category $\Dir (P,\cA)$ is equivalent to the category of left modules over the incidence algebra
$I(P^{op},A)$. Therefore, combining these facts one obtains the following:

\begin{teo}\label{existencia de colimites proyectivos}
If $P$ is a finite poset, then the category $\Dir (P,\cA)$ has enough injective and projective objects.
\end{teo}

\begin{rk}
The reader should notice that \textbf{Theorem} \ref{existencia de colimites proyectivos} guarantees the existence of enough
projective objects, but it does not provide a priori any information about how these projectives look like.
\end{rk}

Now, we face a similar problem to the one considered in the previous subsection; indeed, projective direct systems are not suitable
for our homological purposes. For this reason, we have to introduce the following subclass of objects of $\Dir (P,\cA)$;
to the best of our knowledge, the following notion was introduced
in \cite[D\'efinition 1.I.4]{Laudal1965}

\begin{df}\label{objetos coflasque}
Let $\xymatrix@1{P\ar[r]^-{F}& \cA}$ be a direct system; moreover, we regard $P$ as a topological space with the
Alexandrov topology. It is said that $F$ is \emph{coflasque} if, for any open
subset $U\subseteq P$, the natural insertion map
$\colim_{u\in U} F(u)\xymatrix{\ar[r]& }\colim_{p\in P} F(p)$ is injective.
\end{df}


The following result shows that $\Dir (P,\cA)$ has enough coflasque objects and that this class
of direct systems can be used to calculate the left derived functors of the colimit; we refer
to \cite[pages 53--55, overall Satz 8]{Nobeling1962} for details.

\begin{prop}\label{consideramos los coflasque para algo}
The following statements hold.

\begin{enumerate}[(i)]

\item Any direct system can be expressed as a homomorphic image of a coflasque direct system.

\item Any projective direct system is coflasque.

\item Any direct system admits a coflasque resolution.


\item The left derived functors of the colimit can be computed through coflasque resolutions.


\end{enumerate}
\end{prop}

\section{The Roos complexes}\label{Roos}
Given an inverse system of modules, it was introduced independently by Roos and N\"obeling  (see \cite{Roos1961} and \cite{Nobeling1962})
a cochain complex which has as $i$th cohomology the $i$th right derived functors of the limit. In this subsection,
we shall review this definition as well as its dual notion for direct systems.

\subsection{The homological Roos complex}\index{Roos complex!homological}
We consider a direct system over $P$ valued on $\cA$ given by a covariant functor $\xymatrix@1{P\ar[r]^-{F}& \cA}$.
Then, we construct a chain complex \cite[page 33]{Jensen1972}
\[
\xymatrix{\cdots\ar[r]& \roos_k (F)\ar[r]^-{d_k}& \roos_{k-1}
(F)\ar[r]& \cdots}
\]
in the following way:
\begin{enumerate}[(a)]

\item The spots of the complex are
\[
\roos_k (F):=\bigoplus_{p_0<\cdots <p_k}F_{p_0\ldots p_k},
\]
where $F_{p_0\ldots p_k}:=F(p_0)\in\cA$.

\item The boundary map $\xymatrix@1{\roos_k (F)\ar[r]^-{d_k}& \roos_{k-1} (F)}$ is defined on each direct summand
$F_{p_0\ldots p_k}$ as
\[
j_{p_1\ldots p_k}\circ F(p_0\rightarrow p_1)+\sum_{l=1}^k (-1)^l j_{p_0\ldots\widehat{p_l}\ldots p_k},
\]
where $j_{p_0\ldots p_k}$ denotes the natural inclusion map $\xymatrix{F_{p_0\ldots p_k}\ar@{^{(}->}[r]& \roos_k (F).}$

\end{enumerate}
From now on, we denote by $\roos_* (F)$ this chain complex. We
collect in the following result the main feature of this
construction; we omit its proof because it is completely analogous
to the one of Lemma \ref{coaugmentacion del limite inverso}, that we prove later in this paper.

\begin{lm}\label{augmentacion del limite directo}
There is an augmented chain complex $\roos_* (F)\longrightarrow\colim_{p\in P} F(p)\longrightarrow 0$ in
the category $\cA$; moreover, the homology of this chain complex gives the left derived functors of the colimit;
that is, $H_i (\roos_* (F))=\Lf_i \colim_{p\in P} F(p).$
\end{lm}

Before going on, we would like to write down a small example of $\roos_* (F)$ for the benefit
of the reader.

\begin{ex}\label{roos homologico en miniatura: 1}
Let $A$ be a commutative Noetherian ring, let $J\subseteq A$ be an ideal, and let $\Gamma_J$ be
the corresponding torsion functor. Given an ideal $I\subseteq A$ with primary decomposition
$I=I_1\cap I_2$, consider the poset $P$ given by $I_1$, $I_2$ and $I_1+I_2$, where the order is
given by reverse inclusion; in this way, one can define the direct system $\xymatrix@1{P\ar[r]^-{F}&
\cA}$ given by $p\longmapsto\Gamma_{I_p}(A).$ In this case, the Roos complex $\roos_* (F)$ is as follows:

On the one hand, the spots of this chain complex are
\[
\roos_0 (F)=\Gamma_{I_1+I_2}(A)\oplus\Gamma_{I_1}(A)\oplus\Gamma_{I_2}(A)\text{ and }\roos_1 (F)=\Gamma_{I_1+I_2}(A)
\oplus\Gamma_{I_1+I_2}(A).
\]
On the other hand, the unique non--zero differential of the complex $d_1$ is given by
\begin{align*}
& \Gamma_{I_1+I_2}(A)\oplus\Gamma_{I_1+I_2}(A)\longrightarrow\Gamma_{I_1+I_2}(A)\oplus\Gamma_{I_1}(A)\oplus\Gamma_{I_2}(A)\\
& (a,b)\longmapsto (-a-b,a,b).
\end{align*}
It is straightforward to check that one
has an isomorphism $\Coker (d_1)\cong\colim_{p\in P}\Gamma_{I_p}(A).$
\end{ex}

Although the following fact is so elementary, we want to state it because it will play a key role later on
(see \textbf{Setup} \ref{las hipotesis del caso homologico}).

\begin{lm}\label{tira para arriba, grafo}
Assume, in addition, that $F(1_{\widehat{P}})$ is defined; that is, that $F$ is not only defined on
$P$ but also on $P\cup\{1_{\widehat{P}}\}$. Then, there is a unique functorial map
$\xymatrix@1{\colim_{p\in P} F(p)\ar[r]^-{\psi}& F\left(1_{\widehat{P}}\right)}$ which makes 
the following diagram of chain complexes commutative:
\[
\xymatrix{\roos_* (F)\ar@{=}[d]\ar[r]^-{d_0}& \colim_{p\in P} F(p)\ar[d]^-{\psi}\ar[r]& 0\\
\roos_* (F)\ar[r]^-{\psi\circ d_0}& F\left(1_{\widehat{P}}\right)\ar[r]& 0.}
\]
\end{lm}

\begin{proof}
For any $p\in P$, $p<1_{\widehat{P}}$ and therefore there is an arrow
$\xymatrix@1{F(p)\ar[r]& F\left(1_{\widehat{P}}\right)}$. In this way, the universal property of the colimit
produces $\psi$ with the desired properties.
\end{proof}

\begin{rk}
Preserving the assumptions of \textbf{Lemma} \ref{tira para arriba, grafo}, in general $\psi$ does not define
an isomorphism. Equivalently, in general it is not true that the natural map
\[
\colim_{p\in P} F(p)\xymatrix{ \ar[r]^-{\psi} & }\colim_{p\in\widehat{P}} F(p)=F\left(1_{\widehat{P}}\right)
\]
is bijective. It is noteworthy that the equality $\colim_{p\in\widehat{P}} F(p)=F\left(1_{\widehat{P}}\right)$
is well known \cite[Exercise 5.22 (iii)]{Rotman2009}.
\end{rk}

Before going on, we need to introduce the following:

\begin{df}\label{the rank of a poset}
We define the \emph{rank of $P$} (namely, $\rank (P)$) in the following way:
\[
\rank (P):=\left(\max_{p\in P}\# [p,1_{\widehat{P}})\right)-1.
\]
\end{df}

\begin{rk}\label{rank and dimension remark}
Let $\Delta (P)$ be the order complex associated to $P;$ by the very definition of the order complex
of a poset, it follows that $\rank (P)=\dim (\Delta (P)),$ where by $\dim (\Delta (P))$ we mean its
dimension as simplicial complex.
\end{rk}
Keeping in mind the notion of rank of a poset, we are ready to establish another important property of the
Roos chain complex which we plan to exploit later on (see proof of \textbf{Theorem} \ref{una primera construccion para calentar});
namely:

\begin{lm}\label{anulacion a la Grothendieck: caso homologico}
If $k>\rank (P)$, then $\roos_k (F)=0$ for any direct system $F$; in particular, one has $\Lf_k \colim_{p\in P} F(p)=0$.
\end{lm}

\begin{proof}
Fix $k>\rank (P)$ and any direct system $F$. Since $k>\rank (P)$
there are no ordered chains $p_0<\cdots <p_k$ in $P$, and therefore
\[
\roos_k (F)=\bigoplus_{p_0<\cdots <p_k}F_{p_0\ldots p_k}=0,
\]
as required. Then $\Lf_k \colim_{p\in P} F(p)=0$ follows directly from
\textbf{Lemma} \ref{augmentacion del limite directo}.
\end{proof}
Another quite important (but well-known) property of the Roos chain complex which we use several times in this
paper is the following:

\begin{lm}\label{calculas homologia simplicial}
Let $G$ be an abelian group, and let $\lvert G\rvert$ be the constant direct system on $\widehat{P}$ given by
$G$ with identities on $G$ as structural morphisms. Then, $H_i (\roos_* (\lvert G\rvert))=\widetilde{H}_i (P;G)$.
In other words, $\roos_* (\lvert G\rvert)$ may be regarded as a chain complex for computing the reduced simplicial
homology of the topological space $P$ (regarding $P$ as a topological space with the Alexandrov topology) with
coefficients in $G$.
\end{lm}

\begin{proof}
Notice that $\widetilde{H}_i (P;G)=\widetilde{H}_i (\Delta(P);G),$ because the order complex $\Delta(P)$ provides a triangulation of $P$ as topological space; moreover, $\widetilde{H}_i (\Delta(P);G)=H_i(\widetilde{\mathcal{C}}_{\bullet} (\Delta(P))\otimes_{\Z} G),$ where $\widetilde{\mathcal{C}}_{\bullet} (\Delta(P))$ denotes the chain complex for computing the reduced simplicial homology of $\Delta(P)$ with integer coefficients. In this way, the result follows from the fact that $\widetilde{\mathcal{C}}_{\bullet} (\Delta(P))\otimes_{\Z} G=\roos_* (\lvert G\rvert).$
\end{proof}

We end this subsection with the following remark, that can be proved by just using
the very definition of the Roos chain complex (essentially, that its spots consist of finite direct sums):

\begin{rk}
We denote by $\cCh (\cA)$ the category of chain complexes of objects of $\cA$. In this way, we just constructed a
functor $\xymatrix@1{\Dir (P,\cA)\ar[r]^-{\roos_*}& \cCh (\cA)}$ which is exact and commutes with direct sums.
\end{rk}

\subsection{The cohomological Roos complex}\index{Roos complex!cohomological}
Now, we consider an inverse system over $P$ valued on $\cA$ given by
a contravariant functor $\xymatrix@1{P\ar[r]^-{G}& \cA}$. Thus, we
build a cochain complex of inverse systems
$\xymatrix@1{\cdots\ar[r]& \roos^k (G)\ar[r]^-{d^k}&
\roos^{k+1}(G)\ar[r]& \cdots}$ as follows
(see either \cite[pp.\,31-32]{Jensen1972}
or \cite[pp.\,346-347]{NeemanTriangulatedCategories}).
\begin{enumerate}[(a)]

\item The pieces of the complex are
\[
\roos^k (G):=\prod_{p_0<\cdots <p_k}G_{p_0\ldots p_k},
\]
where $G_{p_0\ldots p_k}:=G(p_0)$.

\item The coboundary map $\xymatrix@1{\roos^k (G)\ar[r]^-{d^k}& \roos^{k+1}(G)}$ is defined on each factor
$G_{p_0\ldots p_k}$ as
\[
G(p_0\rightarrow p_1)\circ\pi_{p_0\ldots p_k}+\sum_{l=1}^{k+1} (-1)^l \pi_{p_0\ldots\widehat{p_l}\ldots p_{k+1}},
\]
where $\pi_{p_0\ldots p_k}$ denotes the natural projection $\xymatrix{\roos^k (G)\ar@{->>}[r]& G_{p_0\ldots p_k}}$.

\end{enumerate}
Hereafter, we shall denote by $\roos^* (G)$ this cochain complex. As in the homological case, the main
feature of this cohomological construction is reviewed in the following result; in this case, we provide a
proof for the reader's convenience (cf.\,\cite[Proof of Lemma A.3.2]{NeemanTriangulatedCategories}).

\begin{lm}\label{coaugmentacion del limite inverso}
There is a coaugmented cochain complex $0\longrightarrow\lim_{p\in P} G(p)\longrightarrow \roos^* (G)$
in the category $\cA$; moreover, the cohomology of this cochain complex yields the right derived functors
of the limit; that is, $H^i (\roos^* (G))=\R^i \lim_{p\in P} G(p).$
\end{lm}

\begin{proof}
What we show is that the right derived functors of the
limit $(U^j)_{j\geq 0}:=\left(\R^j \lim_{p\in P}\right)_{j\geq 0}$ and the cohomology of the Roos cochain complex
$(V^j)_{j\geq 0}:=\left(H^j (\roos^* (-))\right)_{j\geq 0}$ are isomorphic universal $\delta$-functors.

Indeed, on the one hand it is clear that both $(U^j)_{j\geq 0}$ and $(V^j)_{j\geq 0}$ are universal $\delta$-functors with
\[
U^0=\R^0 \lim_{p\in P}=\lim_{p\in P}=H^0 (\roos^* (-))=V^0.
\]
On the other hand, given any $j\geq 1$ and any injective inverse system $I$ it is also clear that
\[
U^j (I)=\R^j \lim_{p\in P} I_p=0=H^j (\roos^* (I))=V^j (I);
\]
notice that the equality $0=H^j (\roos^* (I))$ is a direct application of \cite[Satz 6]{Nobeling1962}. Combining the previous two facts one obtains a canonical isomorphism
\[
\left(\R^j \lim_{p\in P}\right)_{j\geq 0}\cong\left(H^j (\roos^* (-))\right)_{j\geq 0}
\]
of universal $\delta$-functors, just what we finally wanted to check.
\end{proof}

Before going on, we would like to write down a small example of $\roos^* (G)$ for the benefit
of the reader.

\begin{ex}\label{roos cohomologico en miniatura: 1}
Let $A$ be a commutative Noetherian ring; given an ideal $I\subseteq A$ with primary decomposition
$I=I_1\cap I_2$, consider the poset $P$ given by $I_1$, $I_2$ and $I_1+I_2$, where the order is
given by reverse inclusion; in this way, one can define the inverse system $\xymatrix@1{P\ar[r]^-{G}&
\cA}$ given by $p\longmapsto A/I_p.$ In this case, the Roos complex $\roos^* (G)$ is as follows:

On the one hand, the spots of this cochain complex are
\[
\roos^0 (G)=A/(I_1+I_2)\times A/I_1\times A/I_2\text{ and }\roos^1 (G)=A/(I_1+I_2)\times A/(I_1+I_2).
\]
On the other hand, the unique non--zero differential of the complex $d^0$ is given by
\begin{align*}
& A/(I_1+I_2)\times A/I_1\times A/I_2\longrightarrow A/(I_1+I_2)\times A/(I_1+I_2)\\
& (a,b,c)\longmapsto (-a+b,-a+c).
\end{align*}
In this case, it is straightforward to check that the constant map
\begin{align*}
& A/I\longrightarrow A/(I_1+I_2)\times A/I_1\times A/I_2\\
& a\longmapsto (a,a,a).
\end{align*}
establishes an isomorphism $\lim_{p\in P}A/I_p=\ker (d^0)\cong A/I.$
\end{ex}

We state the cohomological analogous of \textbf{Lemma} \ref{tira para arriba, grafo}. We skip the details.

\begin{lm}\label{tira para abajo, grafo}
Suppose, in addition, that $G(1_{\widehat{P}})$ is defined; that is, that $G$ is not only defined on $P$ but also in
$P\cup\{1_{\widehat{P}}\}$. Then, there is a unique functorial map
$\xymatrix@1{G(1_{\widehat{P}})\ar[r]^-{\alpha}& \lim_{p\in P} G(p)}$ which makes commutative the below diagram
of cochain complexes
\[
\xymatrix{0\ar[r]& G(1_{\widehat{P}})\ar[d]_-{\alpha}\ar[r]^-{d^0\circ\alpha}& \roos^* (G)\ar@{=}[d]\\ 0\ar[r]& \lim_{p\in P} G(p)\ar[r]^-{d^0}& \roos^* (G).}
\]
\end{lm}

\begin{rk}\label{un apunte interesante para limites inversos}
Preserving the assumptions of \textbf{Lemma} \ref{tira para abajo, grafo},  it is in general not true that $\alpha$ defines an
isomorphism. In other words, it is in general not true that the natural restriction map
\[
G(1_{\widehat{P}})=\lim_{p\in\widehat{P}} G(p)\xymatrix{ \ar[r]^-{\alpha}& }\lim_{p\in P} G(p)
\]
is an isomorphism. It is known (see \cite[3.3]{BrunBrunsRomer2007}
and \cite[Theorem 5.6]{BrunRomer2008}) to be an isomorphism in some particular situations.
\end{rk}

Next, we state the cohomological analogous of \textbf{Lemma} \ref{anulacion a la Grothendieck: caso homologico}; once again,
the details are left to the interested reader.

\begin{lm}\label{anulacion a la Grothendieck: caso cohomologico}
If $k>\rank (P)$, then $\roos^k (G)=0$ for any inverse system $G$; in particular, $\R^k \lim_{p\in P} G(p)=0$.
\end{lm}

Now, we write down the cohomological analogous of \textbf{Lemma} \ref{calculas homologia simplicial}; once again, we omit
the details.

\begin{lm}\label{calculas cohomologia simplicial}
Let $G$ be an abelian group, and let $\lvert G\rvert$ be the constant inverse system on $\widehat{P}$ given by $G$
with identities on $G$ as structural morphisms. Then, $H^i (\roos^* (\lvert G\rvert))=\widetilde{H}^i (P;G)$.
In other words, $\roos^* (\lvert G\rvert)$ may be regarded as a cochain complex for computing the reduced simplicial
cohomology of the topological space $P$ (regarding $P$ as a topological space with the dual Alexandrov topology) with
coefficients in $G$.
\end{lm}

As in the homological case, we end with the following:

\begin{rk}
If we denote by $\ccoCh (\cA)$ the category of cochain complexes of objects of $\cA$, then we have produced a
functor $\xymatrix@1{\Inv (P,\cA)\ar[r]^-{\roos^*}& \ccoCh (\cA)}$ which is exact and commutes with direct products.
\end{rk}

\section{Homological spectral sequences associated to modules}\label{construccion de libros}

Hereafter, $A$ will denote a commutative Noetherian ring and $\cA$
will stand for the category of $A$-modules. To any ideal $I\subseteq
A$ one may associate the poset $P$ given in \textbf{Example} \ref{el poset de
Mayer-Vietoris}. Namely, let $I=I_1\cap\ldots\cap I_n$ be a
decomposition of $I$ into $n$ ideals, then $P$ is the poset given by
all the possible sums of the ideals $I_1,\ldots ,I_n$ ordered by
reverse inclusion but we underline that we have to identify these
sums when they describe the same ideal.


The aim of this section is to construct several spectral sequences
even though we will be mostly interested on those involving local
cohomology modules; as we have already pointed out in the
Introduction, the spectral sequences we construct in this section
involve in their second page the left derived functors of the
colimit, for this reason we refer to them as \emph{homological}. The
method will be as follows: given any object $M$ of $\cA$, we shall
construct a direct system over $P$ of objects of $\cA$. Then, using
the homological Roos complex we shall build a double complex with a
finite number of non-zero columns (or rows) that will rise to a
spectral sequence that, with the help of a technical lemma, will
converge to a certain object of $\cA$.

Before moving to our promised construction, we want to review now
the following:
\begin{df}\label{the W sets}
Given $J,K\subseteq A$ ideals, set $\mathbf{W}(K,J):=\left\{\mathfrak{q}\in\spec (A)\mid\ K^n\subseteq\mathfrak{q}+J\text{ for some integer }n\geq 1\right\}.$
Remember (see \cite{TakahashiYoshinoYoshizawa2009}) that, when
$J=(0),$ $\mathbf{W} (K,J)=\mathbf{V} (K).$
\end{df}
We also want to single out in the next result some elementary properties which the subsets of the form
$\mathbf{W} (K, J)$ verify because they will play a crucial role later on
(see \textbf{Lemma} \ref{lema clave en el caso homologico}); the below lemma is just a piece
of \cite[1.6]{TakahashiYoshinoYoshizawa2009}, and since its proof is straightforward, we omit it.

\begin{lm}\label{gracias, Santi}
Let $I,I'$ be ideals of $A,$ let $\mathfrak{p}\in\spec (A)$, and let $J$ be an ideal of $A$. Then, the following statements hold.

\begin{enumerate}[(i)]

\item If $\mathfrak{p}\in\mathbf{W}(I,J)\cap\mathbf{W}(I',J)$, then $\mathfrak{p}\in\mathbf{W}(I+I',J)$.

\item If $I'\subseteq I$ and $\mathfrak{p}\in\mathbf{W}(I,J)$, then $\mathfrak{p}\in\mathbf{W}(I',J)$.

\end{enumerate}
\end{lm}


\subsection{Construction of homological spectral sequences}
The setup we will need to build the announced spectral sequences
is the following direct system of functors that we plan to describe
with detail in the below:


\begin{cons}\label{las hipotesis del caso homologico}
Let $\xymatrix@1{\cA\ar[r]^-{T_{[*]}}& \Dir (\widehat{P},\cA)}$ be an additive, covariant functor which verifies
the following requirements:
\begin{enumerate}[(i)]

\item For any $p\in\widehat{P}$, $T_p$ is a covariant, left exact functor which commutes with direct sums.

\item If $p\leq q$ then there exists a natural transformation of functors $\xymatrix@1{T_p \ar[r]^-{\psi_{pq}}& T_q;}$ notice that this natural transformation extends in a unique way to a natural transformation $\xymatrix@1{\R^i T_p \ar[r]^-{\psi_{pq}^i}& \R^i T_q}$ for any $i\geq 0$ such that $\psi_{pq}^0=\psi_{pq}$ (see \cite[1.3.4 and 1.3.5]{BroSha} for details).

\item There exists an ideal $J$ of $A$ such that the following holds: for any $\mathfrak{p}\in\spec (A)$ and for any maximal ideal $\mathfrak{m}$ of $A$, there exist $A$-modules
$X, Y$ such that, for any $p\in\widehat{P},$
\[
T_p (E(A/\mathfrak{p}))_{\mathfrak{m}}=\begin{cases} X,\text{ if }\mathfrak{p}\in\mathbf{W}(I_p,J)
\text{ and }\mathfrak{p}\subseteq\mathfrak{m},\\
Y,\text{ if }\mathfrak{p}\notin\mathbf{W} (I_p,J)\text{ and }\mathfrak{p}\subseteq\mathfrak{m},\\ 0,\text{ otherwise. }\end{cases}
\]
Here, $X$ and $Y$ may depend on $J,$ $\mathfrak{p}$ and $\mathfrak{m}$, but not on $p.$ Moreover, given
$p\leq q,$
\[
\xymatrix{T_p (E(A/\mathfrak{p}))_{\mathfrak{m}}\ar[r]& T_q (E(A/\mathfrak{p}))_{\mathfrak{m}}}
=\begin{cases}
\mathbbm{1}_X,\text{ if }T_p (E(A/\mathfrak{p}))_{\mathfrak{m}}
=X=T_q (E(A/\mathfrak{p}))_{\mathfrak{m}},\\
\mathbbm{1}_Y,\text{ if }T_p (E(A/\mathfrak{p}))_{\mathfrak{m}}
=Y=T_q (E(A/\mathfrak{p}))_{\mathfrak{m}},\\
0,\text{ otherwise. }
\end{cases}
\]

\end{enumerate}




Hereafter in this subsection, under the assumptions of
\textbf{Setup \ref{las hipotesis del caso homologico}} we set $T:=T_{1_{\widehat{P}}}$; moreover, we also define the
\emph{cohomological dimension} of $T$, denoted $\cdim (T)$, in the following way:
\[
\cdim (T):=\max\{i\in\N\ \mid\ \R^i T (X)\neq 0\text{ for some }A-\text{module }X\}.
\]
Moreover, given an $A$-module $M$, $\R^i T_{[*]} (M)$ will denote the direct system $(\R^i T_p (M))_{p\in P}$ with
structural maps $\xymatrix@1{\R^i T_p (M)\ar[r]& \R^i T_q (M)}$ for $p\leq q.$
\end{cons}

%

\subsubsection{Main result}
In order to compute the abutment of our homological spectral
sequences we will need the following Lemma, that may be regarded as a
generalization of \cite[Lemma of page
38]{AlvarezGarciaZarzuela2003}.

\begin{lm}\label{lema clave en el caso homologico}
If $E$ is any injective $A$-module, then the augmented homological
Roos chain complex
$\xymatrix@1{\roos_*\left(T_{[*]}\left(E\right)\right)\ar[r]&
T\left(E\right)\ar[r]& 0}$ is exact.
\end{lm}

\begin{proof}
First of all, as
$\roos_*$ and $T_{[*]}$ commutes with direct sums we deduce from the
Matlis-Gabriel Theorem \cite[10.1.9]{BroSha} that we may
assume, without loss of generality, that there is
$\mathfrak{p}\in\spec (A)$ such that $E=E_A (A/\mathfrak{p})$, a
choice of injective hull of $A/\mathfrak{p}$ over $A$. Moreover, as
being exact is a local property it is enough to check that, for any
maximal ideal $\mathfrak{m}$ of $A$, the chain complex
\[
\xymatrix{\roos_*\left(T_{[*]}\left(E\right)\right)_{\mathfrak{m}}\ar[r]&
T\left(E\right)_{\mathfrak{m}}\ar[r]& 0}
\]
is exact. Indeed, we can express $P$ as the disjoint union of $Q$
and $Q'$, where
\begin{align*}
& Q:=\{q\in P\mid\quad I_q\subseteq\mathfrak{m}\},\\
& Q':=\{q\in P\mid\quad I_q\not\subseteq\mathfrak{m}\}.
\end{align*}
Notice that $Q$ is clearly a subposet of $P$ (this just means that
$Q$ is a subset of $P$ that is also a poset with the same
partial order of $P$).

Now, assume that $Y=0;$ we have
to distinguish two cases. Firstly, if
$\mathfrak{p}\not\subseteq\mathfrak{m}$, then the previous chain
complex is identically zero, whence we are done. Otherwise, suppose
that $\mathfrak{p}\subseteq\mathfrak{m}$; in this case, we split $Q$
as the disjoint union $Q=Q_1\cup Q_2$, where
\begin{align*}
& Q_1 :=\{p\in Q\mid\ \mathfrak{p}\in\mathbf{W} (I_p,J)\},\\
& Q_2 :=\{p\in Q\mid\ \mathfrak{p}\notin\mathbf{W} (I_p,J)\}.
\end{align*}
Now, we have to distinguish two cases. Indeed, if
$\mathfrak{p}\notin\mathbf{W} (I_p,J)$ then the previous chain
complex is identically zero and therefore we are done. Otherwise,
suppose that $\mathfrak{p}\in\mathbf{W} (I_p,J)$ for at least one
$p$; in this case, this assumption combined with \textbf{Lemma} \ref{gracias,
Santi} ensures that $Q_1$ is a non-empty subposet of $P$ of the form
$[r,1_{\widehat{P}})$, where $r\in P$ such that
\[
I_r=\sum_{q\in Q_1}I_q :=J.
\]
Indeed, since the ideal $J$ is clearly the greatest ideal among the
ideals of $Q_1$ (we want to stress that here is where we are using
the finiteness of $Q_1$ combined with \textbf{Lemma} \ref{gracias, Santi}),
it turns out that there is an element $r\in P$ such that $I_r=J$ and
therefore $Q_1=[r,1_{\widehat{P}})$. Summing up, our chain complex
\[
\xymatrix{\roos_*\left(T_{[*]}\left(E\right)\right)_{\mathfrak{m}}\ar[r]&
T\left(E\right)_{\mathfrak{m}}\ar[r]& 0}
\]
agrees with the one obtained considering the Roos chain complex on
$Q_1$ instead of $P$; finally, this augmented chain complex where we
only consider the subposet $Q_1$ in the construction of the Roos
chain complex equals the augmented one for computing the simplicial
homology of the topological space $[r,1_{\widehat{P}})$ with
coefficients in $X$ (see \textbf{Lemma} \ref{calculas homologia simplicial}).
But we have checked in \textbf{Lemma} \ref{contractibilidad de abiertos en la
topologia de Alexandrov} that this topological space is
contractible; this concludes the proof provided that $Y=0.$

Now, assume that $X=0$ and that $\mathfrak{p}\subseteq\mathfrak{m}$; now,
consider the short exact sequence of direct systems
\begin{equation}\label{ecuacion de santi}
\xymatrix{0\ar[r]& \lvert Y\rvert_{\mid P_1}\ar[r]& \lvert
Y\rvert\ar[r]& \lvert Y\rvert_{\mid P-P_1}\ar[r]& 0.}
\end{equation}
Here, $P_1:=\{p\in P\mid\ \mathfrak{p}\in\mathbf{W}(I_p,J),\
I_p\subseteq\mathfrak{m}\}$ and $\lvert Y\rvert_{\mid P_1}$
(respectively, $\lvert Y\rvert_{\mid P-P_1}$) denotes the direct
system with constant value $Y$ in all the points of $P_1$
(respectively, $P-P_1$) and zero elsewhere; moreover, its non-zero
structural homomorphisms are all identities on $Y$. As we have seen
in the first part of this proof, we can write
$P_1=[r,1_{\widehat{P}})$ for some $r\in P$. On the other hand,
\eqref{ecuacion de santi} induces the following short exact sequence
of chain complexes:

\begin{equation}\label{ecuacion de santi dos}
\xymatrix{0\ar[r]& \roos_*\left(\lvert Y\rvert_{\mid
P_1}\right)\ar[r]& \roos_*\left(\lvert Y\rvert\right)\ar[r]&
\roos_*\left(\lvert Y\rvert_{\mid P-P_1}\right)\ar[r]& 0.}
\end{equation}
Now, since $X=0$ it follows that
$\roos_*\left(\lvert Y\rvert_{\mid
P-P_1}\right)=\roos_*\left(T_{[*]}\left(E\right)\right)_{\mathfrak{m}}$,
whence we can rewrite \eqref{ecuacion de santi dos} in the following
manner:
\begin{equation}\label{ecuacion de santi tres}
\xymatrix{0\ar[r]& \roos_*\left(\lvert Y\rvert_{\mid
P_1}\right)\ar[r]& \roos_*\left(\lvert Y\rvert\right)\ar[r]&
\roos_*\left(T_{[*]}\left(E\right)\right)_{\mathfrak{m}}\ar[r]& 0.}
\end{equation}
However, $\roos_*\left(\lvert Y\rvert_{\mid P_1}\right)$
(respectively, $\roos_*\left(\lvert Y\rvert\right)$) turns out to be
the chain complex for computing the reduced simplicial homology of
the topological space $P_1=[r,1_{\widehat{P}})$ (respectively, $P$)
(see \textbf{Lemma} \ref{calculas homologia simplicial}); since both $P_1$
and $P$ are contractible (see \textbf{Lemma} \ref{contractibilidad de
abiertos en la topologia de Alexandrov}) one has that both
$\roos_*\left(\lvert Y\rvert_{\mid P_1}\right)$ and
$\roos_*\left(\lvert Y\rvert\right)$ are exact.

Summing up, all the foregoing implies that \eqref{ecuacion de santi
tres} is a short exact sequence of chain complexes with two of them being
exact; whence the remaining one (namely,
$\roos_*\left(T_{[*]}\left(E\right)\right)_{\mathfrak{m}}$) is also exact,
just what we wanted to prove.

Finally, assume that both $X$ and $Y$ are nonzero, and that
$\mathfrak{p}\subseteq\mathfrak{m};$ as before, $P_1:=\{p\in P\mid\ \mathfrak{p}\in\mathbf{W}(I_p,J),\
I_p\subseteq\mathfrak{m}\}$ can be written as $P_1=[r,1_{\widehat{P}})$ for some $r\in P.$ Thus, we
have the following short exact sequence of direct systems:
\[
\xymatrix{0\ar[r]& \lvert X\rvert_{\mid
P_1}\ar[r]& T_{[*]}\left(E\right)_{\mathfrak{m}}\ar[r]&
\lvert Y\rvert_{\mid P-P_1}\ar[r]& 0.}
\]
Since $\roos_*$ is exact, this short exact sequence of direct systems
induce the following short exact sequence of chain complexes:
\[
\xymatrix{0\ar[r]& \roos_*\left(\lvert X\rvert_{\mid
P_1}\right)\ar[r]& \roos_*\left(T_{[*]}\left(E\right)_{\mathfrak{m}}\right)\ar[r]&
\roos_*\left(\lvert Y\rvert_{\mid P-P_1}\right)\ar[r]& 0.}
\]
But we have seen during this proof that both $\roos_*\left(\lvert X\rvert_{\mid
P_1}\right)$ and $\roos_*\left(\lvert Y\rvert_{\mid P-P_1}\right)$
are exact, hence $\roos_*\left(\lvert Y\rvert_{\mid P-P_1}\right)$
is also exact; the proof is therefore completed.
\end{proof}

\begin{rk}
Notice that, given any $A$--module $M,$ it is not true, in general, that
$T(M)$ is isomorphic to $\colim_{p\in P} T_p (M);$ however, \textbf{Lemma
\ref{lema clave en el caso homologico}} shows, in particular, that
if $E$ is any injective $A$--module, then $T(E)\cong\colim_{p\in P} T_p (E).$
\end{rk}

Our next statement is the main result of this section.

\begin{teo}\label{una primera construccion para calentar}
Let $M$ be any $A$-module. Then, we have the following spectral
sequence
\[
E_2^{-i,j}=\Lf_i \colim_{p\in P} \R^j T_{[*]}
(M)\xymatrix{\hbox{}\ar@{=>}[r]_-{i}& \R^{j-i} T (M)}
\]
in the category of $A$-modules.
\end{teo}

\begin{proof}
Let $\xymatrix@1{0\ar[r]& M\ar[r]& E^*}$ be an injective resolution
of $M$ in the category $\cA$. Applying to this resolution the
functor $T_{[*]}$ one gets the following cochain complex of direct
systems:
\[
\xymatrix{0\ar[r]& T_{[*]} (M)\ar[r]& T_{[*]} (E^0)\ar[r]& T_{[*]}
(E^1)\ar[r]& \ldots}
\]
Now, we use the Roos chain complex in order to produce the bicomplex
$\roos^{-i} \left(T_{[*]}\left(E^j\right)\right)=\roos_i
\left(T_{[*]}\left(E^j\right)\right)$; the reader will easily note
that we put a minus in the index $i$ because we want to work with a
bicomplex located in the second quadrant. We hope the following
picture illustrates this sentence:
\[
\xymatrix{& \vdots& \vdots& \\ \ldots\ar[r]&
\roos^{-1}\left(T_{[*]}\left(E^1\right)\right)\ar[u]\ar[r]&
\roos^{0}\left(T_{[*]}\left(E^1\right)\right)\ar[u]\ar[r]& 
0\\ \ldots\ar[r]&
\roos^{-1}\left(T_{[*]}\left(E^0\right)\right)\ar[u]\ar[r]&
\roos^{0}\left(T_{[*]}\left(E^0\right)\right)\ar[u]\ar[r]& 
0\\ & 0\ar[u]& 0\ar[u]& }
\]
Moreover, we have to stress that the horizontal differentials are the
ones of the Roos chain complex and the vertical ones are the
induced by the injective resolution of $M$; so, the bicomplex
$\roos^{-i} \left(T_{[*]}\left(E^j\right)\right)$ produces two spectral
sequences; namely, the ones provided respectively by the first and
the second filtration of the previous bicomplex. In this way, the
first thing one should ensure is that both spectral sequences
converge and calculate their common abutment.

First of all, we want to check that our spectral sequences converge;
indeed, \textbf{Lemma} \ref{anulacion a la Grothendieck: caso homologico}
ensures that $\roos^{-i} (T_{[*]} (E^j))=0$ for all $i\gg 0$.
Therefore, this implies that the bicomplex $\roos^{-i}
\left(T_{[*]}\left(E^j\right)\right)$ has just a finite number of
columns, which implies the convergence of both spectral sequences.

Second, notice that the $E_2$-page of one of these spectral
sequences is obtained by firstly computing the cohomology of the
rows and then computing the cohomology of the columns; regardless,
\textbf{Lemma} \ref{lema clave en el caso homologico} guarantees that all the
rows of the bicomplex $\roos^{-i}
\left(T_{[*]}\left(E^j\right)\right)$ are exact up to the $0$th spot.
Therefore, this fact
implies that this spectral sequence collapses, whence its abutment
turns out to be $\R^* T (M)$.

Finally, the other spectral sequence that we can produce is the one
obtained by firstly taking cohomology on the columns and then
calculating the cohomology of the rows; in this case, one obtains
as $E_1$-page $E_1^{-i,j}=\roos^{-i} \left(\R^j T_{[*]}
\left(M\right)\right).$ In addition, since the boundary map of the
$E_1$-page is the one of the Roos chain complex, and this chain
complex computes the $i$th left derived functor of the colimit, its
$E_2$-page turns out to be $E_2^{-i,j}=\Lf_i \colim_{p\in P} \R^j
T_{[*]}(M).$ Summing up, combining all the foregoing facts one obtains
the spectral sequence
\[
E_2^{-i,j}=\Lf_i \colim_{p\in P} \R^j T_{[*]}
(M)\xymatrix{\hbox{}\ar@{=>}[r]_-{i}& \R^{j-i} T (M)}
\]
in the category of $A$-modules; the proof is therefore completed.
\end{proof}

\subsection{Examples} \label{examples_typeI}
We provide several examples of local cohomology spectral sequences.
In particular, the corresponding  functors verify the established
conditions in \textbf{Setup} \ref{las hipotesis del caso homologico}.
In all of these examples, $N$  will denote a finitely generated
$A$-module.


\vskip 2mm

$\bullet$ {\bf Generalized local cohomology:}\label{cohomologia
local generalizada: take 1} The direct system of functors
$T_{[*]}=\Gamma_{[*]}(N,-)$ of \emph{generalized torsion functors}
given by $T_p (M):=\Gamma_{I_p} (N,M)$ verifies all requirements
established in \textbf{Setup} \ref{las hipotesis del caso homologico};
moreover, in this case $T=\Gamma_I (N,-)$. Indeed, first of all we
remember that, for any $p\in\widehat{P}$ and for any $A$-module $M$,
\[
H_{I_p}^j (N,M):=\colim_{k\in\N} \Ext_A^j \left(N/I_p^k N,M\right).
\]
Now, fix $k\in\N$. Using the adjoint associativity between $\Hom$ and tensor product \cite[2.76]{Rotman2009}
one has that $\Hom_A \left(N/I_p^k N,E(A/\mathfrak{p})\right)\cong\Hom_A \left(N,\Hom_A \left(A/I_p^k,E(A/\mathfrak{p})\right)\right).$ On the other hand, since $N$ is finitely generated it follows that
\[
\colim_{k\in\N}\Hom_A \left(N,\Hom_A \left(A/I_p^k,E(A/\mathfrak{p})\right)\right)\cong\Hom_A \left(N,\colim_{k\in\N}\Hom_A \left(A/I_p^k,E(A/\mathfrak{p})\right)\right).
\]
Summing up, one has that $\colim_{k\in\N}\Hom_A \left(N/I_p^k N,E(A/\mathfrak{p})\right)\cong\Hom_A \left(N,\Gamma_{I_p}\left(E(A/\mathfrak{p})\right)\right).$ Moreover, it is well known \cite[10.1.11]{BroSha} that
\[
\Gamma_{I_p} (E(A/\mathfrak{p}))=\begin{cases} E(A/\mathfrak{p}),\text{ if }\mathfrak{p}\in\mathbf{V} (I_p),\\ 0,\text{ if }\mathfrak{p}\notin\mathbf{V} (I_p).\end{cases}
\]
In this way, combining these two facts it follows that
\[
\Gamma_{I_p} (N,E(A/\mathfrak{p}))\cong\Hom_A (N,\Gamma_{I_p} (E(A/\mathfrak{p})))=\begin{cases} \Hom_A (N,E(A/\mathfrak{p})),\text{ if }\mathfrak{p}\in\mathbf{V} (I_p),\\ 0,\text{ if }\mathfrak{p}\notin\mathbf{V} (I_p).\end{cases}
\]
Finally, given any $\mathfrak{m}\in\Max (A)$ one has, as a direct application of \cite[4.1.7]{BroSha} (indeed, $N$ is finitely generated and the localization map $\xymatrix@1{A\ar[r]& A_{\mathfrak{m}}}$ is flat), that
\[
\Gamma_{I_p} (N,E(A/\mathfrak{p}))_{\mathfrak{m}}=\begin{cases} \Hom_{A_{\mathfrak{m}}} \left(N_{\mathfrak{m}},E(A/\mathfrak{p})_{\mathfrak{m}}\right),\text{ if }\mathfrak{p}\in\mathbf{V} (I_p)\text{ and }\mathfrak{p}\subseteq\mathfrak{m},\\ 0,\text{ otherwise.}\end{cases}
\]

\vskip 2mm

Therefore we get the following spectral sequence:
\[
E_2^{-i,j}=\Lf_i \colim_{p\in P} H_{I_p}^j
(N,M)\xymatrix{\hbox{}\ar@{=>}[r]_-{i}& }H_I^{j-i} (N,M).
\]
This spectral sequence is a generalization of \cite[page 39]{AlvarezGarciaZarzuela2003} and \cite[2.1]{Lyu2006}.
In case the ideal has two components $I=I_1\cap I_2$, we recover
partially the Mayer-Vietoris long exact sequence obtained in
\cite[2.14]{Yassemi1994}:
\[
\xymatrix{\ldots\ar[r]& H_{I_1}^j (N,M)\oplus H_{I_2}^j (N,M)\ar[r]&
H_I^j (N,M)\ar[r]& H_{I_1+I_2}^{j+1} (N,M)\ar[r]& \ldots}
\]
For the sake of simplicity, we single out the version of it that we
shall consider in this paper, especially for the case when $N=A$ is
the ring itself.
\[
E_2^{-i,j}=\Lf_i \colim_{p\in P} H_{I_p}^j
(M)\xymatrix{\hbox{}\ar@{=>}[r]_-{i}& }H_I^{j-i}(M).
\]

\vskip 2mm

It is also worth mentioning that the Mayer-Vietoris spectral
sequence considered in \cite[2.1]{Lyu2006} is slightly different at
the $E_1$-page (because he used a different poset associated to the
ideal $I$) but they coincide at the $E_2$-page; indeed, it follows
from the fact that our poset is cofinal with respect to the one
considered
in \cite[2.1]{Lyu2006}.

\vskip 2mm


$\bullet$ {\bf Generalized ideal transform:} \label{transformadas de
Nagata: take 1} The direct system of functors $T_{[*]}=D_{[*]}
(N,-)$ of \emph{generalized Nagata's ideal transforms} verifies all
the requirements \cite{AazarSazeedeh2004}. In this case,
$T=D_I (N,-)$ and recall that by definition, given
$p\in\widehat{P}$, we have $D_{I_p} (N,M):=\colim_{k\in\N}\Hom_A
(I_p^k N,M)$ for any $A$-module $M$. By
\cite[2.2]{AazarSazeedeh2004}, there is an exact sequence
\[
0\longrightarrow\Gamma_{I_p} (N,E(A/\mathfrak{p}))\longrightarrow\Hom_A (N,E(A/\mathfrak{p}))\longrightarrow D_{I_p} (N,E(A/\mathfrak{p}))\longrightarrow H_{I_p}^1 (N,E(A/\mathfrak{p})).
\]
So, since $E(A/\mathfrak{p})$ is injective, it follows that $H_{I_p}^1 \left(N,E(A/\mathfrak{p})\right)=0,$ whence one can arrange the previous exact sequence in the following way:
\[
\xymatrix{0\ar[r]& \Gamma_{I_p} (N,E(A/\mathfrak{p}))\ar[r]& \Hom_A (N,E(A/\mathfrak{p}))\ar[r]& D_{I_p} (N,E(A/\mathfrak{p}))\ar[r]& 0.}
\]
Therefore, combining this short exact sequence with the calculation carried out in the previous part we get
\[
D_{I_p} (N,E(A/\mathfrak{p}))=\begin{cases}\Hom_A (N,E(A/\mathfrak{p})),\text{ if }\mathfrak{p}\notin\mathbf{V} (I_p),\\ 0,\text{ if }\mathfrak{p}\in\mathbf{V} (I_p).\end{cases}
\]
Finally, given any $\mathfrak{m}\in\Max (A)$ one has, as a direct application of \cite[4.1.7]{BroSha} (indeed, $N$ is finitely generated and the localization map $\xymatrix@1{A\ar[r]& A_{\mathfrak{m}}}$ is flat), that
\[
D_{I_p} (N,E(A/\mathfrak{p}))_{\mathfrak{m}}=\begin{cases} \Hom_{A_{\mathfrak{m}}} \left(N_{\mathfrak{m}},E(A/\mathfrak{p})_{\mathfrak{m}}\right),\text{ if }\mathfrak{p}\notin\mathbf{V} (I_p)\text{ and }\mathfrak{p}\subseteq\mathfrak{m},\\ 0,\text{ otherwise.}\end{cases}
\]

\vskip 2mm

Therefore we obtain the following spectral sequence:
\[
E_2^{-i,j}=\Lf_i \colim_{p\in P} \R^j
D_{[*]}(N,M)\xymatrix{\hbox{}\ar@{=>}[r]_-{i}&} \R^{j-i}
D_I (N,M).
\]
Notice that, when $N=A$, $$E_2^{-i,j}=\Lf_i \colim_{p\in P} \R^j
D_{[*]} (M)\xymatrix@1{\hbox{}\ar@{=>}[r]_-{i}& \R^{j-i}D_I(M).}$$
In fact, this spectral sequence for ordinary Nagata's ideal
transforms may be regarded as the module version of the one obtained in
\cite[3.5]{Castano2013}. On the other hand, this spectral sequence
is very closely related with the previous Mayer-Vietoris spectral
sequence for local cohomology modules because of the well-known
isomorphism $\R^j D_J (M)\cong H_J^{j+1} (M)$ for any $j\geq 1$ and
for any ideal $J$ of $A$ \cite[2.2.6]{BroSha}. Finally, note
that, when $n=2$, we recover the long exact sequence
\begin{align*}
& 0\longrightarrow D_{I_1+I_2} (M)\longrightarrow D_{I_1} (M)\oplus
D_{I_2} (M)\longrightarrow D_I (M)\longrightarrow H_{I_1+I_2}^2
(M)\\ & \ldots\longrightarrow H_{I_1}^j (M)\oplus H_{I_2}^j
(M)\longrightarrow H_I^j (M)\longrightarrow H_{I_1+I_2}^{j+1}
(M)\longrightarrow \ldots
\end{align*}
given in \cite[3.2.5]{BroSha}.


\vskip 2mm

$\bullet$ {\bf Local cohomology with respect to pairs of ideals:}
\label{cohomologia local de parejas: take 1} Let $J$ be an arbitrary
ideal of $A$. The direct system of \emph{torsion functors with
respect to pairs of ideals} $T_{[*]}=\Gamma_{[*],J}$ is given by
$T_p (M):=\Gamma_{I_p,J} (M)$; that is, the $(I_p,J)$-torsion module
with respect to $M$. The reader should remember that, for any ideal
$K$ of $A$, the torsion functor $\Gamma_{K,J}$ is defined in the
following manner:
\[
\Gamma_{K,J} (M):=\left\{m\in M\mid\quad K^l m\subseteq Jm\text{ for some }l\in\N\right\}.
\]
Furthermore, it is known \cite[1.11]{TakahashiYoshinoYoshizawa2009} that $T_{[*]}$ verifies the previous requirements. In particular, $T=\Gamma_{I,J}$. Moreover, using \cite[1.11]{TakahashiYoshinoYoshizawa2009} once again, one deduces that
\[
\Gamma_{I_p,J} (E(A/\mathfrak{p}))=\begin{cases} E(A/\mathfrak{p}),\text{ if }\mathfrak{p}\in\mathbf{W} (I_p,J),\\ 0,\text{ if }\mathfrak{p}\notin\mathbf{W} (I_p,J).\end{cases}
\]
Here, $\mathbf{W}(I_p,J):=\{\mathfrak{p}\in\spec (A)\mid\quad I_p^n\subseteq\mathfrak{p}+J\text{ for some integer }n\geq 1\}$.

\vskip 2mm

Therefore we obtain the following spectral sequence:
$$E_2^{-i,j}=\Lf_i \colim_{p\in P} H_{[*],J}^j
(M)\xymatrix@1{\hbox{}\ar@{=>}[r]_-{i}& }H_{I,J}^{j-i}(M).$$ It
turns out that, when $n=2$, this spectral sequence degenerates into
the following long exact sequence:
\[
\xymatrix{\ldots\ar[r]& H_{I_1,J}^j (M)\oplus H_{I_2,J}^j (M)\ar[r]&
H_{I,J}^j (M)\ar[r]& H_{I_1+I_2,J}^{j+1} (M)\ar[r]& \ldots}
\]
This long exact sequence may be regarded as a Mayer-Vietoris long
exact sequence for local cohomology modules with respect to pairs of
ideals; it is worth mentioning that, to the best of our knowledge,
this is the first time that this long exact sequence appears in the
literature.

\subsection{Enhanced structure} The main purpose now is to show,
roughly speaking, that if $M$ is not only an $A$-module, but also
has a certain additional structure, then the spectral sequence
produced in \textbf{Theorem} \ref{una primera construccion para calentar}
also inherits this structure. Throughout this part, in addition to
the assumptions and notations we establish in
\textbf{Setup \ref{las hipotesis del caso homologico}}, we also
require the following:

\begin{ass}\label{subcategoria de Serre maja}
Let $\mathcal{S}\subseteq\cA$ be an abelian subcategory closed under
subobjects, subquotients, and extensions such that, for any object
$N\in\mathcal{S}$ and for any $p\in\widehat{P}$, $T_p
(N)\in\mathcal{S}$; moreover, if $M\in\mathcal{S}$ then we suppose
that there is a long exact sequence $\xymatrix@1{0\ar[r]& M\ar[r]&
E^*}$ in $\cA$ such that:

\begin{enumerate}[(i)]

\item For any $p\in\widehat{P}$ and for any $j\geq 0$, $\R^k T_p (E^j)=0$ for all $k\geq 1$.

\item The long exact sequence $\xymatrix@1{0\ar[r]& M\ar[r]& E^*}$ may be regarded as an exact cochain complex
in $\mathcal{S}.$

\item If $A$ contains a field $\K,$ then we also assume that, for any
$M\in\mathcal{S},$ $\K\otimes_A M$ is also an object of
$\mathcal{S}$ such that the map $\xymatrix@1{\K\otimes_A M\ar[r]& M,}$
given by the assignment $r\otimes m\longmapsto rm,$ is a natural
isomorphism in $\mathcal{S}.$

\end{enumerate}

\end{ass}

We are now in position to establish the main result of
this part; namely, the following:

\begin{teo}\label{enhanced1}
Under the assumptions of \textbf{Setup \ref{las hipotesis del caso homologico}}, if parts (i) and (ii) of \textbf{Assumption} \ref{subcategoria de Serre maja} are satisfied, then
one has that the spectral sequence $E_2^{-i,j}=\Lf_i \colim_{p\in P} \R^j T_{[*]}
(M)\xymatrix@1{\hbox{}\ar@{=>}[r]_-{i}& \R^{j-i} T (M)}$ can be
regarded as spectral sequence in the category $\mathcal{S}$.
\end{teo}

\begin{proof}
Let $\xymatrix@1{0\ar[r]& M\ar[r]&E^*}$ be the long exact sequence given by
Assumption \ref{subcategoria de Serre maja}; along the proof of
Theorem \ref{una primera construccion para calentar} we already saw
that the $E_0$--page of our spectral sequence is given by the double
complex $\roos_i\left(T_{[*]}\left(E^j\right)\right),$ which again
using Assumption \ref{subcategoria de Serre maja} can be regarded
as a double complex in $\mathcal{S}.$ In this way, since all the pages
of our spectral sequence are given by subquotients of our starting
double complex, and $\mathcal{S}$ is closed under subquotients, it
follows that $E_2^{-i,j}=\Lf_i \colim_{p\in P} \R^j T_{[*]}
(M)\xymatrix@1{\hbox{}\ar@{=>}[r]_-{i}& \R^{j-i} T (M)}$ can be
regarded as spectral sequence in the category $\mathcal{S},$ as claimed.
\end{proof}

Now, we want to introduce some specific examples where the
assumptions of \ref{subcategoria de Serre maja} are satisfied; in
all these examples, unless otherwise is specified, $T_{[*]}$ is either the torsion functor
$\Gamma_{[*]}$ or the ideal transform functor $D_{[*]}$.


\vskip 2mm

$\bullet$  \textbf{F-modules}: Let $A$ be a commutative regular ring
containing a field of prime characteristic $p$; it was shown in
\cite[(1.2b'')]{Lyu1997} that, if $M$ is an $F$-module, and
$\xymatrix@1{0\ar[r]& M\ar[r]& E^*}$ is its minimal injective
resolution in $\cA$, then $\xymatrix@1{0\ar[r]& M\ar[r]& E^*}$ can
also be regarded as an exact cochain complex in the category of
$F$-modules. On the other hand, any injective $A$-module is
$T_p$-acyclic for any $p\in\widehat{P}$; moreover, if $N$ is any
$F$-module, then, for any $p\in\widehat{P}$, $T_p (N)$ is so.
Summing up, if $\mathcal{S}$ is the category of $F$-modules, then parts (i) and (ii)
of \textbf{Assumption} \ref{subcategoria de Serre maja} hold; this fact was already pointed
out in \cite[1.2(iii)]{AlvarezGarciaZarzuela2003}. Finally, notice that
one has to require that $\K=\mathbb{F}_p$ in order to guarantee
that part (iii) also holds; indeed, this is because, given an
$F$--module $M$ with structural isomorphism $\phi,$ if one wants to
put an $F$--module structure on $\K\otimes_A M$ compatible with
the natural isomorphism $\xymatrix@1{\K\otimes_A M\ar[r]& M,}$ then
one way to do so is to require that the action of $\phi$ on $\K$
is trivial, and this is the case when $\K=\mathbb{F}_p.$

\vskip 2mm

$\bullet$ \textbf{Holonomic D-modules:} Let $\K$ be a field of
characteristic zero, let $A$ be a commutative Noetherian regular
ring containing $\K$, let $D=D_{A\mid\K}$ be the ring of $\K$-linear
differential operators on $A$, let $\mathcal{S}$ be the category of
holonomic (left) $D$-modules, and let $M$ be a (left) holonomic
$D$-module. Since the category of left $D$-modules has enough
injectives, we can pick  an
injective resolution $\xymatrix@1{0\ar[r]& M\ar[r]& E^*}$ of $M$ in $\mathcal{S}$; moreover, since $D$ is
a free, right $A$-module it follows either from \cite[Corollary 1.1
(2)]{Nastasescu1989} or from \cite[Theorem 1 and Example
1]{Puthenpurakal2013} that $\xymatrix@1{0\ar[r]& M\ar[r]& E^*}$ can
also be regarded as an injective resolution of $A$-modules, which
implies that for any $p\in\widehat{P}$ and for any $j\geq 0$, $\R^k
T_p (E^j)=0$ for all $k\geq 1$. Finally, it is also known that, for
any $p\in\widehat{P}$ and for any holonomic $D$-module $N$, $T_p (N)$ is also
a holonomic $D$-module; in this way, $\mathcal{S}$ satisfies parts (i) and (ii) of the
assumptions required in \ref{subcategoria de Serre maja}. This fact
was already pointed out in
\cite[1.2(iii)]{AlvarezGarciaZarzuela2003}. In particular, it
recovers and extends \cite[3.3]{Puthenpurakal2013}. Finally, notice
that part (iii) also holds mainly because the action on any $\K$--linear
differential operator on $\K$ is trivial (see \cite[Chapter 14, Corollary 3.5]
{Coutinho1995}, in the notation there one has to consider the projection
onto a point to obtain the result).


\vskip 2mm

$\bullet$ \textbf{Quasi--holonomic D-modules:} Let $\K$ be a field of
characteristic zero, let $A$ be a \textbf{differentiable admissible
$\K$--algebra} (see \cite[Hypothesis 2.3]{Betancourt2013} and
\cite[Definition 1.2.3.6]{Narvaez2014}), let $T_{[*]}$ be the torsion functor $\Gamma_{[*]}$
or the torsion functor with respect to pairs of ideals $\Gamma_{[*],J},$ let $D=D_{A\mid\K}$ be the ring of $\K$-linear
differential operators on $A$, let $\mathcal{S}$ be the category of
\textbf{quasi--holonomic} (left) $D$-modules
\cite[Definition 3.2]{SarriaCallejasCaro2017}, and let $M$ be a (left) quasi--holonomic
$D$-module; first of all, it is known \cite[Lemma 3.5]{SarriaCallejasCaro2017}
that $\mathcal{S}$ is closed under formation of submodules, quotients and
extensions. Secondly, since the category of left $D$-modules has enough
injectives, we can pick $\xymatrix@1{0\ar[r]& M\ar[r]& E^*}$ an
injective resolution of $M$ in $\mathcal{S}$; moreover, since $D$ is
a free, right $A$-module \cite[pages 2240--2241]{SarriaCallejasCaro2017} it follows either from \cite[Corollary 1.1
(2)]{Nastasescu1989} or from \cite[Theorem 1 and Example
1]{Puthenpurakal2013} that $\xymatrix@1{0\ar[r]& M\ar[r]& E^*}$ can
also be regarded as an injective resolution of $A$-modules, which
implies that for any $p\in\widehat{P}$ and for any $j\geq 0$, $\R^k
T_p (E^j)=0$ for all $k\geq 1$. Finally, it is also known that, for
any $p\in\widehat{P}$ and for any quasi--holonomic $D$-module $N$, $T_p (N)$ is also
a quasi--holonomic $D$-module \cite[Example 3.14]{SarriaCallejasCaro2017}; in this way, we have
already checked that $\mathcal{S}$ satisfies part (i) and (ii) of Assumption
\ref{subcategoria de Serre maja}. It also satisfies part (iii) because
quasi--holonomic $D$--modules are exactly the ones that can be expressed as colimit
of holonomic $D$--modules \cite[Corollary 3.7]{SarriaCallejasCaro2017}, tensor
products commute with colimits, and that we already saw that, when $M$ is
holonomic, $\K\otimes_A M$ is also holonomic.

\vskip 2mm

$\bullet$ \textbf{Graded modules:} Let $A$ be $\K[x_1,\ldots ,x_d]$
graded by $\Z^n$ (for some $n\geq 1$), and let $\mathcal{S}$ be the
category of $\Z^n$-graded $A$-modules; since $\mathcal{S}$ is an
abelian category with enough injectives \cite[13.2.4]{BroSha},
we can choose, given any graded $A$-module $M$, a resolution
$\xymatrix@1{0\ar[r]& M\ar[r]& E^*}$ by injective objects of
$\mathcal{S}$. Moreover, by \cite[13.2.6]{BroSha} one has that, for
any injective object of $\mathcal{S}$ (namely, $E$), one has that
$\R^k T_p (E)=0$ for all $p\in\widehat{P}$, and $k\geq 1$. Finally,
it is also known that, for any $p\in\widehat{P}$ and for any graded
$A$-module $N$, $T_p (N)$ is so. Therefore, $\mathcal{S}$ satisfies
the assumptions required in \ref{subcategoria de Serre maja} (notice that
part (iii) also holds mainly because $\K$ is concentrated in degree zero). This
fact was already pointed out in \cite[Proof of
2.1]{AlvarezGarciaZarzuela2003}.

\vskip 2mm

$\bullet$ \textbf{Modules over the group ring:} Let $A$ be any
commutative regular ring containing a field, let $G$ be a group, let $A[G]$ be the
group ring, and let $M$ be a (left) $A[G]$-module. Since
$\mathcal{S}:=A[G]-Mod$ has enough injectives, we can pick
$\xymatrix@1{0\ar[r]& M\ar[r]& E^*}$ an injective resolution of $M$
in $\mathcal{S}$; moreover, since $A[G]$ is a free, right $A$-module
it follows either from \cite[Corollary 1.1 (2)]{Nastasescu1989} or
from \cite[Theorem 1 and Example 8]{Puthenpurakal2013} that
$\xymatrix@1{0\ar[r]& M\ar[r]& E^*}$ can also be regarded as an
injective resolution of $A$-modules, which implies that for any
$p\in\widehat{P}$ and for any $j\geq 0$, $\R^k T_p (E^j)=0$ for all
$k\geq 1$. Finally, it is also known that, for any $p\in\widehat{P}$
and for any $A[G]$-module $N$, $T_p (N)$ is also a $A[G]$-module. In
this way, $\mathcal{S}$ satisfies the assumptions required in
\ref{subcategoria de Serre maja} (again, notice that
part (iii) also holds mainly because $\K$ is concentrated in degree zero). This fact was already shown in
\cite[pages 640--641]{Lyu2006} with a more direct approach.

\subsection{Degeneration of homological spectral sequences}
 The goal of this subsection is to provide sufficient conditions in order to guarantee that
the previously mentioned homological spectral sequences degenerate
at the $E_2$-page. First we shall collect some preliminary facts
which will simplify the proofs of the main result of this
subsection.

\begin{df}
Let $q\in P$ and let $M$ be an object of $\cA$. The direct system \emph{represented by $M$ on $q$} (namely, $M_q$)
is defined as follows: for any $p\in P$,
\[
(M_q)_p:=\begin{cases} M,\quad\text{if}\quad p=q,\\ 0,\quad\text{otherwise.}\end{cases}
\]
\end{df}
Next result computes the colimit of this special construction (cf.\,\cite[4.5 (ii)]{BjornerEkedahl97} and
\cite[8.7]{BrunBrunsRomer2007}).

\begin{lm}\label{homologia del skyscraper haz}
Let $q\in P$ and let $M$ be an object of $\cA$. For any $i\in\N$,
\[
\Lf_i \colim_{p\in P} (M_q)_p\cong\widetilde{H}_{i-1} ((q,1_{\widehat{P}}); M),
\]
where the tilde denotes reduced simplicial homology. We agree that the reduced homology of the empty simplicial
complex is $M$ in degree $-1$ and zero otherwise.
\end{lm}

\begin{proof}
Consider the following direct systems:
\[
\left(M_{>q}\right)_p:=\begin{cases} M,\text{ if }p\in (q,1_{\widehat{P}}),\\ 0,\text{ otherwise,}\end{cases} \text{ and }
\left(M_{\geq q}\right)_p:=\begin{cases} M,\text{ if }p\in [q,1_{\widehat{P}}),\\ 0,\text{ otherwise.}\end{cases}
\]
Thus we have the following short exact sequence in $\Dir (P,\mathcal{A})$:
\[
\xymatrix{0\ar[r]& M_{>q}\ar[r]& M_{\geq q}\ar[r]& M_q\ar[r]& 0.}
\]
This leads to the following short exact sequence of chain complexes:
\[
\xymatrix{0\ar[r]& \roos_* (M_{>q})\ar[r]& \roos_* (M_{\geq q})\ar[r]& \roos_* (M_q)\ar[r]& 0.}
\]
Moreover, this short exact sequence of complexes induces the following long exact one in homology:
\[
\ldots\rightarrow H_i (\roos_* (M_{>q}))\rightarrow H_i (\roos_* (M_{\geq q}))\rightarrow H_i (\roos_* (M_q))
\rightarrow H_{i-1} (\roos_* (M_{>q}))\rightarrow\ldots
\]
On the one hand, the chain complex $\roos_* (M_{\geq q})$ (respectively, $\roos_* (M_{> q})$) agrees with the one for
computing the reduced simplicial homology of the topological space $[q,1_{\widehat{P}})$ (respectively,
$(q,1_{\widehat{P}})$) with coefficients on $M$, hence one has the following natural isomorphisms:
\[
H_i (\roos_* (M_{\geq q}))\cong\widetilde{H}_i ([q,1_{\widehat{P}});M),\ H_i (\roos_* (M_{>q}))\cong
\widetilde{H}_i ((q,1_{\widehat{P}});M).
\]
In addition, since $[q,1_{\widehat{P}})$ is contractible (see \textbf{Lemma} \ref{contractibilidad de abiertos en la
topologia de Alexandrov}) it follows that $\widetilde{H}_i ([q,1_{\widehat{P}});M)=0$ for all $i\geq 0$; in this way,
combining all the foregoing one has, for any $i\geq 1$, that there is a canonical isomorphism
$H_i (\roos_* (M_{\geq q}))\cong\widetilde{H}_{i-1} ((q,1_{\widehat{P}});M)$. Finally, combining this isomorphism
with the fact that the homology groups of the Roos chain complex agree with the left derived functors of the colimit
(see \textbf{Lemma} \ref{augmentacion del limite directo}), one finally obtains that
$\Lf_i \colim_{p\in P} (M_q)_p\cong\widetilde{H}_{i-1} ((q,1_{\widehat{P}}); M)$, just what we finally wanted to show.
\end{proof}


Now, we introduce the main result of this subsection; notice that the assumptions imposed in
the following result are a slight generalization of the ones imposed in \cite[1.2]{AlvarezGarciaZarzuela2003} and
\cite[1.1]{BrunBrunsRomer2007}.

\begin{teo}\label{extension del trabajo de master}
Let $A$ be a commutative Noetherian ring containing a field $\K$, let $T_{[*]}$ be the functor introduced in
\textbf{Setup \ref{las hipotesis del caso homologico}} and let $M$ be an object of $\cA$ verifying the following requirements.
\begin{enumerate}[(a)]

\item For any $p\in P$, $\R^j T_p (M)=0$ up to a unique value of $j$ (namely, $h_p$).

\item For any pair of distinct elements $p$ and $q$ of $P$, $\Hom_A (\R^{h_p} T_p (M),\R^{h_q} T_q (M))=0$.

\end{enumerate}
Then, there are canonical isomorphisms of $A$--modules
\[
\Lf_i \colim_{p\in P}\R^j T_{[*]}(M)\cong\bigoplus_{j=h_q}
\left(\widetilde{H}_{i-1} ((q,1_{\widehat{P}}); \K)
\otimes_{\K} \R^{h_q} T_q (M)\right)\cong\bigoplus_{j=h_q}\R^{h_q} T_q (M)^{\oplus m_{i,q}}
\]
(where $m_{i,q}:=\dim_{\K} (\widetilde{H}_{i-1} ((q,1_{\widehat{P}}); \K))$), and
$E_2^{-i,j}=\Lf_i \colim_{p\in P} \R^j T_{[*]} (M)\xymatrix@1{\hbox{}\ar@{=>}[r]_-{i}& \R^{j-i} T (M)}$
degenerates at the $E_2$-page.
\end{teo}

\begin{proof}
Parts (a) and (b) of our assumptions imply that there is a canonical
isomorphism of direct systems of $A$--modules
\[
\R^j T_{[*]}(M)\cong\bigoplus_{j=h_q} \left(\R^{h_q} T_q
(M)\right)_q .
\]
Indeed, the direct system $\R^j T_{[*]}(M)$ is the one given by
$(\R^{j} T_p (M))_{p\in P}=(\R^{h_p} T_p (M))_{p\in P}$ ($j=h_p)$,
where the only non-zero structural maps are identities; in this way,
these facts imply that this direct system splits into the above
direct sum.

\vskip 2mm

Now, fix $i\in\N$. Applying to the above decomposition the $i$th
left derived functor of the colimit over $P$, we get the following
canonical isomorphism of $A$--modules:
\[
\Lf_i \colim_{p\in P}\R^j T_{[*]}(M)\cong\bigoplus_{j=h_q}\Lf_i
\colim_{p\in P}\left(\R^{h_q} T_q (M)\right)_q .
\]
Moreover, \textbf{Lemma} \ref{homologia del skyscraper haz} implies that
there is a canonical isomorphism of $A$--modules:
\[
\Lf_i \colim_{p\in P}\R^j T_{[*]}(M)\cong\bigoplus_{j=h_q} \widetilde{H}_{i-1} ((q,1_{\widehat{P}}); \R^{h_q} T_q (M)).
\]
Next, since the map $\widetilde{H}_{i-1} ((q,1_{\widehat{P}}); \K)\otimes_{\K} \R^{h_q} T_q (M)\rightarrow\widetilde{H}_{i-1} ((q,1_{\widehat{P}}); \R^{h_q} T_q (M))$ given by the assignment $cls (z)\otimes a\longmapsto cls(z\otimes a)$ is a natural isomorphism
of $A$--modules (indeed, because any $\widetilde{H}_{i-1} ((q,1_{\widehat{P}}); \K)$ is flat over the field
$\K$) one obtains the below natural isomorphism of $A$--modules:
\[
\Lf_i \colim_{p\in P}\R^j T_{[*]}(M)\cong\bigoplus_{j=h_q}\left(\widetilde{H}_{i-1} ((q,1_{\widehat{P}}); \K)
\otimes_{\K} \R^{h_q} T_q (M)\right).
\]
Now, set $m_{i,q}:=\dim_{\K} (\widetilde{H}_{i-1} ((q,1_{\widehat{P}}); \K)),$ so
$\widetilde{H}_{i-1} ((q,1_{\widehat{P}}); \K)\cong\K^{\oplus m_{i,q}}.$ On the other hand, the natural map
$\xymatrix@1{\K\otimes_{\K} \R^{h_q} T_q (M)\ar[r]& \R^{h_q} T_q (M)}$ given by the assignment $r\otimes x\longmapsto rx$
is a natural isomorphism of $A$-modules obtained by extending scalars
\cite[Propositions 2.51 and 2.58]{Rotman2009}, and since tensor product
preserves direct sums \cite[Theorem 2.65]{Rotman2009}, one obtains
a natural isomorphism $\widetilde{H}_{i-1} ((q,1_{\widehat{P}}); \K)\otimes_{\K} \R^{h_q} T_q (M)\cong\R^{h_q} T_q (M)^{\oplus m_{i,q}}$
of $A$--modules, and therefore we finally have a natural isomorphism
\[
\bigoplus_{j=h_q}\left(\widetilde{H}_{i-1} ((q,1_{\widehat{P}}); \K)
\otimes_{\K} \R^{h_q} T_q (M)\right)\cong\bigoplus_{j=h_q}\R^{h_q} T_q (M)^{\oplus m_{i,q}}
\]
of $A$--modules. In this way, combining the previous isomorphism
joint with part (b) of our assumptions one obtains the announced degeneration.
\end{proof}


When a spectral sequence degenerates at the $E_2$-page, it is natural to consider the corresponding filtration which
this degeneration provides. This is the content of the following direct consequence of
\textbf{Theorem} \ref{extension del trabajo de master}.

\begin{cor}\label{habemus filtracion}
Let $A$ be a commutative Noetherian ring containing a field $\K$, let $T_{[*]}$ be the functor introduced
in Section \ref{construccion de libros} and let $M$ be an object of $\cA$ verifying the following requirements.
\begin{enumerate}[(a)]

\item For any $p\in P$, $\R^j T_p (M)=0$ up to a unique value of $j$ (namely, $h_p$).

\item For any pair of distinct elements $p$ and $q$ of $P$, $\Hom_A (\R^{h_p} T_p (M),\R^{h_q} T_q (M))=0$.

\end{enumerate}
Then, for each $0\leq r\leq\cdim (T)$ (\textbf{see Setup \ref{las hipotesis del caso homologico} for
the definition of $\cdim (T)$}) there is an increasing, finite filtration $\{G_k^r\}$ of $\R^r T(M)$ by
$A$-modules such that, for any $k\geq 0,$
\[
G_k^r/G_{k-1}^r\cong\bigoplus_{\{q\in P\ \mid\ k+r=h_q\}}
\R^{h_q} T_q (M)^{\oplus m_{k,q}},
\]
where $m_{k,q}:=\dim_{\K} (\widetilde{H}_{k-1} ((q,1_{\widehat{P}}); \K))$ and
we follow the convention that $G_{-1}^r =0.$
\end{cor}

\begin{proof}
Under these assumptions, \textbf{Theorem} \ref{extension del trabajo de master} ensures that the spectral sequence
\[
E_2^{-i,j}=\Lf_i \colim_{p\in P} \R^j T_{[*]} (M)\xymatrix{\hbox{}\ar@{=>}[r]_-{i}& \R^{j-i} T (M)}
\]
degenerates at the $E_2$-page; moreover, \textbf{Theorem} \ref{extension del trabajo de master} says that
\[
E_2^{-i,j}=\bigoplus_{\{q\in P\ \mid\ j=h_q\}} \R^{h_q} T_q (M)^{\oplus m_{i,q}}.
\]
Now, fix $0\leq r\leq\cdim (T)$. From the degeneration at the $E_2$-page obtained in
\textbf{Theorem} \ref{extension del trabajo de master} and the very definition of convergence of a spectral
sequence \cite[pp.\,626--627]{Rotman2009}, one gets an increasing finite filtration $\{G_k^r\}$ of
$\R^r T(M)$ by $A$-modules such that $G_k^r/G_{k-1}^r \cong E_2^{-k,k+r}$. Therefore we get
\[
G_k^r/G_{k-1}^r\cong\bigoplus_{\{q\in P\ \mid\ k+r=h_q\}} \R^{h_q} T_q (M)^{\oplus m_{k,q}},
\]
just what we finally wanted to show.
\end{proof}

We also write down the enriched
version (see \textbf{Assumption} \ref{subcategoria de Serre maja}) of Corollary \ref{habemus filtracion}, which
can be proved using the same kind of arguments involved in the
proof of \textbf{Theorem} \ref{enhanced1} and \textbf{Corollary}
\ref{habemus filtracion}.

\begin{cor}\label{habemus filtracion enriquecida}
Let $A$ be a commutative Noetherian ring containing a field $\K$, let $T_{[*]}$ be the functor introduced
in Section \ref{construccion de libros}, let $\mathcal{S}$ be a subcategory
of $\cA$ satisfying \textbf{Assumption} \ref{subcategoria de Serre maja}, and let $M$ be an object of $\mathcal{S}$ verifying the following requirements.
\begin{enumerate}[(a)]

\item For any $p\in P$, $\R^j T_p (M)=0$ up to a unique value of $j$ (namely, $h_p$).

\item For any pair of distinct elements $p$ and $q$ of $P$, $\Hom_{\mathcal{S}} (\R^{h_p} T_p (M),\R^{h_q} T_q (M))=0$.

\end{enumerate}
Then, for each $0\leq r\leq\cdim (T)$ there is an
increasing, finite filtration $\{G_k^r\}$ of $\R^r T(M)$ by objects of
$\mathcal{S}$ such that, for any $k\geq 0,$
\[
G_k^r/G_{k-1}^r\cong\bigoplus_{\{q\in P\ \mid\ k+r=h_q\}}
\R^{h_q} T_q (M)^{\oplus m_{k,q}},
\]
(where $m_{k,q}:=\dim_{\K} (\widetilde{H}_{k-1} ((q,1_{\widehat{P}}); \K))$ and we follow the convention that $G_{-1}^r =0$), and all these isomorphisms are isomorphisms in the category $\mathcal{S}$.
\end{cor}

\subsubsection{Examples of degeneration}
The goal of this part is to specialize \textbf{Theorem} \ref{extension del
trabajo de master} (and, in particular, Corollary \ref{habemus
filtracion}) to several specific choices of $T_{[*]}$ and $P$.

Keeping this aim in mind, we review for the convenience of the
reader the following:

\begin{df}\label{cohomologically complete intersection ideals}
Let $A$ be a commutative Noetherian ring, and let $I\subseteq A$ be
an ideal; it is said that $I$ is \emph{cohomologically complete
intersection} provided $H_I^k (A)\neq 0$ for any $k\neq\height (I).$
\end{df}
Cohomologically complete intersection ideals were introduced by
Hellus and Schenzel in \cite{HellusSchenzel2008}, where the
interested reader on this notion can find further details and
results. This situation is achieved, among others, in the following
cases:

\begin{itemize}

\item [$\cdot$] $A$ is a regular  ring containing a field of prime
characteristic, which is either local or a polynomial ring
over that field, and $I\subseteq A$ is an ideal (homogeneous
in the polynomial ring case) such that $A/I$ is Cohen-Macaulay \cite[Chapter IV, Proposition
4.1 and Corollary 4.2]{PeskineSzpiro1973}.

\item [$\cdot$] $A$ is either a polynomial ring or a formal power
series ring over a field of arbitrary characteristic, $A/I$ is Cohen-Macaulay and $I$ is a squarefree monomial
ideal \cite[Theorem 1 (iii)]{Lyubeznikregularsequences}.

\end{itemize}

When $A/I$ is Cohen-Macaulay containing a field of characteristic
zero, it is no longer true that $I$ is a cohomologically complete intersection
ideal; take, for instance, as $A$ the polynomial ring
in six variables over any field of characteristic zero, and take $I$ as the ideal generated by the $2\times 2$ minors of a generic
$2\times 3$ matrix \cite[Example 21.31]{Twentyfourhours}.

\vskip 2mm

We plan to use the following property of this kind of ideals.

\begin{lm}\label{c.c.i. in Gorenstein rings}
Let $A$ be a commutative Noetherian Gorenstein ring, and let $\mathfrak{p},\mathfrak{q}$ be prime ideals of $A$; assume, in addition, that $\mathfrak{p}\not\subseteq\mathfrak{q},$ that $\height (\mathfrak{p})\neq\height (\mathfrak{q}),$ and that both $\mathfrak{p},\mathfrak{q}$ are cohomologically complete intersection ideals. Then, one has that $\Hom_A \left(H_{\mathfrak{p}}^{\height (\mathfrak{p})} (A),H_{\mathfrak{q}}^{\height (\mathfrak{q})} (A)\right)=0.$
\end{lm}

\begin{proof}
For the sake of brevity, set $h_p:=\height (\mathfrak{p})$ and
$h_q:=\height (\mathfrak{q})$; so, since
$\mathfrak{p}\not\subseteq\mathfrak{q}$ and $\mathfrak{q}$ is prime,
$(\mathfrak{q}: \mathfrak{p})=\mathfrak{q}$ and therefore
\[
\Hom_A (A/\mathfrak{p},A/\mathfrak{q})=\frac{(\mathfrak{q}:
\mathfrak{p})}{\mathfrak{q}}=0;
\]
this is the basic fact we are going to use very soon.

Assume, to get a contradiction, that there is
$0\neq\psi\in\Hom_A \left(H_{\mathfrak{p}}^{h_p}
(A),H_{\mathfrak{q}}^{h_q} (A)\right);$ since $A$ is Gorenstein, $A$
admits a minimal injective resolution of the form
\[
\xymatrix{0\ar[r]& A\ar[r]^-{d^{-1}}& E^0\ar[r]^-{d^0}& E^1\ar[r]&
\ldots \ar[r]& E^{n-1}\ar[r]^-{d^{n-1}}& E^n\ar[r]^-{d^n}& 0,}
\]
where $n=\dim (A)$ and, for any $0\leq k\leq n$,
\[
E^k=\bigoplus_{\{\mathfrak{p}'\in\spec (A):\ \height
(\mathfrak{p}')=k\}} E (A/\mathfrak{p}').
\]
This implies that $\Gamma_{\mathfrak{p}} (E^{h_p-1})=0$ and
$\Gamma_{\mathfrak{q}} (E^{h_q-1})=0$, hence
\[
H_{\mathfrak{p}}^{h_p} (A)=\frac{\ker (\Gamma_{\mathfrak{p}}
(d^{h^p}))}{\im (\Gamma_{\mathfrak{p}} (d^{h^p-1}))}\subseteq E
(A/\mathfrak{p}),
\]
and
\[
H_{\mathfrak{q}}^{h_q} (A)=\frac{\ker (\Gamma_{\mathfrak{q}}
(d^{h^q}))}{\im (\Gamma_{\mathfrak{q}} (d^{h^q-1}))}\subseteq E
(A/\mathfrak{q}).
\]
These inclusions, combined with the fact that $\Ass_A (E
(A/\mathfrak{p}))=\{\mathfrak{p}\}$ (respectively, $\Ass_A (E
(A/\mathfrak{q}))=\{\mathfrak{q}\}$), imply that $\Ass_A
(H_{\mathfrak{p}}^{h_p} (A))=\{\mathfrak{p}\}$ and $\Ass_A
(H_{\mathfrak{q}}^{h_q} (A))=\{\mathfrak{q}\}.$ Therefore, one has
that, if $\xymatrix@1{H_{\mathfrak{p}}^{h_p} (A)\ar[r]^-{\psi}&
H_{\mathfrak{q}}^{h_q} (A)}$ is non-zero, then there is a non-zero
$A$-linear map $\xymatrix@1{A/\mathfrak{p}\ar[r]& A/\mathfrak{q}};$
but this is impossible because we have already seen that $\Hom_A
(A/\mathfrak{p},A/\mathfrak{q})=0.$
\end{proof}

\begin{rk}\label{we have to avoid equal heigth}
Notice that the assumption $h_p\neq h_q$ in \textbf{Lemma} \ref{c.c.i. in Gorenstein rings} can not be dropped, in general; indeed, suppose that, for instance, $h:=h_p=h_q$ and that $\mathfrak{q}\subseteq\mathfrak{p},$ so we can write $\mathfrak{p}=\mathfrak{q}+J$ for some ideal $J\subseteq A.$ In this case, the Mayer-Vietoris long exact sequence produces the following natural map: $\xymatrix@1{H_{\mathfrak{p}}^h (A)\ar[r]& H_{\mathfrak{q}}^h (A)\oplus H_{J}^h (A).}$ If we compose this map with the projection onto the first factor of its target, then one obtains a potentially non-zero homomorphism $\xymatrix@1{H_{\mathfrak{p}}^h (A)\ar[r]& H_{\mathfrak{q}}^h (A).}$
\end{rk}

As a direct consequence of Corollary \ref{habemus filtracion} and
\textbf{Lemma} \ref{c.c.i. in Gorenstein rings}, we obtain the following:

\begin{teo}\label{habemus filtracion de cohomologia local}
Let $A$ be a commutative Noetherian Gorenstein ring containing a
field $\K$, and let $I\subseteq A$ be an ideal such that the poset
$P$ is made up by cohomologically complete intersection ideals; moreover,
suppose that, for any $p<q$, $h_q<h_p.$
Then, for each $0\leq r\leq\dim (A)$ there is an increasing, finite
filtration $\{G_k^r\}$ of $H_I^r (A)$ by $A$-modules such that, for any $k\geq 0,$
\[
G_k^r/G_{k-1}^r\cong\bigoplus_{\{q\in P\ \mid\ k+r=h_q\}}
H_{I_q}^{\height (I_q)} (A)^{\oplus m_{k,q}},
\]
where $m_{k,q}:=\dim_{\K} (\widetilde{H}_{k-1} ((q,1_{\widehat{P}}); \K))$ and we follow the convention that $G_{-1}^r =0.$
\end{teo}
Now we specialize the conclusion of \textbf{Theorem} \ref{habemus filtracion
de cohomologia local} to some specific examples.

\vskip 2mm

$\bullet$ \textbf{Linear arrangements:} Let $\K$ be a field, and let
$I\subseteq\K [x_1,\ldots ,x_d]=:A$ be the defining ideal of an
arrangement of linear varieties of $\K^d$; it is known that $I$
admits a minimal primary decomposition of the form
$I=I_1\cap\ldots\cap I_n$, where any $I_j$ is a prime ideal of $A$
generated by polynomials of total degree at most $1$; it is also
clear that any sum of the $I_k$'s gives another prime ideal of the
same kind, so the poset $P$ is made up by prime ideals of this form,
in particular all of them are cohomologically complete intersection
ideals; so, the assumptions of \textbf{Theorem} \ref{habemus filtracion de
cohomologia local} are satisfied in this case. This example is the one already studied in detail in
\cite{AlvarezGarciaZarzuela2003}.

\vskip 2mm

$\bullet$ \textbf{Toric face ideals in prime characteristic:} Let
$\Sigma\subseteq\R^d$ be a rational pointed fan, let
$\mathcal{M}_{\Sigma}$ be a monoidal complex supported on $\Sigma$,
let $\K$ be a field, and let $\K[\mathcal{M}_{\Sigma}]$ be the
corresponding toric face ring. It is known \cite[page
539]{HopNguyen2012} that
$\K[\mathcal{M}_{\Sigma}]\cong\K[x_1,\ldots ,x_d]/I$, where $I$ is
generated by squarefree monomials and pure binomials; we refer to
any $I$ of this type as a \emph{toric face ideal}. Moreover, it is
also known that $I$ admits a minimal primary decomposition
\[
I=\mathfrak{p}_{C_1}\cap\ldots\cap\mathfrak{p}_{C_t},
\]
where $\mathfrak{p}_{C_i}$ is the $\Z^d$-graded prime ideal of
$A:=\K [x_1,\ldots ,x_d]$ corresponding to one of the maximal cones
of $\Sigma$ (namely, $C_i$); in addition, since
$\mathfrak{p}_{C_i}+\mathfrak{p}_{C_j}=\mathfrak{p}_{C_i\cap C_j}$
\cite[Lemma 2.3 (i)]{IchimRomer2007} one can ensure that the
poset $P$ is made up by prime ideals of $A$.

Now, assume that $\K$ is a field of prime characteristic $p$ and
$\mathcal{M}_{\Sigma}$ is a Cohen-Macaulay monoidal complex; this
last condition means that, for any cone $C$ of $\Sigma$,
$A/\mathfrak{p}_C$ is a Cohen-Macaulay ring. Under these
assumptions, for any cone $C$ of $\Sigma$, $\mathfrak{p}_C$ is a
cohomologically complete intersection prime ideal.

Summing up, we have seen that, if $I\subseteq\K[x_1,\ldots ,x_d]$ is
the defining ideal of a toric face ring of the form
$\K[\mathcal{M}_{\Sigma}]$, where $\K$ is a field of prime
characteristic $p$ and $\mathcal{M}_{\Sigma}$ is a Cohen-Macaulay
monoidal complex, then the poset $P$ is made up by cohomologically
complete intersection prime ideals, which is what we need to prove.

\begin{rk}
We want to single out here that, from Theorem \ref{habemus filtracion de cohomologia local}, one can
immediately deduce a Musta{\c{t}}{\v{a}}--Terai type formula for certain toric face rings; we plan to write down
this formula later on in this paper (see \textbf{Theorem} \ref{Terai graduado}).
\end{rk}

\subsection{Additive functions}\label{adding structure in the homological case}
Our next goal is to exploit the filtration obtained in Corollary
\ref{habemus filtracion enriquecida} to extend the formula of characteristic
cycles of local cohomology modules given in
\cite[1.3]{AlvarezGarciaZarzuela2003}. Actually, the most general
statement we obtain, which is a direct consequence of Corollary
\ref{habemus filtracion enriquecida}, is the following:

\begin{teo}\label{una de funciones aditivas}
Under the assumptions of Corollary \ref{habemus filtracion enriquecida}, let
$\mathcal{S}$ be as in \textbf{Assumption} \ref{subcategoria de Serre maja},
let $G$ be an abelian group, and let
$\xymatrix@1{\mathcal{S}\ar[r]^-{\lambda}& G}$ be an additive
function. If $M\in\mathcal{S}$ then, for any $0\leq r\leq\cdim (T)$,
\[
\lambda\left(\R^r T(M)\right)=\sum_{p\in P} m_{r,p}\ \lambda\left(\R^{h_p} T_p (M)\right),
\]
where $m_{r,p}=\dim_{\K} \widetilde{H}_{h_p-r-1} ((p,1_{\widehat{P}});\K)$.
\end{teo}

Now, we want to specialize the above formula for several specific
choices of $\lambda$.

\vskip 2mm


$\bullet$ \textbf{Characteristic cycles:} \label{additive formula of
characteristic cycles} \hskip 2mm Let $\K$ be a field of
characteristic zero, let $\mathcal{S}$ be the category of holonomic
$D$-modules on $A$ (where $A$ is a commutative Noetherian regular
ring containing $\K$), and let $T_{[*]}$ be $\Gamma_{[*]}$. Since
the characteristic cycle $\CC$ is additive, we obtain the following
formula:
\[
\CC\left(H_I^r (A)\right)=\sum_{p\in P}m_{r,p}\ T_{X_p}^* \K^d ,
\]
where $T_{X_p}^* \K^d$ denotes the relative conormal subspace of
$T^* \K^d$ attached to $X_p =\mathbf{V} (I_p)$. As the reader can
easily point out, this formula recovers and extends the one given in
\cite[1.3]{AlvarezGarciaZarzuela2003}.

\vskip 2mm


$\bullet$ \textbf{Length of a D-module:}\label{additive formula of
Lyubeznik numbers}  Let $A$ be a commutative Noetherian regular ring
containing a field $\K$, let $D=D_{A\mid\K}$ be the ring of
$\K$-linear differential operators on $A$, let $\mathcal{S}$ be the
category of left $D$-modules, and let $\ell_D$ be the length
function of this category; since $\ell_D$ is additive, \textbf{Theorem}
\ref{una de funciones aditivas} ensures that, for any $D$-module
$M$ for which the assumptions
of Corollary \ref{habemus filtracion} are satisfied,
\[
\ell_D (H_I^r (M))=\sum_{p\in P}m_{r,p}\ \ell_D (H_{I_p}^{h_p} (M)).
\]

The reader will easily note that this formula provides an equality
 of generalized Lyubeznik numbers when $M=A$ \cite[Definition
4.3]{BetancourtWitt2014}.

\vskip 2mm

$\bullet$ \textbf{Length of an F-module:} \label{the quasi length of
Lyubeznik} Let $A$ be a commutative Noetherian regular ring
containing a field of prime characteristic $p$, let $\mathcal{S}$ be
the category of $F$-modules, and let $\ell_F$ be the length function
of this category; since $\ell_F$ is additive, \textbf{Theorem} \ref{una de
funciones aditivas} ensures that, for any $F$-module $M$ for which the assumptions
of Corollary \ref{habemus filtracion} are satisfied,
\[
\ell_F (H_I^r (M))=\sum_{p\in P}m_{r,p}\ \ell_F (H_{I_p}^{h_p} (M)).
\]
In the case $A=\mathbb{F}_p [\![x_1,\ldots ,x_d]\!]$ this equality is equivalent, by the definition of Lyubeznik's quasi
length \cite[Theorem 4.5]{Lyu1997}, to the below one:
\[
\ql(H_{\mathfrak{m}}^{d-r}(A/I))=\sum_{p\in P} m_{r,p}\ \ql (H_{\mathfrak{m}}^{\dim (A/I_p)} (A/I_p)).
\]
On the other hand, if $A=\mathbb{F}_p [x_1,\ldots ,x_n],$ then the above equality provides a closed formula for the calculation of $\ell_F (H_I^r (A));$ it is worth noting that, in \cite[Corollary 3.6]{KatzmanMaSmirnovZhang16}, they provide the following upper bound for the length:
\[
\ell_F (H_I^r (A))\leq\sum_{1\leq i_1\leq\ldots\leq i_j\leq t} \left((d_{i_1}+\ldots+d_{i_j}+1)^n-1\right),
\]
where $I$ is generated by polynomials $f_1,\ldots ,f_t$ with $\deg (f_k)=d_k.$

\section{Cohomological spectral sequences associated to inverse systems}\label{de aqui sacaremos la formula de Hochster}
As in the homological setup, the goal of this section is to build some spectral sequences, in spite of
the fact that we are mostly interested on those involving local cohomology modules; actually, the spectral
sequences we are going to build involve in their second page the right derived functors of the limit, for this
reason we refer to them as \emph{cohomological}.

\vskip 2mm

Let $A$ be a commutative Noetherian ring $A$ and $\cA$ the category
of $A$-modules. In this section we will either consider the poset
associated to an ideal $I\subseteq A$ given in \textbf{Example} \ref{el poset
de Mayer-Vietoris} or the poset associated to a rational pointed fan
$\Sigma\subseteq\R^d$ given in \textbf{Example} \ref{el poset de BBR}. The
main difference with respect to the construction of spectral
sequences given in Section \ref{construccion de libros} is that our
starting point is going to be an inverse system of $A$-modules over
a poset instead of just an $A$-module. Most commonly, given an ideal
$I\subseteq A$ and an $A$-module $M$, we will consider the inverse
system $M/[*]M:=(M/I_pM)_{p\in P}$ or, given a rational pointed fan
$\Sigma\subseteq\R^d$ we will consider the inverse system
$(\K[\mathcal{M}_{\Sigma_p}])_{p\in P}$.

\vskip 2mm

Before going on we should point out that, in general, it is not
true, neither that $M/[*]M$ is acyclic with respect to the
limit, nor the isomorphism $M/IM\cong\lim_{p\in P} M/I_pM.$ In the case that $M=A$ some sufficient conditions to
guarantee this isomorphism were given in \cite[Example 3.3]{BrunBrunsRomer2007}.

\begin{prop}\label{distributive or flasque condition}
Assume that $P$ is a \textbf{subset} of a distributive lattice of
ideals of $A$ (\textbf{with respect to sum and intersection}) or , more generally, that $A/[*]$ is
acyclic with respect to the limit and the
natural homomorphism $A \longrightarrow \lim_{p\in P} A/I_p$ is
surjective. Then, there is a canonical isomorphism $A/I\cong\lim_{p\in P}
A/I_p.$
\end{prop}

For the case of monomial ideals these conditions are satisfied \cite[Example 3.3]{BrunBrunsRomer2007}. For
toric face ideals this is also true \cite[Proposition 2.2 and Proof]{BrunsKochRomer2008}, so
we have an isomorphism $\K[\mathcal{M}_{\Sigma}]\cong\lim_{p\in P}
\K[\mathcal{M}_{\Sigma_p}].$

However, as we already observed, it is in general not true that
$M/IM$ is isomorphic to $\lim_{p\in P}
M/I_p M,$ even in the case where $M=A$. We specially thank Jack Jeffries for pointing out to us the following:

\begin{rk}\label{it doesn't work for arrangements}
The sufficient conditions given in \textbf{Proposition} \ref{distributive or flasque condition} are,
in general, \textbf{not satisfied} for ideals defining arrangements of linear varieties; more precisely, in what follows
we give an example where $A/I$ is not isomorphic to
$\lim_{p\in P} A/I_p.$ Indeed, for instance take the ideal $I=(x)\cap (y)\cap (x-y)\subseteq\K [x,y]=A,$
where $\K$ is any field. In this case, the poset $P$ is
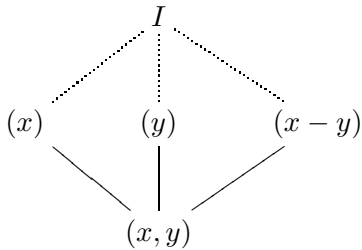
\[
\xymatrix{\hbox{}& I\\ (x)\ar@{.}[ur]& (y)\ar@{.}[u]& (x-y)\ar@{.}[ul]\\ & (x, y)\ar@{-}[ul]\ar@{-}[u]\ar@{-}[ur]& }
\]
It is straightforward to check that
\[
\left((x)+(y)\right)\cap (x-y)=(x-y)\supsetneq (x(x-y),y(x-y))=\left((x)\cap (x-y)\right)+\left((y)\cap (x-y)\right),
\]
hence $P$ can not be a subset of a distributive lattice of ideals of $A$ with respect to sum
and intersection.

Our next goal is to show that $A/I$ can not be isomorphic to $\lim_{p\in P}
A/I_p$; actually, it is enough to show that the natural monomorphism $A/I\hookrightarrow
\lim_{p\in P}A/I_p$ is not surjective. Indeed, $\lim_{p\in P}A/I_p$ is the kernel of the
unique non--zero differential of the Roos cochain complex; namely:
\begin{align*}
& A/ (x)\times A/ (y)\times A/ (x-y)\times A/ (x,y)\longrightarrow A/ (x,y)\times A/ (x,y)\times A/ (x,y)\\
& (a,b,c,d)\longmapsto (a-d,b-d,c-d).
\end{align*}
Consider the constant map $A/I\longrightarrow A/ (x)\times A/ (y)\times A/ (x-y)\times A/ (x,y)$ mapping $a$ to $(a,a,a,a);$ it
is enough to show that this constant map is not surjective.

Indeed, suppose (to get a contradiction) that, given polynomials $a,b,c\in A$ such that (without loss of generality) $a(0,0)=b(0,0)=c(0,0)=0$, there is a polynomial $P=P(x,y)$ such that $P(0,0)=0,$ $P(0,y)=a(0,y)$, $P(x,0)=b(x,0),$ and $P(x,x)=c(x,x)$ (this is equivalent to say that the above constant map is surjective). It is straightforward
to check that, under these assumptions, $\deg (P)$ (by $\deg (P)$ we mean the total degree of $P$) is less or equal than
$\deg_y (a)+\deg_x (b).$ Therefore, to get a contradiction, one only needs to pick $c$ such that
$\deg (c)\geq\deg_y (a)+\deg_x (b)+1$. For a concrete example, just take $a=y$, $b=x$, and $c=x^2$.

Summing up, we have seen that our above constant map can not be surjective, hence $A/I$
can not be isomorphic to $\lim_{p\in P}A/I_p.$ 

\end{rk}

\begin{rk}\label{it doesn't work for arrangements 2}
We also want to point out that, although $P$ may not be a subset of a distributive lattice of ideals
of $A$, \textbf{one might still have that the natural map $A \longrightarrow \lim_{p\in P} A/I_p$
is surjective}; indeed, consider $I=(x,y)\cap (x,z)\cap (x-z,y)\subseteq\K [x,y,z]=A,$ where $\K$ is any field. In this case, the poset $P$ is
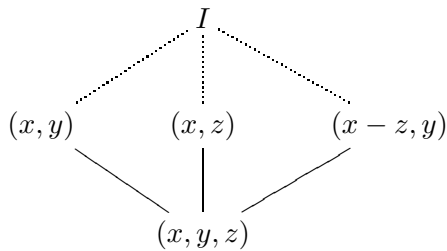
\[
\xymatrix{\hbox{}& I\\ (x,y)\ar@{.}[ur]& (x,z)\ar@{.}[u]& (x-z,y)\ar@{.}[ul]\\ & (x, y,z)\ar@{-}[ul]\ar@{-}[u]\ar@{-}[ur]& }
\]
It is straightforward to check that
\[
\left((x,y)+(x,z)\right)\cap (x-z,y)=(x-z,y)\supsetneq\left((x,y)\cap (x-z,y)\right)+\left((x,z)\cap (x-z,y)\right),
\]
hence $P$ can not be a subset of a distributive lattice of ideals of $A$ with respect to sum
and intersection. However, in this case, a similar calculation than the one carried out in
Remark \ref{it doesn't work for arrangements} shows that $A \longrightarrow \lim_{p\in P} A/I_p$
is surjective, and therefore there is an isomorphism $A/I
\cong\lim_{p\in P}A/I_p.$
\end{rk}

\subsection{Construction of cohomological spectral sequences}\label{ahi va una de sucesiones
espectrales cohomologicas}

The reader is encouraged to compare the following setup with the ones carried out in Sections \ref{las hipotesis del
caso homologico} and \ref{otra construccion mas de lo mismo}; indeed, in this case we are going to
produce an endofunctor on inverse systems of $A$--modules carrying over
a given endofunctor of $A$--modules

\begin{cons}\label{fucntores covariantes hacificados}
Let $\xymatrix@1{\cA\ar[r]^-{T}& \cA}$ be a covariant, left exact functor, and let $P$ be any finite poset.
Building upon $T$, we produce the following endofunctor on $\Inv (P,\cA)$; namely,
\begin{align*}
& \xymatrix{\Inv (P,\cA)\ar[r]^-{\mathcal{T}}& \Inv (P,\cA)}\\
& M=(M_p)_{p\in P}\longmapsto\mathcal{T} (M):=(T(M_p))_{p\in P}.
\end{align*}
In addition, we suppose that $T$ commutes with arbitrary direct sums and that there are ideals
$J,K$ of $A$ such that $T$ verifies one (and only one) of
the following two assumptions.
\begin{enumerate}[(a)]

\item For any $\mathfrak{p}\in\spec (A)$ and for any maximal ideal $\mathfrak{m}$ of $A$, there exists an
$A$-module $X$ such that
\[
T\left(E\left(A/\mathfrak{p}\right)\right)_{\mathfrak{m}}=\begin{cases} X,\text{ if }\mathfrak{p}\in\mathbf{W}(J,K)
\text{ and }\mathfrak{p}\subseteq\mathfrak{m},\\ 0,\text{ otherwise.}\end{cases}
\]
Here, $X$ depends on $J,$ $K,$ $T,$ $\mathfrak{p}$ and $\mathfrak{m}.$

\item For any $\mathfrak{p}\in\spec (A)$ and for any maximal ideal $\mathfrak{m}$ of $A$, there exists an $A$-module
$Y$ such that
\[
T\left(E\left(A/\mathfrak{p}\right)\right)_{\mathfrak{m}}=\begin{cases} Y,\text{ if }\mathfrak{p}\notin\mathbf{W}(J,K)\text{ and }
\mathfrak{p}\subseteq\mathfrak{m},\\ 0,\text{ otherwise.}\end{cases}
\]
Here, $Y$ depends on $J,$ $K,$ $T,$ $\mathfrak{p}$ and $\mathfrak{m}.$
\end{enumerate}
In what follows, for a functor $\mathcal{T},$ we denote by
$\R^j \mathcal{T}$ its corresponding $j$th right derived functor
in the abelian category $\Inv (P,\cA).$
\end{cons}

\subsubsection{Main result}

As we have previously explained, our goal is to construct an
spectral sequence which involves the right derived functors of the
limit. The following lemma turns out to be the first step in this
construction.

\begin{lm}\label{esto me suena de los tres de siempre}
Let $\Upsilon$ be an injective inverse system of the form
\[
\bigoplus_{j\in J} (E_j P)_{\geq q_j},
\]
where $J$ is a (not necessarily finite) index set, $q_j\in P$, $E_j$
is an indecomposable injective $A$-module, and
\[
\left[ (E_j P)_{\geq q_j}\right]_p:=\begin{cases} E_j,\text{ if
}p\in [q_j,1_{\widehat{P}}),\\ 0,\text{otherwise.}\end{cases}
\]
Then, the coaugmented cochain complex
\[
\xymatrix{0\ar[r]& }\left(\lim_{p\in
P}\circ\mathcal{T}\right)(\Upsilon)\xymatrix{ \ar[r]& \roos^*
(\mathcal{T}\left(\Upsilon\right))}
\]
is exact.
\end{lm}

\begin{proof}
As  $\mathcal{T}$ and $\roos^*$ commute with direct sums we may
assume, without loss of generality, that $\Upsilon
=(E(A/\mathfrak{p})P)_{\geq q}$ for some $(\mathfrak{p},q)\in\spec
(A)\times P$.

Now, we carry out a similar strategy as the one employed in the
proof of \textbf{Lemma} \ref{lema clave en el caso homologico}; indeed, fix a
maximal ideal $\mathfrak{m}$ of $A$. By the usual generalities, it
is enough to show that the cochain complex
\[
\xymatrix{0\ar[r]& }\left(\lim_{p\in
P}\circ\mathcal{T}\right)(\Upsilon)_{\mathfrak{m}}\xymatrix{ \ar[r]&
\roos^* (\mathcal{T}\left(\Upsilon\right))_{\mathfrak{m}}}
\]
is exact. If $\mathfrak{p}\not\subseteq\mathfrak{m}$ then the
previous coaugmented cochain complex is zero and we are done;
therefore, from now on we suppose that
$\mathfrak{p}\subseteq\mathfrak{m}$.

First of all, suppose that $T$ verifies requirement (a) of
\textbf{Setup} \ref{fucntores covariantes hacificados}; on the one hand,
if $\mathfrak{p}\notin\mathbf{W} (J,K)$, then our coaugmented
cochain complex is identically zero, whence we are done. On the
other hand, if $\mathfrak{p}\in\mathbf{W} (J,K)$, then this
coaugmented cochain complex turns out to be equal to the one for
computing the simplicial cohomology of the topological space
$[q,1_{\widehat{P}})$ with coefficients in $X$ (see \textbf{Lemma}
\ref{calculas cohomologia simplicial}). But this topological space
is contractible by \textbf{Lemma} \ref{contractibilidad de abiertos en la
topologia de Alexandrov}; this fact concludes the proof just in case
$T$ verifies assumption (a) of \textbf{Setup} \ref{fucntores
covariantes hacificados}.

Now, suppose that $T$ verifies
hypothesis (b) of \textbf{Setup} \ref{fucntores covariantes
hacificados} (and also that $\mathfrak{p}\subseteq\mathfrak{m}$). On
one hand, if $\mathfrak{p}\in\mathbf{W} (J,K)$, then our coaugmented
cochain complex is identically zero, whence we are done. On the
other hand, if $\mathfrak{p}\notin\mathbf{W} (J,K)$, then such
coaugmented cochain complex turns out to be equal to the one for
computing the simplicial cohomology of the topological space
$[q,1_{\widehat{P}})$ with coefficients in $Y$ (see \textbf{Lemma}
\ref{calculas cohomologia simplicial}). But this topological space
is contractible by \textbf{Lemma} \ref{contractibilidad de abiertos en la
topologia de Alexandrov}; therefore, the proof is completed.
\end{proof}
The proof of the existence of the following spectral sequence is quite similar to the ones of \textbf{Theorems}
\ref{una primera construccion para calentar} and \ref{otra
construccion Josep es cansino} and it will be omitted.

\begin{teo}\label{la sucesion espectral para rematar}
Given an inverse system $V\in\Inv(\widehat{P},\mathcal{A})$, where
$P$ is any finite poset, there is a first quadrant spectral sequence
\[
E_2^{i,j}= \R^i \lim_{p\in P}\xymatrix{\R^j\mathcal{T}
\left(V\right)\ar@{=>}[r]_-{i}& }\R^{i+j}\left(\lim_{p\in
P}\circ\mathcal{T}\right) \left(V\right).
\]
If, in addition, there is a natural equivalence of functors
\[
\lim_{p\in P}\circ\mathcal{T}\cong T\circ\lim_{p\in P},
\]
and $V$ is acyclic with respect to the limit functor, then the previous spectral sequence can be arranged
in the following manner:
\[
E_2^{i,j}= \R^i \lim_{p\in P}\xymatrix{\R^j\mathcal{T}
\left(V\right)\ar@{=>}[r]_-{i}& }\R^{i+j}T\left(\lim_{p\in
P}V_p\right).
\]
\end{teo}

\begin{rk}
It is worth mentioning here that \textbf{Theorem} \ref{la sucesion
espectral para rematar} can be regarded as an extension of the
argument pointed out in \cite[Remark 8.8]{BrunBrunsRomer2007}.
\end{rk}

\subsection{Examples} \label{examples cohomological type II}

The goal of this part is to specialize \textbf{Theorem} \ref{la sucesion
espectral para rematar} to several functors that fulfill the
assumptions of \textbf{Setup} \ref{fucntores covariantes hacificados}
and  also satisfy the natural equivalence of functors
requested in \textbf{Theorem \ref{la sucesion espectral para rematar}}.
In what follows, $J,K$ will denote arbitrary ideals of $A$,$N$ will
stand for a finitely generated $A$-module, and
$V\in\Inv(\widehat{P},\mathcal{A})$ being acyclic with respect to
the limit. Moreover we shall use the fact
that filtered colimits commute with finite limits. Before going on,
we have to review the following notion \cite{BijanZadeh1980}.

\begin{df}
Let $\Phi$ be a non-empty set of ideals of $A$. It is said that
$\Phi$ is a \emph{system of ideals} of $A$ if, whenever
$\mathfrak{a},\mathfrak{b}\in\Phi$, there is an ideal $\mathfrak{c}$
in $\Phi$ such that $\mathfrak{c}\subseteq\mathfrak{a}\mathfrak{b}$.
The reader should notice that, regarding a system of ideals $\Phi$
as a poset ordered by reverse inclusion, $\Phi$ turns out to be a
filtered poset. Keeping this in mind, one can define the bifunctor
$H_{\Phi}^i (-,-)$ by
\[
H_{\Phi}^i (N,M):=\colim_{\mathfrak{a}\in\Phi} \Ext_A^i
(N/\mathfrak{a}N,M).
\]
\end{df}



\vskip 2mm

$\bullet$ {\bf Covariant Hom:} The functor $\Hom_A (N,-)$ verifies
the assumptions of \textbf{Setup} \ref{fucntores covariantes
hacificados}. Indeed, it is enough to point out that
\[
\Hom_A \left(N,\lim_{p\in P} V_p\right)\cong\lim_{p\in P}\sHom_A
\left(N,V\right).
\]
Moreover, given $\mathfrak{p}\in\spec (A)$ and $\mathfrak{m}\in\Max
(A)$ it follows, once again as a direct consequence of
\cite[4.1.7]{BroSha}, that
\[
\Hom_A
(N,E(A/\mathfrak{p}))_{\mathfrak{m}}=\begin{cases}\Hom_{A_{\mathfrak{m}}}
\left(N_{\mathfrak{m}},E(A/\mathfrak{p})_{\mathfrak{m}}\right),\text{
if }\mathfrak{p}\subseteq\mathfrak{m},\\ 0,\text{
otherwise.}\end{cases}
\]

Therefore we obtain the following spectral sequence:

\[
E_2^{i,j}=\R^i \lim_{p\in P}\sExt_A^j \left(\lvert N\rvert,
V\right)\xymatrix{ \ar@{=>}[r]_-{i}& }\Ext_A^{i+j}
\left(N,\lim_{p\in P}V_p\right),
\]
where, as usual, $\lvert N\rvert$ denotes the constant inverse system given by $N$ with identities on $N$ as structural
morphisms.

\begin{rk}
When $I=I_1\cap I_2$ and $V=M/[*]M$ for some $A$-module $M$, this spectral sequence
degenerates without assumptions to the long exact sequence
\begin{align*}
& \longrightarrow\Ext_A^j\left(N,M/IM\right)\longrightarrow\Ext_A^j\left(N,M/I_1 M\right)\oplus\Ext_A^j
\left(N,M/I_2 M\right)\longrightarrow\Ext_A^j\left(N,M/(I_1+I_2)M\right)\\ & \longrightarrow\Ext_A^{j+1}
\left(N,M/IM\right)\longrightarrow\ldots
\end{align*}
obtained after applying the functor $\Hom_A (N,-)$ to the natural short exact sequence
\[
\xymatrix{0\ar[r]& M/IM\ar[r]& M/I_1 M\oplus M/I_2 M\ar[r]& M/\left(I_1+I_2\right)M\ar[r]& 0.}
\]
\end{rk}


\vskip 2mm

$\bullet$ {\bf Generalized local cohomology:} The generalized
torsion functor $\Gamma_J (N,-)$ verifies these requirements too. It
may be verified in the following way:
\begin{align*}
& \Gamma_J \left(N,\lim_{p\in P}V_p\right)\cong\Hom_A \left(N,\Gamma_J \left(\lim_{p\in P}V_p\right)\right)\cong\Hom_A \left(N,\lim_{p\in P}\mathcal{H}_J^0 \left(V\right)\right)\\
& \cong\lim_{p\in P}\sHom_A \left(N,\mathcal{H}_J^0
\left(V\right)\right)\cong\lim_{p\in P}\mathcal{H}_J^0
\left(N,V\right).
\end{align*}
In addition, we also have to point out that, for any
$\mathfrak{p}\in\spec (A)$ and $\mathfrak{m}\in\Max (A)$,
\[
\Gamma_J\left(N,E(A/\mathfrak{p})\right)_{\mathfrak{m}}=\begin{cases}
\Hom_{A_{\mathfrak{m}}}
\left(N_{\mathfrak{m}},E(A/\mathfrak{p})_{\mathfrak{m}}\right),\text{
if }\mathfrak{p}\in\mathbf{V}(J)\text{ and
}\mathfrak{p}\subseteq\mathfrak{m},\\ 0,\text{
otherwise.}\end{cases}
\]

\begin{rk} The ordinary torsion functor $\Gamma_J (-)$ also verifies the
 assumptions. This fact follows from the next chain of
isomorphisms:
\begin{align*}
& \Gamma_J \left(\lim_{p\in P}V_p\right)\cong\colim_{t\in\N}\Hom_A \left(A/J^t,\lim_{p\in P}V_p\right)\cong\colim_{t\in\N}\lim_{p\in P}\sHom_A \left(A/J^t,V\right)\\
& \cong\lim_{p\in P}\colim_{t\in\N}\sHom_A
\left(A/J^t,V\right)\cong\lim_{p\in P}\mathcal{H}_J^0
\left(V\right).
\end{align*}
Furthermore, the reader should also remember, given
$\mathfrak{p}\in\spec (A)$ and $\mathfrak{m}\in\Max (A)$, that
\[
\Gamma_J\left(E(A/\mathfrak{p})\right)_{\mathfrak{m}}=\begin{cases}
E(A/\mathfrak{p})_{\mathfrak{m}},\text{ if
}\mathfrak{p}\in\mathbf{V}(J)\text{ and
}\mathfrak{p}\subseteq\mathfrak{m},\\ 0,\text{
otherwise.}\end{cases}
\]
\end{rk}

Therefore we obtain the following spectral sequence:

\[
E_2^{i,j}=\R^i \lim_{p\in P}\mathcal{H}_J^j \left(\lvert N\rvert,
V\right)\xymatrix{ \ar@{=>}[r]_-{i}& }H_J^{i+j} \left(N,\lim_{p\in
P}V_p\right).
\]

\begin{rk}
When $I=I_1\cap I_2$ and $V=M/[*]M$ for some $A$-module $M$,  this spectral sequence boils
down to the long exact sequence
\begin{align*}
& \longrightarrow H_J^j\left(N,M/IM\right)\longrightarrow H_J^j\left(N,M/I_1 M\right)\oplus H_J^j\left(N,M/I_2 M\right)
\longrightarrow H_J^j\left(N,M/(I_1+I_2)M\right)\\ & \longrightarrow H_J^{j+1}\left(N,M/IM\right)\longrightarrow\ldots
\end{align*}
obtained after applying the functor $\Gamma_J (N,-)$ to the natural short exact sequence
\[
\xymatrix{0\ar[r]& M/IM\ar[r]& M/I_1 M\oplus M/I_2 M\ar[r]& M/\left(I_1+I_2\right)M\ar[r]& 0.}
\]
\end{rk}


\vskip 2mm

$\bullet$ {\bf Generalized ideal transforms:} The generalized
Nagata's ideal transform functor $D_J (N,-)$ also verifies the
previous assumptions. Indeed, we only have to notice that
\[
D_J \left(N,\lim_{p\in P}V_p\right)\cong\colim_{t\in\N}\Hom_A
\left(J^t N,\lim_{p\in P}V_p\right)\cong\colim_{t\in\N}\lim_{p\in
P}\sHom_A \left(J^t N,V\right)\cong\lim_{p\in P}\mathcal{D}_J (N,V).
\]
In addition, we also have to point out that, for any
$\mathfrak{p}\in\spec (A)$ and $\mathfrak{m}\in\Max (A)$, 
\[
D_J\left(N,E(A/\mathfrak{p})\right)_{\mathfrak{m}}=\begin{cases}
\Hom_{A_{\mathfrak{m}}}
\left(N_{\mathfrak{m}},E(A/\mathfrak{p})_{\mathfrak{m}}\right),\text{
if }\mathfrak{p}\notin\mathbf{V}(J)\text{ and
}\mathfrak{p}\subseteq\mathfrak{m},\\ 0,\text{
otherwise.}\end{cases}
\]

Therefore we obtain the following spectral sequence:

\[
E_2^{i,j}=\R^i \lim_{p\in P}\R^j\mathcal{D}_J \left(\lvert N\rvert,
V\right)\xymatrix{ \ar@{=>}[r]_-{i}& }\R^{i+j}D_J \left(N,\lim_{p\in
P}V_p \right).
\]

\begin{rk}
When $I=I_1\cap I_2$ and $V=M/[*]M$ for some $A$-module $M$, the previous spectral sequence
becomes the long exact sequence
\begin{align*}
& \longrightarrow \R^j D_J\left(N,M/IM\right)\longrightarrow\R^j D_J\left(N,M/I_1M\right)\oplus\R^j D_J
\left(N,M/I_2M\right)\\ & \longrightarrow \R^j D_J\left(N,M/(I_1+I_2)M\right)\longrightarrow\R^{j+1} D_J\left(N,M/IM\right)
\longrightarrow\ldots
\end{align*}
obtained after applying the functor $D_J (N,-)$ to the natural short exact sequence
\[
\xymatrix{0\ar[r]& M/IM\ar[r]& M/I_1 M\oplus M/I_2 M\ar[r]& M/\left(I_1+I_2\right)M\ar[r]& 0.}
\]
\end{rk}


\vskip 2mm

$\bullet$ {\bf Local cohomology with respect to pairs of ideals:}
The torsion functor $\Gamma_{J,K}$ with respect to $(J,K)$ verifies
the previous requirements. Indeed, set $\widetilde{W}(J,K)$ as the
set of ideals $\mathfrak{a}$ of $A$ such that
$J^t\subseteq\mathfrak{a}+K$ for some $t\in\N$. We regard
$\widetilde{W}(J,K)$ as a poset with order given by reverse inclusion
of ideals. In this way, applying
\cite[3.2]{TakahashiYoshinoYoshizawa2009}, we obtain
\[
\Gamma_{J,K} \left(\lim_{p\in
P}V_p\right)\cong\colim_{\mathfrak{a}\in
\widetilde{W}(J,K)}\Gamma_{\mathfrak{a}} \left(\lim_{p\in
P}V_p\right).
\]
Combining this isomorphism with the fact
that $\widetilde{W}(J,K)$ is filtered we get
\[
\Gamma_{J,K} \left(\lim_{p\in P}V_p\right)\cong\colim_{\mathfrak{a}\in \widetilde{W}(J,K)}\Gamma_{\mathfrak{a}} \left(\lim_{p\in P}V_p\right)\cong\colim_{\mathfrak{a}\in \widetilde{W}(J,K)}\lim_{p\in P}\mathcal{H}_{\mathfrak{a}}^0 \left(V\right)\\
\cong\lim_{p\in P}\mathcal{H}_{J,K}^0 \left(V\right).
\]
Moreover, we also notice that, for any $\mathfrak{p}\in\spec (A)$ and
$\mathfrak{m}\in\Max (A)$, 
\[
\Gamma_{J,K}
\left(E(A/\mathfrak{p})\right)_{\mathfrak{m}}=\begin{cases}
E(A/\mathfrak{p})_{\mathfrak{m}},\text{ if
}\mathfrak{p}\in\mathbf{W}(J,K)\text{ and
}\mathfrak{p}\subseteq\mathfrak{m},\\ 0,\text{
otherwise.}\end{cases}
\]

Therefore we obtain the following spectral sequence:
\[
E_2^{i,j}=\R^i \lim_{p\in P}\mathcal{H}_{J,K}^j
\left(V\right)\xymatrix{ \ar@{=>}[r]_-{i}& }H_{J,K}^{i+j}
\left(\lim_{p\in P}V_p\right).
\]

\begin{rk}
When $I=I_1\cap I_2$ and $V=M/[*]M$ for some $A$-module $M$, this spectral sequence boils
down to the long exact sequence
\begin{align*}
& \longrightarrow H_{J,K}^j\left(M/IM\right)\longrightarrow H_{J,K}^j\left(M/I_1 M\right)\oplus H_{J,K}^j
\left(M/I_2 M\right)\longrightarrow H_{J,K}^j\left(M/(I_1+I_2)M\right)\\ & \longrightarrow H_{J,K}^{j+1}
\left(M/IM\right)\longrightarrow\ldots
\end{align*}
obtained after applying the functor $\Gamma_{J,K}$ to the natural short exact sequence
\[
\xymatrix{0\ar[r]& M/IM\ar[r]& M/I_1 M\oplus M/I_2 M\ar[r]& M/\left(I_1+I_2\right)M\ar[r]& 0.}
\]
\end{rk}


\vskip 2mm

$\bullet$ {\bf Local cohomology with respect to inverse systems of
ideals:} Let $\Phi$ be a system of ideals. We claim that
$\Gamma_{\Phi} (N,-)$ also verifies these requirements; indeed, it
is enough to point out that
\begin{align*}
& \Gamma_{\Phi} \left(N,\lim_{p\in P}V_p\right)\cong\colim_{\mathfrak{a}\in\Phi}\Hom_A \left(N/\mathfrak{a}N,\lim_{p\in P}V_p\right)\cong\colim_{\mathfrak{a}\in\Phi}\lim_{p\in P}\sHom_A \left(N/\mathfrak{a}N,V\right)\\
& \cong\lim_{p\in P}\colim_{\mathfrak{a}\in\Phi}\sHom_A
\left(N/\mathfrak{a}N,V\right)\cong\lim_{p\in P}\mathcal{H}_{\Phi}^0
\left(N,V\right).
\end{align*}
Furthermore, we have to point out that, for any $\mathfrak{p}\in\spec
(A)$ and $\mathfrak{m}\in\Max (A)$, 
\[
\Gamma_{\Phi}
\left(N,E(A/\mathfrak{p})\right)_{\mathfrak{m}}=\begin{cases}
\colim_{\mathfrak{a}\in\Phi}\Hom_{A_{\mathfrak{m}}}
\left(N_{\mathfrak{m}}/\mathfrak{a}_{\mathfrak{m}}N_{\mathfrak{m}},E(A/\mathfrak{p})_{\mathfrak{m}}\right),\text{
if }\mathfrak{p}\subseteq\mathfrak{m},\\ 0,\text{
otherwise.}\end{cases}
\]

Therefore we obtain the following spectral sequence:

\[
E_2^{i,j}=\R^i \lim_{p\in P}\mathcal{H}_{\Phi}^j \left(\lvert
N\rvert, V\right)\xymatrix{ \ar@{=>}[r]_-{i}& }
H_{\Phi}^{i+j}\left(N,\lim_{p\in P}V_p\right).
\]

\begin{rk}
When $I=I_1\cap I_2$ and $V=M/[*]M$ for some $A$-module $M$, this spectral sequence boils
down to the long exact sequence
\begin{align*}
& \longrightarrow H_{\Phi}^j\left(N,M/IM\right)\longrightarrow H_{\Phi}^j\left(N,M/I_1M\right)\oplus H_{\Phi}^j
\left(N,M/I_2 M\right)\longrightarrow H_{\Phi}^j\left(N,M/(I_1+I_2)M\right)\\ & \longrightarrow H_{\Phi}^{j+1}
\left(N,M/IM\right)\longrightarrow\ldots
\end{align*}
obtained after applying the functor $\Gamma_{\Phi} (N,-)$ to the natural short exact sequence
\[
\xymatrix{0\ar[r]& M/IM\ar[r]& M/I_1 M\oplus M/I_2 M\ar[r]& M/\left(I_1+I_2\right)M\ar[r]& 0.}
\]
\end{rk}

\vskip 2mm

So far, in all the aforementioned examples we only picked different
choices of $T$; in the following example, we also make a different
choice of poset $P$.

\vskip 2mm

$\bullet$ {\bf Local cohomology of toric face rings:} Let
$\Sigma\subseteq\R^d$ be a rational pointed fan with
$\Sigma=\Sigma_1\cup\ldots\cup\Sigma_n$ for certain subfans
$\Sigma_j\subseteq\Sigma$, let $\mathcal{M}_{\Sigma}$ be a monoidal
complex supported on $\Sigma$, and let $P$ be the poset given by all
the possible intersections of the fans $\Sigma_j$ ordered by
inclusion. Moreover, let $\K$ be a field, and let $\K
[\mathcal{M}_{\Sigma}]$ be the corresponding toric face ring, with
$\mathfrak{m}$ as graded maximal ideal. In this case, if
$T=\Gamma_{\mathfrak{m}}$ and $A=\K
[\mathcal{M}_{\Sigma}],$ then for any $p\in P$ we can
regard $\K[\mathcal{M}_{\Sigma_p}]$ as a
$\K[\mathcal{M}_{\Sigma}]$--module and therefore we obtain the next spectral
sequence:
\[
E_2^{i,j}=\R^i \lim_{p\in P}\mathcal{H}_{\mathfrak{m}}^j \left(\K [\mathcal{M}_{\Sigma_p}]\right)\xymatrix{ \ar@{=>}[r]_-{i}& }H_{\mathfrak{m}}^{i+j}\left(\lim_{p\in P}\K[\mathcal{M}_{\Sigma_p}]\right).
\]
Notice that, in this specific setting, the inverse system
$\left(\K [\mathcal{M}_{\Sigma_p}]\right)_{p\in P}$ is acyclic with
respect to the limit functor.

\begin{rk}
When $\Sigma=\Sigma_1\cup\Sigma_2$, then the above spectral sequence
boils down to the Mayer-Vietoris long exact sequence obtained in
\cite[4.3]{IchimRomer2007}.
\end{rk}

\subsection{Enhanced  structure}
This part is completely analogous with the one carried out in the
homological case; indeed, our goal is to show that the spectral
sequence established in \textbf{Theorem} \ref{la sucesion espectral para
rematar} acquires a certain additional structure provided that the source
inverse system has it as well. The
following assumption should be compared with \textbf{Assumption}
\ref{subcategoria de Serre maja}.

\begin{ass}\label{subcategoria de Serre maja: segunda parte}
Let $\mathcal{S}\subseteq\cA$ be an abelian subcategory closed under
subobjects, subquotients, and extensions such that, for any object
$G\in\Inv(P,\mathcal{S})$, $\mathcal{T} (G)\in\Inv(P,\mathcal{S})$;
moreover, if $M\in\Inv(P,\mathcal{S})$ then we suppose that there is
a long exact sequence $\xymatrix@1{0\ar[r]& M\ar[r]& E^*}$ in
$\Inv(P,\cA)$ such that:

\begin{enumerate}[(i)]

\item For any $j\geq 0$, $\R^k \mathcal{T} (E^j)=0$ for all $k\geq 1$.

\item The long exact sequence $\xymatrix@1{0\ar[r]& M\ar[r]& E^*}$ may be regarded as an exact cochain complex in $\Inv(P,\mathcal{S})$.

\item If $A$ contains a field $\K$, then we also assume that, for
any $N\in\mathcal{S},$  $\Hom_A (\K,N)$ is also an object of
$\mathcal{S}$ such that the evaluation at $1$ map
$\xymatrix@1{\Hom_A (\K,N)\ar[r]& N}$ is a natural isomorphism
in $\mathcal{S}.$

\end{enumerate}
\end{ass}
In this way, the main result of this part is the following:

\begin{teo}\label{additional structure in the cohomological case}
Suppose that $\mathcal{S}\subseteq\cA$ satisfies part (i) and (ii) of \textbf{Assumption} \ref{subcategoria
de Serre maja: segunda parte}, let $P$ be any finite poset, and let
$V\in\Inv(\widehat{P},\mathcal{S})$; moreover, suppose that there is a
natural equivalence
\[
\lim_{p\in P}\circ\mathcal{T}\cong T\circ\lim_{p\in P},
\]
and that $V$ is
acyclic with respect to the limit functor. Then, the spectral sequence obtained in \textbf{Theorem} \ref{la sucesion
espectral para rematar}
\[
E_2^{i,j}= \R^i \lim_{p\in P}\xymatrix{\R^j\mathcal{T}
\left(V\right)\ar@{=>}[r]_-{i}& }\R^{i+j}T\left(\lim_{p\in
P}V_p\right)
\]
can be naturally regarded as spectral sequence in the category
$\mathcal{S}.$
\end{teo}
Now, we want to introduce the two examples we are mostly interested
on.


\vskip 2mm

$\bullet$ {\bf Graded modules:} \label{multigraduacion a la formula
de Hochster} Let $A$ be a $\Z^n$-graded commutative Noetherian ring,
let $T$ be either the torsion functor $\Gamma_J$ or the ideal
transform $D_J$ with respect to some homogeneous ideal $J$; firstly, it
is known that, given a $\Z^n$--graded $A$--module $M$, both
$\Gamma_J (M)$ and $D_J (M)$ are 
$\Z^n$--graded $A$--modules as well \cite[Chapter 13]{BroSha}. Secondly, \textbf{Theorem}
\ref{sistemas inversos graduados inyectivos especiales} ensures that
any $M\in\Inv(P,\hbox{}^*\mathcal{A})$ can be embedded into a long
exact sequence $\xymatrix@1{0\ar[r]& M\ar[r]& E^*}$, where 
$E^k\in\Inv(P,\hbox{}^*\mathcal{A})$ are made up by
$\hbox{}^*$injectives objects, which are clearly acyclic with respect to the
functor $T$. Finally, part (iii) also holds mainly because
$\K$ is concentrated in degree zero.


\vskip 2mm

$\bullet$ {\bf Modules with Frobenius action:} \label{accion del
Frobenius a la formula de Hochster} Let $A$ be a commutative
Noetherian regular domain containing an \textbf{$F$--finite field} $\K$ of prime
characteristic $p$, i.e. $\K$ is a finite dimensional
$\K^p$-vector space. Let $T$ be either the torsion functor
$\Gamma_J$ or the ideal transform $D_J$ with respect to any ideal
$J$ of $A$, and let $A[\Theta ;F]$ be the Frobenius-skew polynomial
ring. First of all, given a left $A[\Theta ;F]$--module $M$, both
$\Gamma_J (M)$ and $D_J (M)$ also become
in $A[\Theta ;F]$--modules with the induced action of
the Frobenius on $M$ \cite[Chapter 5]{BroSha}. Second, since $A$ is regular and contains an $F$-finite field, Kunz's
Theorem ensures that $A[\Theta ;F]$ is not only a free left
$A$-module, but also a free right $A$-module; indeed, under our
assumptions there is an isomorphism $A\Theta^i \cong\Theta^i
A^{1/p^i}$ (for any $i\geq 0$), and $A^{1/p^i}$ is a flat (actually,
free) $A$-module. This implies, by \cite[Corollary 1.1
(2)]{Nastasescu1989}, that any injective $A[\Theta ;F]$-module is,
in particular, an injective $A$-module; notice that here we are also
using that $A$ is Noetherian, which is equivalent to say that any
arbitrary direct sum of injective $A$-modules is also injective.

However, notice that, as in the case of $F$--modules, part (iii)
of \textbf{Assumption} \ref{subcategoria
de Serre maja: segunda parte} work when $\K=\mathbb{F}_p$ and
we restrict our attention to the category of
$\mathbb{F}_p[\Theta ;F]$--modules.

\subsection{Degeneration of cohomological spectral sequences}
Now we are in position to provide sufficient conditions to ensure
the degeneration at the $E_2$-sheet of the spectral sequence
obtained in \textbf{Theorem} \ref{la sucesion espectral para rematar};
namely:

\begin{teo}\label{sucesion espectral y colapso todo en uno}
Let $\K$ be any field, let $A$ be a commutative Noetherian ring
containing $\K$, and let $P$ be any finite poset. We further assume
that the inverse system $V$ is acyclic with respect to the limit on
$P$, that there is a natural equivalence of functors $\lim_{p\in
P}\circ\mathcal{T}\cong T\circ\lim_{p\in P}$, that, for any $p\in
P$, $\R^j T \left(V_p\right)=0$ up to a single value of $j$ (namely,
$d_p$), and that for any $p\neq q$, $\Hom_A \left(\R^{d_p} T
\left(V_p\right),\R^{d_q} T \left(V_q\right)\right)=0$. Then, there are
canonical isomorphisms of $A$--modules
\[
\R^i \lim_{p\in P} (\R^j\mathcal{T}(V))\cong\bigoplus_{j=d_q}\Hom_{\K}\left(\widetilde{H}_{i-1}
((q,1_{\widehat{P}}); \K), \R^{d_q}T(V_q)\right)\cong\bigoplus_{j=d_q}
\R^{d_q}T (V_q)^{\oplus M_{i,q}},
\]
where $M_{i,q}:=\dim_{\K} (\widetilde{H}_{i-1} ((q,1_{\widehat{P}});
\K));$ moreover, there exists a first quadrant spectral sequence of the form
\[
E_2^{i,j}=\bigoplus_{j=d_q} \R^{d_q} T \left(V_q\right)^{\oplus
M_{i,q}}\xymatrix{ \ar@{=>}[r]_-{i}& }\R^{i+j}T \left(\lim_{p\in
P}V_p\right),
\]
which degenerates at the $E_2$-sheet.
\end{teo}

\begin{proof}
Since $\R^j T (V_p)=0$ up to a single value of $j$ and $\Hom_A
\left(\R^{d_p} T \left(V_p\right),\R^{d_q} T
\left(V_q\right)\right)=0$, it follows that there is a canonical
isomorphism of inverse systems of $A$--modules
\[
\R^j\mathcal{T}(V)\cong\bigoplus_{j=d_q} (\R^{d_q}T (V_q))_q.
\]
Indeed, the inverse system $\R^j\mathcal{T}(V)$ is the one given, by
the very definition of the functor $\mathcal{T}$, by $(\R^j T
(V_p))_{p\in P}=(\R^{d_p} T (V_p))_{p\in P}$ $(j=d_p)$, where the
only non-zero structural morphisms are identities; these facts imply
the decomposition of $\R^j\mathcal{T}(V)$ into the direct sum above.

Now, fix $i\in\N$. Applying the $i$th
right derived functor of the inverse limit over $P$ to the above decomposition, we get the
following canonical isomorphism of $A$--modules:
\[
\R^i \lim_{\substack{\longleftarrow\\ p\in
P}}\R^j\mathcal{T}(V)\cong\bigoplus_{j=d_q}\R^i
\lim_{\substack{\longleftarrow\\ p\in P}}\left(\R^{d_q}T
(V_q)\right)_q .
\]
Moreover, the cohomological analogue of \textbf{Lemma} \ref{homologia del skyscraper haz}
implies that there is a canonical isomorphism of $A$--modules:
\[
\R^i \lim_{\substack{\longleftarrow\\ p\in
P}}\R^j\mathcal{T}(V)\cong\bigoplus_{j=d_q} \widetilde{H}^{i-1}
((q,1_{\widehat{P}}); \R^{d_q}T (V_q)).
\]
Now, since the map $\widetilde{H}^{i-1}
((q,1_{\widehat{P}}); \R^{d_q}T (V_q))
\rightarrow\Hom_{\K}\left(\widetilde{H}_{i-1}
((q,1_{\widehat{P}}); \K),\R^{d_q}T (V_q)\right)$ is a natural isomorphism of
$A$--modules, one obtains the following natural isomorphism
of $A$--modules:
\[
\R^i \lim_{\substack{\longleftarrow\\ p\in
P}}\R^j\mathcal{T}(V)\cong\bigoplus_{j=d_q}\Hom_{\K}\left(\widetilde{H}_{i-1}
((q,1_{\widehat{P}}); \K), \R^{d_q}T(V_q)\right).
\]
Now, set $M_{i,q}:=\dim_{\K} (\widetilde{H}_{i-1} ((q,1_{\widehat{P}});
\K)),$ so $\widetilde{H}_{i-1}
((q,1_{\widehat{P}}); \K)\cong\K^{\oplus M_{i,q}}.$ As the evaluation at $1$ map
\[
\xymatrix{\Hom_{\K}(\K , \R^{d_q} T (V_q))\ar[r]& \R^{d_q} T  (V_q)}
\]
is a canonical isomorphism of $A$-modules one obtains a natural isomorphism
\[
\Hom_{\K}\left(\widetilde{H}_{i-1} ((q,1_{\widehat{P}}); \K),
\R^{d_q}T (V_q)\right)\cong\R^{d_q} T (V_q)^{\oplus M_{i,q}}
\]
of $A$--modules and therefore we finally have a natural $A$--module isomorphism
\[
\R^i \lim_{\substack{\longleftarrow\\ p\in
P}}\R^j\mathcal{T}(V)\cong\bigoplus_{j=d_q} \R^{d_q}T (V_q)^{\oplus
M_{i,q}},
\]
as claimed.
\end{proof}

Moreover, we also want to state the cohomological analogue of
\textbf{Corollary} \ref{habemus filtracion}; in this case, the details are
left to the interested reader.

\begin{cor}\label{Josep me hace repetirme}
Under the assumptions of \textbf{Theorem} \ref{sucesion espectral y colapso
todo en uno}, for each $0\leq r\leq\cdim (T)$ there is an
increasing, finite filtration $\{H_k^r\}$ of $\R^r T(\lim_{p\in
P}V_p)$ by $A$-modules such that, for any $k\geq 0,$
\[
H_k^r/H_{k-1}^r\cong\bigoplus_{\{q\in P\ \mid\ r-k=d_q\}}
\R^{d_q}T \left(V_q\right)^{\oplus M_{k,q}},
\]
where $M_{k,q}=\dim_{\K} \widetilde{H}^{k-d_q-1} ((q,1_{\widehat{P}});
\K),$ and we follow the convention that $H_{-1}^r =0.$
\end{cor}

We conclude this part writing down the enriched version (see \textbf{Assumption}
\ref{subcategoria de Serre maja: segunda parte}) of \textbf{Corollary}
\ref{Josep me hace repetirme}; namely:

\begin{teo}\label{Josep me hace repetirme enriquecido}
Let $\K$ be any field, let $A$ be a commutative Noetherian ring
containing $\K$, let $\mathcal{S}$ be a subcategory of the
category of $A$--modules satisfying \textbf{Assumption} \ref{subcategoria de Serre maja: segunda parte}, and let $P$ be any finite poset. We further assume
that the inverse system $V\in\Inv (P,\mathcal{S})$ is acyclic with respect to the limit on
$P$, that there is a natural equivalence of functors $\lim_{p\in
P}\circ\mathcal{T}\cong T\circ\lim_{p\in P}$, that, for any $p\in
P$, $\R^j T \left(V_p\right)=0$ up to a single value of $j$ (namely,
$d_p$), and that for any $p\neq q$, $\Hom_{\mathcal{S}} \left(\R^{d_p} T
\left(V_p\right),\R^{d_q} T \left(V_q\right)\right)=0.$ Then, for each $0\leq r\leq\cdim (T)$ there is an increasing, finite filtration $\{H_k^r\}$ of
$\R^r T(\lim_{p\in
P}V_p)$ by objects of $\mathcal{S}$ such that, for any $k\geq 0,$
\[
H_k^r/H_{k-1}^r\cong\bigoplus_{\{q\in P\ \mid\ r-k=d_q\}}
\R^{d_q}T \left(V_q\right)^{\oplus M_{k,q}},
\]
where we follow the convention that $H_{-1}^r =0,$ and all these isomorphisms are isomorphisms in the category $\mathcal{S}.$
\end{teo}

\subsubsection{A spectral sequence of local cohomology modules
supported at the maximal ideal}\label{problemas de extension a saco}

The aim of this part is to single out the spectral sequence produced
in Theorem \ref{sucesion espectral y colapso todo en uno} in the particular
case where $T=\Gamma_{\mathfrak{m}}$ and $P$ is the poset associated to a
decomposition of an ideal $I\subseteq A$ (see \textbf{Theorem}
\ref{sucesion espectral y colapso todo en uno de cohomologia local}); moreover, we
also provide specific examples where the assumptions required in the statement
of this result are fulfilled. We also want to point out that, later in this paper
(see \textbf{Sections \ref{section of Hochster decompositions} and
\ref{section of Grabe formulae}}), we focus on this spectral sequence in order to, on the one hand, obtain
some \textbf{Hochster type decompositions} of local cohomology modules and, on
the other hand, to study the extension problems attached to the filtration
produced by its degeneration.

We start with the following auxiliary result.

\begin{lm}\label{playing with attached primes}
Let $A$ be any commutative Noetherian ring, and let $I_p,I_q$ be two
ideals of $A$ contained in a fixed maximal one (say, $\mathfrak{m}$)
such that $I_q\not\subseteq\mathfrak{p}$ for any
prime ideal $\mathfrak{p}\supseteq I_p$ such that
$\dim (A/\mathfrak{p})=d_p.$ Then, $\Hom_A
\left(H_{\mathfrak{m}}^{d_p}\left(A/I_p\right),H_{\mathfrak{m}}^{d_q}\left(A/I_q\right)\right)=0.$
\end{lm}

\begin{proof}
First of all, in order to simplify notation, set
$H_p:=H_{\mathfrak{m}}^{d_p}\left(A/I_p\right)$ and
$H_q:=H_{\mathfrak{m}}^{d_q}\left(A/I_q\right)$; we assume,
to get a contradiction, that there is a $0\neq\psi\in\Hom_A
\left(H_p,H_q\right)$. Write $\Att_A
\left(H_p\right)=\{\mathfrak{p}_1,\ldots, \mathfrak{p}_s\},$ where $\Att$ denotes the set of
attached primes as defined, for instance, in \cite[7.2]{BroSha}; in
this way, we get the following commutative square:
\[
\xymatrix{H_p\ar@{->>}[d]\ar[r]^-{\psi}& H_q\\ Q\ar[r]_-{\sim}& \im
(\psi).\ar@{^{(}->}[u]}
\]
Here, $Q:=H_p/\ker (\psi)$ and the bottom isomorphism is the one
provided by the First Isomorphism Theorem. Moreover, $\Att_A
(Q)\subseteq\Att_A (H_p)$ (indeed, this fact follows from
\cite[7.2.6]{BroSha}) and $\Att_A \left(H_p\right)=\{\mathfrak{p}_1,\ldots, \mathfrak{p}_s\}$
; since $Q\neq 0$ because we are assuming that
$\psi\neq 0$, we can assume, without loss of generality, that
there is $1\leq r\leq s$ such that $\Att_A (Q)=\{\mathfrak{p}_1,\ldots,
\mathfrak{p}_r\}$. Moreover, as
$Q\cong\im (\psi)$, it implies that $\Att_A (\im
(\psi))=\{\mathfrak{p}_1,\ldots,
\mathfrak{p}_r\}$. On the other hand, since $\im
(\psi)\subseteq H_q$, it is clear that
\[
\sqrt{\left(0:_A H_q\right)}\subseteq\sqrt{\left(0:_A \im
(\psi)\right)};
\]
regardless, combining \cite[7.2.11]{BroSha} and the foregoing facts
it follows that
\[
\sqrt{I_q}=\sqrt{\left(0:_A H_q\right)}\subseteq\sqrt{\left(0:_A
\im (\psi)\right)}=\mathfrak{p}_1\cap\ldots\cap\mathfrak{p}_r.
\]
But this contradicts our assumption that
$I_q\not\subseteq\mathfrak{p}$ for any
prime ideal $\mathfrak{p}\supseteq I_p$ such that
$\dim (A/\mathfrak{p})=d_p,$ whence $\psi$ must be zero;
the proof is therefore completed.
\end{proof}
Now, we are ready to establish the main result of this part; the reader
will easily note that, in the statement, we do not need to require any vanishing of $\Hom$'s
because of \textbf{Lemma} \ref{playing with attached primes}.


\begin{teo}\label{sucesion espectral y colapso todo en uno de cohomologia local}
Let $\K$ be any field, let $A$ be any commutative Noetherian ring
containing $\K$, let $P$ be the poset given in \textbf{Example} \ref{el poset
de Mayer-Vietoris}; suppose that all of them are contained in a
certain maximal ideal of $A$ (namely, $\mathfrak{m}$). We further
assume that, for any $p\in P$, $A/I_p$ is a Cohen-Macaulay ring and that, for
any $p\neq q,$ $I_q$ is not contained in any minimal prime
of $I_p$ (this holds, for instance, if one supposes that all
the $A/I_p$'s are Cohen-Macaulay domains), and that the inverse system
$A/[*]$ is acyclic with respect to the limit. Then, there exists a first quadrant spectral sequence of the form:
\[
E_2^{i,j}=\bigoplus_{j=d_q} H_{\mathfrak{m}}^{d_q}
\left(A/I_q\right)^{\oplus M_{i,q}}\xymatrix{ \ar@{=>}[r]_-{i}& }
H_{\mathfrak{m}}^{i+j} \left(\lim_{p\in P} A/I_p\right),
\]
where $M_{i,q}=\dim_{\K} \widetilde{H}^{i-d_q-1} ((q,1_{\widehat{P}});
\K).$ Moreover, this spectral sequence degenerates at the $E_2$-sheet
and, for each $0\leq r\leq\dim (A)$ there is an increasing, finite filtration $\{H_k^r\}$ of
$H_{\mathfrak{m}}^r (\lim_{p\in P} A/I_p)$ by $A$--modules such that, for any $k\geq 0,$
\[
H_k^r/H_{k-1}^r\cong\bigoplus_{\{q\in P\ \mid\ r-k=d_q\}}
H_{\mathfrak{m}}^{d_q} \left(A/I_q\right)^{\oplus M_{k,q}},
\]
where we follow the convention that $H_{-1}^r =0.$
\end{teo}


Some specific situations where the assumptions of \textbf{Theorem}
\ref{sucesion espectral y colapso todo en uno de cohomologia local}
are fulfilled are the following:

\vskip 2mm

$\bullet$ {\bf Squarefree monomial ideals:} When $I$ is a squarefree monomial ideal it admits a minimal primary decomposition in terms of prime ideals
generated by variables, which is a family of prime ideals closed
under sum. Therefore, under these assumptions $P$ is made up by
prime ideals generated by linear forms, which also verify all the
assumptions of \textbf{Theorem} \ref{sucesion espectral y colapso todo en uno
de cohomologia local}.

\vskip 2mm

$\bullet$ {\bf Central arrangements of linear varieties:} More generally, let $I$ be the defining ideal
of a central arrangement of linear varieties over any field $\K$; in this case, we also have that
$I$ admits a minimal primary decomposition in terms of prime ideals generated by linear forms, which
is a family of prime ideals closed under sum. Therefore, under these assumptions $P$ is made up by
prime ideals generated by linear forms, which also verify all the
assumptions of \textbf{Theorem} \ref{sucesion espectral y colapso todo en uno
de cohomologia local} up to the fact that, in general, we can not
guarantee that $A/[*]$ is acyclic with respect to the limit
functor.

\vskip 2mm

$\bullet$ {\bf Monomial ideals in regular sequences:} Let $A$ be a
commutative Noetherian ring containing a field $\K$, and let
$y_1,\ldots ,y_n$ be an $A$-regular sequence contained in the
Jacobson radical of $A$. Moreover, let $I:=J \K [Y_1,\ldots ,Y_n]$,
where $\xymatrix@1{\K [Y_1,\ldots ,Y_n]\ar[r]^-{\psi}& A}$ is the
map of $\K$-algebras which sends each indeterminate $Y_i$ to $y_i$,
and $J$ is a monomial ideal in the usual sense. Thus,
since $\psi$ is flat \cite[2.1]{SabzrouTousi2008}, it follows
that \textbf{Theorem} \ref{sucesion espectral y colapso todo en uno de
cohomologia local} can also be applied in this setting; indeed, it is known
that $I$ admits a unique primary decomposition in terms of
monomial (not necessarily prime) ideals in $A,$ and that this decomposition is determined
by the corresponding decomposition of $J$ inside
$\K [Y_1,\ldots ,Y_n]$ \cite[Theorem 4.9 and Corollary 4.12]{HeMiRaShah1997}. Moreover, flatness
of $\psi$ implies, because local cohomology commutes with
flat base change, that, if one denotes by $I_p$ one of the sums
of the primary components of $I,$ then one has that $A/I_p$ is
Cohen--Macaulay. Finally, the condition that, for any $p\neq q,$
$I_q$ is not contained in any minimal prime
of $I_p$ is immediate from the fact that the decomposition of $I$
is determined by the corresponding decomposition of $J$ inside
$\K [Y_1,\ldots ,Y_n].$

\vskip 2mm

$\bullet$ {\bf Monomial ideals in semigroup rings:} Let $Q$ be an
affine semigroup \cite[7.4]{MillerSturmfels2005}, let $\K$ be
any field, and suppose that the semigroup ring $A=\K [Q]$ is either
normal, or simplicial and Cohen-Macaulay. Moreover, let
$I\subseteq\K [Q]$ be a squarefree monomial ideal
\cite[7.9]{MillerSturmfels2005}. We claim that \textbf{Theorem}
\ref{sucesion espectral y colapso todo en uno de cohomologia local}
can also be applied in this case; indeed, $I$ admits a minimal
primary decomposition in terms of monomial prime ideals
\cite[7.13]{MillerSturmfels2005}. In addition, for any monomial
prime ideal $\mathfrak{p}$ it is known (see
\cite[6.3.5]{BrunsHerzog1993}, \cite[Paragraph below Remark
3.4]{Yanagawa2001posets} in the normal case, and
\cite[2.4]{Yanagawa2008} in the remainder case) that
$A/\mathfrak{p}$ is a Cohen-Macaulay ring. Therefore, the poset $P$
is made up by Cohen-Macaulay monomial prime ideals, which is just
what we need to check.

\vskip 2mm

$\bullet$ {\bf Toric face rings:} Let $\Sigma\subseteq\R^d$ be a
rational pointed fan, let $\mathcal{M}_{\Sigma}$ be a Cohen-Macaulay
monoidal complex supported on $\Sigma$, let $\K$ be a field, and let
$A=\K[\mathcal{M}_{\Sigma}]$ be the corresponding toric face ring
(see \cite[page 251]{IchimRomer2007} for unexplained terminology).
We claim that \textbf{Theorem} \ref{sucesion espectral y colapso todo en uno
de cohomologia local} can also be applied in this case; indeed,
since the sum of $\Z^d$-graded prime ideals of $A$ is again a
$\Z^d$-graded prime ideal \cite[Lemma 2.3
(i)]{IchimRomer2007}, and that there is a bijection between the set
of non-empty cones of $\Sigma$ and the set of $\Z^d$-graded prime
ideals of $A$ \cite[Lemma 2.1]{IchimRomer2007}  it follows,
since every $\K [M_C]$ is Cohen-Macaulay for any cone $C\in\Sigma$,
that the poset $P$, made up by all the possible $\Z^d$-graded prime
ideals of $A$, is closed under sum and satisties that, for any $p\in
P$, $A/I_p\cong\K [M_C]$ is a Cohen-Macaulay domain, which is just
what we need to check.


\section{Homological spectral sequences associated to inverse systems}\label{homological_typeII}

The formalism given in Section \ref{construccion de libros} can be used to provide
several examples of homological spectral sequences of local cohomology modules.
However,  it can NOT be applied in the case that $T_{[*]}=\Hom_A (A/[*],-)$
so we can NOT construct a spectral sequence of the form
\[
E_2^{-i,j}=\Lf_i \colim_{p\in P} \Ext_A^j (A/I_p,M)\xymatrix{\hbox{}\ar@{=>}[r]_-{i}& \Ext_A^{j-i} (A/I,M)}
\]
as it was wrongly stated in \cite[1.4(iii)]{AlvarezGarciaZarzuela2003}.  In this section, following a similar approach to
the one given in Section \ref{de aqui sacaremos la formula de Hochster}, we will provide
the right formalism to obtain a spectral sequence for $\Ext$ modules.

\vskip 2mm

First of all we will show that the key \textbf{Lemma} \ref{lema clave en el caso homologico} is no longer
true for the functor  $T_{[*]}=\Hom_A (A/[*],-)$ just because it does not satisfies the conditions
of \textbf{Setup} \ref{las hipotesis del caso homologico}.

\begin{carg*}
Let $\K$ be any field, set $A:=\K [\![x,y,z]\!]$ and
\[
I:=\langle xy,xz,yz\rangle=\langle x,y\rangle\cap\langle x,z\rangle\cap\langle y,z\rangle =I_1\cap I_2\cap I_3.
\]
Moreover, set $E:=E_A (\K)$ as a choice of injective hull of $\K$ over $A$; our goal in this example is to compute
explicitly the following augmented chain complex:
\begin{equation}\label{complejo inicial}
\xymatrix{\roos_* (\Hom_A (A/[*],E))\ar[r]& \Hom_A (A/I,E)\ar[r]& 0.}
\end{equation}
In addition, since it is noteworthy that, for any ideal $J$ of $A$, there is a canonical isomorphism of $A$-modules
$\Hom_A (A/J,E)\cong (0:_E J)$ it turns out that \eqref{complejo inicial} is canonically isomorphic to the next
augmented chain complex $\xymatrix@1{\roos_* ((0:_E [*]))\ar[r]& (0:_E I)\ar[r]& 0,}$ which, in this case, is nothing but
\begin{equation}\label{el complejo dual}
\xymatrix{0\ar[r]& \roos_1 ((0:_E [*]))\ar[r]^-{d_1}& \roos_0 ((0:_E [*]))\ar[r]^-{d_0}& (0:_E I)\ar[r]& 0.}
\end{equation}
So, our aim is to calculate explicitly \eqref{el complejo dual}. Firstly, we determine its spots: on the one hand,
its $0$th spot is $\roos_0 ((0:_E [*]))=(0:_E \mathfrak{m})\oplus (0:_E I_3)\oplus (0:_E I_2)\oplus (0:_E I_1).$
On the other hand, its $1$th spot is
\[
\roos_1 ((0:_E [*]))=(0:_E \mathfrak{m})\oplus (0:_E \mathfrak{m})\oplus (0:_E \mathfrak{m}).
\]
Now, we have to compute its differentials; namely, $d_0$ and $d_1$.

\begin{enumerate}[(i)]

\item The $0$th differential turns out to be
\begin{align*}
& (0:_E \mathfrak{m})\oplus (0:_E I_3)\oplus (0:_E I_2)\oplus (0:_E I_1)\stackrel{d_0}{\longrightarrow} (0:_E I)\\
& (a,a_1,a_2,a_3)\longmapsto -a+a_3-a_2+a_1.
\end{align*}

\item The first differential $d_1$ is given by
\begin{align*}
& (0:_E \mathfrak{m})\oplus (0:_E \mathfrak{m})\oplus (0:_E \mathfrak{m})\stackrel{d_1}{\longrightarrow}
(0:_E \mathfrak{m})\oplus (0:_E I_3)\oplus (0:_E I_2)\oplus (0:_E I_1)\\
& (b_3,b_2,b_1)\longmapsto (0,b_2-b_3,b_1-b_3,b_1-b_2).
\end{align*}

\end{enumerate}
Summing up, the augmented chain complex \eqref{el complejo dual} is the one induced by the augmented chain complex
\[
\xymatrix{0\ar[r]& E^{\oplus 3}\ar[r]^-{A_1}\ar[r]& E^{\oplus 4}\ar[r]^-{A_0}& E\ar[r]& 0,}
\]
where $A_0:=\begin{pmatrix} -1& 1& -1& 1\end{pmatrix}$ and
\[
A_1:=\begin{pmatrix}[r] 0& 0& 0\\ -1& 1& 0\\ -1& 0& 1\\ 0& -1& 1\end{pmatrix}.
\]
Applying Matlis duality $(-)^{\vee}$, one obtains the following coaugmented cochain complex:
\[
\xymatrix{0\ar[r]& A/I\ar[r]& A/\mathfrak{m}\oplus A/I_3 \oplus A/I_2 \oplus A/I_1\ar[r]& A/\mathfrak{m}\oplus
A/\mathfrak{m}\oplus A/\mathfrak{m}\ar[r]& 0.}
\]
As the reader can easily note, this coaugmented cochain complex is the induced one given by the next complex:
$\xymatrix@1{0\ar[r]& A\ar[r]^-{A_0^t}\ar[r]& A^{\oplus 4}\ar[r]^-{A_1^t}& A^{\oplus 3}\ar[r]& 0.}$ Regardless,
neither the previous lifted complex nor the induced one are exact. Indeed, we have checked that the lifted complex
is not exact using Macaulay2 \cite{M2}. Of course, the reader might think that perhaps the lifted complex is
not exact, but the induced complex after taking equivalence classes is so. Unfortunately, this is not the case,
because $(\cls (1),\cls (x),\cls (y), \cls(z))$ is a member of the kernel of the map given by $A_1^t$ which does
not belong to the image of the map given by $A_0^t$.

\end{carg*}

\subsection{Construction of homological spectral sequences}
The setup we need for the appropriate  spectral sequence is the following:

\begin{cons}\label{otra construccion mas de lo mismo}
Let $\xymatrix@1{\cA\ar[r]^-{T}& \cA}$ be a contravariant, left exact functor, and let $P$ be any finite poset.
Carrying over $T$, we produce a new functor (namely, $\mathcal{T}$) in the following manner:
\begin{align*}
& \xymatrix{\Inv (P,\cA)\ar[r]^-{\mathcal{T}}& \Dir (P,\cA)}\\ & G=(G_p)_{p\in P}\longmapsto
\left(T\left(G_p\right)\right)_{p\in P}.
\end{align*}
Moreover, we also assume that $T$ commutes with finite direct sums and that $T(A)=Z$ for some $A$-module $Z$;
in particular, one has that $\colim_{p\in P} \mathcal{T} \left(A_{\leq q}\right)=Z$, where $q\in P$ (notice that $Z$
only depends on $T$, but not on $q$).
\end{cons}

The following lemma will provide the abutment of the spectral sequence we want to construct.

\begin{lm}\label{lema clave en este arreglo}
Let $G$ be a projective object of $\Inv (P,\cA)$ of the form
\[
G=\bigoplus_{j\in J} A_{\leq q_j},
\]
where $J$ is a finite index set and $q_j\in P$ (see \textbf{Theorem} \ref{proyectivos que interesan}). Then, the augmented
chain complex
\[
\roos_* \left(\mathcal{T}\left(G\right)\right)\longrightarrow\left(\colim_{p\in P}\circ\mathcal{T}\right)
\left(G\right)\longrightarrow 0
\]
is exact.
\end{lm}

\begin{proof}
Since $\roos_*$, $\mathcal{T}$ and the colimit functor commutes with finite direct sums we may suppose, without
loss of generality, that $G=A_{\leq q}$ for some fixed $q\in P$; regardless, in this case, our augmented chain
complex is exactly the one for computing the simplicial homology of the interval $[q,1_{\widehat{P}})$ (indeed,
notice that we have to take this interval because $\mathcal{T}$ is contravariant) with coefficients in $Z$. But
$[q,1_{\widehat{P}})$ is contractible by \textbf{Lemma} \ref{contractibilidad de abiertos en la topologia de Alexandrov};
the proof is therefore completed.
\end{proof}
The following result provides the announced spectral sequence; since its proof is almost verbatim to the one given
in \textbf{Theorem} \ref{una primera construccion para calentar}, it will be skipped.

\begin{teo}\label{otra construccion Josep es cansino}
Given an inverse system $V\in\Inv (\widehat{P},\cA)$ that is acyclic with
respect to the limit functor, there is a first quadrant
spectral sequence $$E_2^{-i,j}=\Lf_i \colim_{p\in P}\R^j \mathcal{T} \left(V\right)\xymatrix@1{\ar@{=>}[r]_-{i}& }H^{j-i}
\left(\colim_{p\in P}\circ\mathcal{T}\right)\left(V\right),$$ where the abutment denotes the cohomology of the cochain
complex
\[
0\longrightarrow\colim_{p\in P} \mathcal{T}\left(V\right)\longrightarrow\colim_{p\in P} \mathcal{T}\left(F_0\right)
\longrightarrow\colim_{p\in P} \mathcal{T}\left(F_1\right)\longrightarrow\ldots
\]
and $\xymatrix@1{\ldots\ar[r]& F_1\ar[r]& F_0\ar[r]& V\ar[r]& 0}$ denotes a projective resolution of $V$ in
$\Inv (P,\cA)$, where any $F_i$ is made up by direct summands of the form $A_{\leq p}$ ($p\in P$).
\end{teo}

\subsection{Examples}
Of course, the example we are mainly interested on is the following
one:

\vskip 2mm

$\bullet$ {\bf Contravariant Hom:}  Let $N$ be an arbitrary $A$-module. Then, the functor $\Hom_A
\left(-,N\right)$ is clearly left exact, contravariant, and commutes
with finite direct sums; whence $\Hom_A \left(-,N\right)$ can be
regarded as a particular case of \textbf{Setup} \ref{otra construccion
mas de lo mismo}. Furthermore, given an inverse system $V\in\Inv (\widehat{P},\cA)$ that is acyclic with
respect to the limit functor we have a canonical
isomorphism $$\Hom_A \left(\lim_{p\in P} V_p, N\right)\cong\colim_{p\in P} \sHom_A \left(V,\lvert N\rvert\right),$$
where $\lvert N\rvert$ is the constant inverse system given by $N$ with identities as structural maps.
Therefore, we obtain the spectral sequence:

\[
E_2^{-i,j}=\Lf_i \colim_{p\in P}\sExt_A^j \left(V,\lvert N\rvert\right)\xymatrix{\ar@{=>}[r]_-{i}& }
\Ext_A^{j-i}\left(\lim_{p\in P}V_p ,N\right).
\]





\begin{rk}
When $I=I_1\cap I_2$ and $V=M/[*]M$ for some $A$-module $M$, this spectral sequence boils down to the long exact sequence
\begin{align*}
& \longrightarrow\Ext_A^i \left(M/\left(I_1+I_2\right)M,N\right)\longrightarrow\Ext_A^i \left(M/I_1 M,N\right)\oplus
\Ext_A^i \left(M/I_2 M,N\right)\longrightarrow\Ext_A^i \left(M/IM,N\right)\\ & \longrightarrow\Ext_A^{i+1}
\left(M/\left(I_1+I_2\right)M,N\right)\longrightarrow\ldots
\end{align*}
obtained after applying the functor $\Hom_A \left(-,N\right)$ to the natural short exact sequence:
\[
\xymatrix{0\ar[r]& M/IM\ar[r]& M/I_1 M \oplus M/I_2 M\ar[r]& M/\left(I_1+I_2\right)M\ar[r]& 0.}
\]
\end{rk}

\subsection{Degeneration of homological spectral sequences}
Carrying out the same strategy used to produce \textbf{Theorem} \ref{sucesion espectral y colapso todo en uno} and
Corollary \ref{Josep me hace repetirme}, we obtain the following pair of results, which involve the spectral
sequence constructed in \textbf{Theorem} \ref{otra construccion Josep es cansino}; the details are left to the interested reader.

\begin{teo}\label{sucesion espectral y colapso todo en uno: segunda parte}
Let $\K$ be any field, let $A$ be a commutative Noetherian ring containing $\K$, let $P$ be any finite poset, and let
$T$ and $\mathcal{T}$ be as in \textbf{Setup} \ref{otra construccion mas de lo mismo}. We further assume that the inverse
system $V$ is acyclic with respect to the limit on $P$, that there is a natural equivalence of functors
\[
\colim_{p\in P}\circ\mathcal{T}\cong T\circ\lim_{p\in P},
\]
that for any $p\in P$, $\R^j T \left(V_p\right)=0$ up to a single value of $j$ (namely, $h_p$), and that for any $p\neq q$,
$\Hom_A \left(\R^{h_p} T \left(V_p\right),\R^{h_q} T \left(V_q\right)\right)=0$. Then, there exists a third quadrant
spectral sequence of the form:
\[
E_2^{-i,j}=\bigoplus_{j=h_q} \R^{h_q} T \left(V_q\right)^{\oplus m_{i,q}}\xymatrix{ \ar@{=>}[r]_-{i}& }\R^{j-i}T \left(\lim_{p\in P}V_p\right),
\]
where $m_{i,q}:=\dim_{\K} \left(\widetilde{H}_{i-h_q-1} \left(\left(q,1_{\widehat{P}}\right);\K\right)\right)$. Moreover,
this spectral sequence degenerates at the $E_2$-sheet.
\end{teo}

\begin{cor}\label{Josep me hace repetirme: segunda parte}
Under the assumptions of \textbf{Theorem} \ref{sucesion espectral y colapso todo en uno: segunda parte}, for each
$0\leq r\leq\cdim (T)$ there is an increasing, finite filtration $\{H_k^r\}$ of $\R^r T(\lim_{p\in P}V_p)$ by
$A$-modules such that, for any $k\geq 0,$
\[
H_k^r/H_{k-1}^r\cong\bigoplus_{\{q\in P\ \mid\ r-k=h_q\}} \R^{h_q} T \left(V_q\right)^{\oplus m_{k,q}},
\]
where we follow the convention that $H_{-1}^r =0.$
\end{cor}
We want to single out here that we plan to use \textbf{Corollary} \ref{Josep me hace repetirme: segunda parte}
to provide a certain decomposition of the so--called
deficiency modules (see \textbf{Example} \ref{deficiency modules decomposition}).


\section{Some Hochster type decompositions}\label{section of Hochster decompositions}

A celebrated result of Hochster provides a decomposition of the local cohomology modules
$H^r_{\mathfrak{m}}(A/I)$ in terms of $H^r_{\mathfrak{m}}(A/I_p)$ in the case that
$I$ is a squarefree monomial ideal in the polynomial ring $A=\K[x_1,\cdots,x_d]$ over a field $\K$.
This formula was generalized, on the one hand, to arbitrary monomial
ideals by Takayama in \cite[Theorem 1]{Takayama2005} (see also
\cite[Corollary 2.3]{BrunRomer2008}) and, on the other hand, to toric face rings by Brun, Bruns and R\"omer in \cite{BrunBrunsRomer2007}.

Moreover, Musta{\c{t}}{\v{a}} \cite[Theorem 2.1 and Corollary 2.2]{Mustata2000symbolic} and Terai \cite{Teraiunpublished} (see also \cite[Corollary 13.16]{MillerSturmfels2005}) provide a decomposition of the local cohomology modules
$H^r_{I}(A)$ in terms of $H^r_{I_p}(A)$ in the case that
$I$ is a squarefree monomial ideal in the polynomial ring $A=\K[x_1,\cdots,x_d]$ over a field $\K$.

The goal of this section is, on the one hand, to provide a decomposition of local
cohomology modules such that, in the case that $I=I_{\Delta}$ is a
Stanley-Reisner ideal, is just the classical Hochster's
decomposition of the local cohomology of a Stanley-Reisner ring
\cite[13.13]{MillerSturmfels2005}; on the other hand, we also produce a
decomposition of local cohomology modules such that, in the case that $I=I_{\Delta}$ is a
Stanley-Reisner ideal, is just Musta{\c{t}}{\v{a}}--Terai's formula. In particular, we obtain a
decomposition of $H_{I_{\Sigma}}^r (A)$, where $I_{\Sigma}$ is the defining ideal
of certain toric face rings in the polynomial ring $A=\K[x_1,\cdots,x_d]$ over a field $\K$
\textbf{of prime characteristic} (see the Examples after \textbf{Theorem}
\ref{habemus filtracion de cohomologia local} to realize
why we have to restrict ourselves to this situation); to the best of our knowledge,
this is the first time this formula appears in the literature.
\vskip 2mm

First of all, we want to start with our Hochster type decompositions; indeed, the first
main result of this section is the following:

Consider the cohomological spectral sequences constructed in Section
\ref{de aqui sacaremos la formula de Hochster}. In \textbf{Theorem}
\ref{sucesion espectral y colapso todo en uno} we produced sufficient
conditions for the spectral sequence to degenerate at the
$E_2$-page. In this case we obtained:
\[
E_2^{i,j}=\bigoplus_{j=d_q} \R^{d_q} T \left(V_q\right)^{\oplus
M_{i,q}}\xymatrix{ \ar@{=>}[r]_-{i}& }\R^{i+j}T \left(\lim_{p\in
P}V_p\right),
\]
where $M_{i,q}=\dim_{\K} \widetilde{H}^{i-d_q-1} ((q,1_{\widehat{P}});
\K).$ Moreover, according to Corollary \ref{Josep me hace repetirme} we have a
collection of short exact sequences (for some $b\in\N$)
\begin{align*}
& \xymatrix{0\ar[r]& H_0^r\ar[r]& H_1^r\ar[r]& H_1^r/H_0^r\ar[r]& 0}\\
& \xymatrix{0\ar[r]& H_1^r\ar[r]& H_2^r\ar[r]& H_2^r/H_1^r\ar[r]& 0}\\
& \xymatrix{\vdots& \vdots& &\vdots& &\vdots& & }\\
& \xymatrix{0\ar[r]& H_{b-1}^r\ar[r]& H_b^r\ar[r]&
H_b^r/H_{b-1}^r\ar[r]& 0.}
\end{align*}
such that, for each $r$ the quotients $H_k^r/H_{k-1}^r$ can be
decomposed in the following manner:
\[
H_k^r/H_{k-1}^r\cong\bigoplus_{\{q\in P\ \mid\ r-k=d_q\}}
\left(\R^{d_q}T \left(V_q\right)\otimes_{\K} \widetilde{H}_{k-1}
\left( \left(q,1_{\widehat{P}}\right);\K\right)\right).
\]

These short exact sequences split as $\K$-vector spaces so we obtain
the following result.

\begin{teo}\label{formula de Hochster: no me digas}
Let $\K$ be a field, let $A$ be a commutative Noetherian ring
containing $\K$, let $T$ and $\mathcal{T}$ be as in
\textbf{Setup} \ref{fucntores covariantes hacificados}, and let $P$ be any finite poset. Moreover, assume
that $V\in\Inv(\widehat{P},\cA)$ is acyclic with respect to the
limit, that there is a natural equivalence of functors $\lim_{p\in
P}\circ\mathcal{T}\cong T\circ\lim_{p\in P}$, that for any $p\in P$,
$\R^j T (V_p)=0$ up to a single value of $j$ (namely, $d_p$) and
that, for any $p\neq q$, $\Hom_A \left(\R^{d_p}T (V_p), \R^{d_q}T
(V_q)\right)=0$. Then, there is a $\K$-vector space isomorphism
\[
\R^j T\left(\lim_{p\in P} V_p\right)\cong\bigoplus_{q\in P} \R^{d_q}
T (V_q)^{\oplus M_{j,q}},
\]
where $M_{j,q}=\dim_{\K} \widetilde{H}^{j-d_q-1} ((q,1_{\widehat{P}});
\K).$
\end{teo}


Hochster's formula boils down to the case where
$A=\K[x_1,\cdots,x_n]$ is a polynomial ring,
$T=\Gamma_{\mathfrak{m}},$ $\mathfrak{m}$ is the graded maximal
ideal of $A$, and $V=A/[\ast]$ is the inverse system associated to a
Stanley-Reisner ring $A/I$. Of course, the result still holds true
for central arrangements of linear varieties satisfying that
$A/[*]$ is acyclic with respect to the limit. Brun, Bruns, R\"omer gave a
generalization of Hochster's formula in
\cite[4.1]{BrunBrunsRomer2007}. Their result is the specialization
of \textbf{Theorem} \ref{formula de Hochster: no me digas} to the case where
$T=\Gamma_{J}$ for some ideal $J\subseteq A$. Namely, the
decomposition can be written down in the following way:
\[
H_{J}^j \left(\lim_{p\in P} V_p\right)\cong\bigoplus_{q\in P}
H_{J}^{d_q} (V_q)^{\oplus M_{j,q}}.
\]

\vskip 2mm

The same game can be played for the homological spectral sequences
constructed in Section \ref{homological_typeII}. In \textbf{Theorem}
\ref{sucesion espectral y colapso todo en uno: segunda parte} we
produced sufficient conditions for their degeneration so we may
obtain the following decomposition result; the details of the proof
are left to the interested reader.

\begin{teo}\label{formula de Hochster: no me digas: segunda parte}
Under the assumptions of \textbf{Theorem} \ref{sucesion espectral y colapso
todo en uno: segunda parte}, there is a $\K$-vector space
isomorphism
\[
\R^j T\left(\lim_{p\in P} V_p\right)\cong\bigoplus_{q\in P} \R^{h_q}
T (V_q)^{\oplus m_{j,q}}.
\]
where $m_{j,q}=\dim_{\K} \widetilde{H}_{j-h_q-1} ((q,1_{\widehat{P}});
\K).$
\end{teo}
Before applying this result to obtain a certain decomposition
of deficiency modules, we want to state the following technical result, which
is a direct consequence of \textbf{Lemma} \ref{playing with attached primes} using local
duality.

\begin{lm}\label{playing with associated primes}
Let $(A,\mathfrak{m})$ be a commutative Noetherian local ring where there exists a canonical
module $\omega_A,$ (so, $A$ is a Cohen--Macaulay ring
which can be expressed as homomorphic image of a Gorenstein
local ring $B$ of dimension $d$) let $I_p,I_q$ be two
ideals of $A$ of finite projective dimension
such that $I_q\not\subseteq\mathfrak{p}$ for any
prime ideal $\mathfrak{p}\supseteq I_p$ such that
$\dim (A/\mathfrak{p})=d_p.$ Then,
$\Hom_A (\Ext_A^{d-d_q} (A/I_q,\omega_A),
\Ext_A^{d-d_p} (A/I_p,\omega_A))=0.$
\end{lm}

Now, we want to single out the following interesting particular
case:

\begin{ex}\label{deficiency modules decomposition}
Suppose that $A$ is a commutative Noetherian local ring containing a field $\K$
where there exists a canonical module $\omega_A$ (remember that this
assumption implies that $A$ is a Cohen--Macaulay ring which can be expressed
as homomorphic image of a local Gorenstein ring $B$ of dimension $d$);
given any finitely generated $A$-module $N$, its $j$th
\emph{deficiency module} can be defined as $$K^j (N):=\Ext_A^{d-j} (N,\omega_A).$$ Now, assume that $P$
is the poset given by a decomposition of an ideal $I\subseteq A$ as
usual such that $A/[*]$ is acyclic with respect to the limit, and
that $\lim_{p\in P} A/I_p$ is a finitely generated
$A$-module, such that, for any $p\in P,$ $A/I_p$
is a Cohen--Macaulay ring with finite projective dimension, and
such that, for any $p\neq q,$ $I_q\not\subseteq\mathfrak{p}$ for any
prime ideal $\mathfrak{p}\supseteq I_p$ such that
$\dim (A/\mathfrak{p})=d_p.$ Then, the decomposition obtained in \textbf{Theorem}
\ref{formula de Hochster: no me digas: segunda parte} boils down to
the following isomorphism of $\K$-vector spaces:
\[
K^j \left(\lim_{p\in P} A/I_p\right)\cong\bigoplus_{q\in P} K^{\dim
(A/I_q)} (A/I_q)^{\oplus m_{j,q}}.
\]
Indeed, this is a consequence of \textbf{Lemma} \ref{playing with associated primes}
jointly with the fact that, since, for any $p\in P,$ $A/I_p$
is Cohen--Macaulay, one has that $K^j (A/I_p)=0$ up to a single
value of $j.$
\end{ex}
Finally, we can also obtain a similar result for our starting
homological spectral sequences; namely:

\begin{teo}\label{formula de Terai: no me digas: segunda parte}
Let $A$ be a commutative Noetherian ring containing a field $\K$, let
$I\subseteq A$ be an ideal, let $I=I_1\cap\ldots\cap I_n$ be its primary
decomposition, let $P$ be the poset given by all the possible different
sums of the ideals $I_k$'s ordered by reverse inclusion, let $T_{[*]}$ be
the functor of \textbf{Setup \ref{las hipotesis del caso homologico}},
and let $M$ be an $A$-module such that, for any $p\in P$, $\R^j T_p (M)=0$ up
to a unique value of $j$ (namely, $h_p$) and such that, for any pair of elements
$p<q$, $\Hom_{A}\left(\R^{h_p} T_p (M),\R^{h_q} T_p (M)\right)=0.$ Then, there is a
$\K$-vector space isomorphism
\[
\R^j T(M)\cong\bigoplus_{q\in P} \R^{h_q}
T_q (M)^{\oplus m_{j,q}},
\]
where $m_{j,q}=\dim_{\K}\widetilde{H}_{h_q-j-1} \left((q,1_{\widehat{P}});\K\right).$
\end{teo}
Musta{\c{t}}{\v{a}}--Terai's formula boils down to the case where
$A=\K[x_1,\cdots,x_n]$ is a polynomial ring,
$T=\Gamma_{I},$ and $I$ is a Stanley--Reisner ideal.

\subsection{Some additional structures to Hochster type decompositions}

For the case of Stanley-Reisner rings, Enescu and Hochster
\cite[5.1]{EnescuHochster2008} gave an additional Frobenius
structure to Hochster's formula. Moreover, Brun, Bruns and
R\"omer \cite{BrunBrunsRomer2007} gave a $\Z^d$-graded structure to
their main decomposition result. Finally, Terai also gave a $\Z^d$-graded
structure to his formula.

The aim of this section is to extend the decompositions, as $\K$-vector spaces, obtained in \textbf{Theorem}
\ref{formula de Hochster: no me digas} and \textbf{Theorem} \ref{formula de Terai: no me digas: segunda parte}
to accommodate these extra structures.

\subsubsection{Additional graded structure}
In order to recover and extend the $\Z^d$-graded main result of
\cite{BrunBrunsRomer2007}, we need to review first a technical fact; indeed, given
a field $\K$ and a group $G$, denote by
$\hbox{}^*\operatorname{vec}_{\K}$ the category
of $G$--graded $\K$--vector spaces such that, for any
$V\in\hbox{}^*\operatorname{vec}_{\K},$ and for any $g\in G,$ $V_g$
is a finite--dimensional $\K$--vector space. The below result
was proved by Raicu when $\K=\C$ and $G=\GL_m (\C)\times
\GL_n (\C);$ however, since his proof works verbatim also
in this setting, we refer to \cite[Lemma 2.8]{Raicu2017}
for details.

\begin{lm}\label{semisimplicity of graded vector spaces}
$\hbox{}^*\operatorname{vec}_{\K}$ is semisimple.
\end{lm}

Now, we are ready to establish the main result of this part, whose
proof boils down to the fact that any short exact sequence splits in
a semisimple category.

\begin{teo}\label{en categorias semisimples hay formula}
Let $\K$ be a field, let $A$ be a commutative Noetherian ring
containing $\K$, let $T$ and $\mathcal{T}$ be as in
\textbf{Setup} \ref{fucntores covariantes hacificados}, and let $P$ be any finite poset. Moreover, assume
that $V\in\Inv(\widehat{P},\cA)$ is acyclic with respect to the
limit, that there is a natural equivalence of functors $\lim_{p\in
P}\circ\mathcal{T}\cong T\circ\lim_{p\in P}$, that for any $p\in P$,
$\R^j T (V_p)=0$ up to a single value of $j$ (namely, $d_p$) and
that, for any $p\neq q$, $\Hom_A \left(\R^{d_p}T (V_p), \R^{d_q}T
(V_q)\right)=0$. Then, if $\mathcal{B}$ is an abelian, semisimple
category containing $\cA$ such that
$T(\mathcal{B})\subseteq\mathcal{B}$, and if
$V\in\Inv(\widehat{P},\mathcal{B})$, then there is an isomorphism
\[
\R^j T\left(\lim_{p\in P} V_p\right)\cong\bigoplus_{q\in P} \R^{d_q}
T (V_q)^{\oplus M_{j,q}}
\]
in the category $\mathcal{B},$ where $M_{j,q}=\dim_{\K} \widetilde{H}^{j-d_q-1} ((q,1_{\widehat{P}});
\K).$
\end{teo}
We can also write down the corresponding statement for our Musta{\c{t}}{\v{a}}--Terai type formulas;
namely:

\begin{teo}\label{formula de Terai: no me digas: tercera parte}
Let $A$ be a commutative Noetherian ring containing a field $\K$, let
$I\subseteq A$ be an ideal, let $I=I_1\cap\ldots\cap I_n$ be its primary
decomposition, let $P$ be the poset given by all the possible different
sums of the ideals $I_k$'s ordered by reverse inclusion, let $T_{[*]}$ be
the functor of \textbf{Setup \ref{las hipotesis del caso homologico}},
and let $M$ be an $A$-module such that, for any $p\in P$, $\R^j T_p (M)=0$ up
to a unique value of $j$ (namely, $h_p$) and such that, for any pair of elements
$p<q$, $\Hom_{A}\left(\R^{h_p} T_p (M),\R^{h_q} T_p (M)\right)=0.$ Then, if $\mathcal{B}$ is an abelian, semisimple
category containing $\cA$ such that
$T_p(\mathcal{B})\subseteq\mathcal{B}$ for any $p\in P$, and if
$M$ is an object of $\mathcal{B}$, then there is an isomorphism
\[
\R^j T(M)\cong\bigoplus_{q\in P} \R^{h_q}
T_q (M)^{\oplus m_{j,q}}
\]
in the category $\mathcal{B},$ where $m_{j,q}=\dim_{\K}\widetilde{H}_{h_q-j-1} \left((q,1_{\widehat{P}});\K\right).$
\end{teo}

As an immediate consequence of \textbf{Lemma} \ref{semisimplicity of graded vector spaces}, \textbf{Theorem} \ref{en categorias semisimples hay formula}, and
the fact (that we already proved) that the inverse system defining
a toric face ring is acyclic with respect to the limit functor, we obtain the following:

\begin{teo}[\textbf{Hochster decomposition for toric face rings}]\label{BBR graduado}
Let $\Sigma\subseteq\R^d$ be a rational pointed fan, let
$\mathcal{M}_{\Sigma}$ be a Cohen-Macaulay monoidal complex
supported on $\Sigma$, let $\K$ be a field, and let $\K
[\mathcal{M}_{\Sigma}]$ be the corresponding toric face ring with
$\mathfrak{m}$ as unique graded maximal ideal. Then, there is an
isomorphism of $\Z^d$-graded $\K$-vector spaces
\[
H_{\mathfrak{m}}^i \left(\K
[\mathcal{M}_{\Sigma}]\right)\cong\bigoplus_{C\in\Sigma}
H_{\mathfrak{m}}^{d_C} \left(\K [M_C]\right)^{\oplus M_{i,C}},
\]
where $d_C:=\dim (\K [M_C])$ and $M_{i,C}=\dim_{\K} \widetilde{H}^{i-d_C-1} ((p_C,1_{\widehat{P}});
\K).$
\end{teo}

\begin{rk}
The reader will easily note that \textbf{Theorem} \ref{BBR graduado} recovers
and extends \cite[1.3]{BrunBrunsRomer2007}; it is worth to point out that
this formula is still valid under the assumption that $\K
[\mathcal{M}_{\Sigma}]$ is seminormal \cite[4.5]{HopNguyen2012}. However, notice
that the seminormal case is not covered by our formalism.
\end{rk}

Also as consequence of \textbf{Lemma} \ref{semisimplicity of graded vector spaces} and \textbf{Theorem} \ref{en categorias semisimples hay formula}, we
recover the Takayama's type decomposition for the local cohomology
of monomial ideals obtained by Brun and R\"omer in
\cite[Corollary 2.3]{BrunRomer2008}; namely:

\begin{teo}[\textbf{Takayama decomposition for monomial ideals}]\label{takayama graduado}
Let $\Delta$ be a simplicial complex of $d$ vertices, let $\K$ be any
field, and let $I\subseteq\K [x_1,\ldots, x_d]$ be a monomial ideal
such that $\sqrt{I}$ is given by $\Delta$ through the Stanley correspondence. Then, there
is an isomorphism of $\Z^d$-graded $\K$-vector spaces
\[
H_{\mathfrak{m}}^i \left(\K
[x_1,\ldots ,x_d]/I\right)\cong\bigoplus_{F\in\Delta}
H_{\mathfrak{m}}^{d_F} \left(\K [x_1,\ldots,x_d]/I_p\right)^{\oplus M_{i,F}},
\]
where $F\subset\Delta$ is the face determined by $\sqrt{I_p},$ $d_F:=\dim (\K [x_1,\ldots,x_d]/I_p),$ and, finally,
$M_{i,F}=\dim_{\K} \widetilde{H}^{i-d_F-1} ((p_F,1_{\widehat{P}});
\K).$
\end{teo}

\begin{proof}
It is known \cite[Theorem 1.3.1 and Corollary 1.3.2]{monomialideals}
that any monomial ideal admits a minimal primary decomposition
made up by ideals generated by pure powers of the variables; so, let
$I=I_1\cap\ldots\cap I_n$ be such a minimal primary decomposition, and
consider $P$ as the poset attached to it. Then it is clear that, for
each $p\in P$, $A/I_p$ is a Cohen--Macaulay ring and that, for
each $p\neq q,$ $I_q$ is not contained in any minimal prime
of $I_p.$ Therefore, our claimed decomposition follows
immediately combining \textbf{Theorem} \ref{sucesion espectral y colapso todo en uno de cohomologia local}
jointly with \textbf{Lemma} \ref{semisimplicity of graded vector spaces} and \textbf{Theorem} \ref{en categorias semisimples hay formula}.
\end{proof}


We can also obtain a Hochster type decomposition for the local
cohomology modules attached to some affine, central subspace arrangements
of linear varieties over a field $\K$; namely:

\begin{teo}[\textbf{Hochster decompostion for some central arrangements}]\label{Hochster decomposition for subspace arrangements}
Let $\K$ be a field, and let $I\subseteq\K [x_1,\ldots ,x_d]$ be the
vanishing ideal of a central arrangement of linear varieties in
$\K^d$. Moreover, let $P$ be the poset given by all the possible
non-empty intersections of subarrangements ordered by
inclusion; for each $p\in P$, denote by $I_p$ the corresponding
vanishing ideal, and suppose that the inverse system
$(\K [x_1,\ldots
,x_d]/I_p)_{p\in P}$ is acyclic with respect to the limit. Then, there is an isomorphism of $\Z$-graded $\K$-vector spaces
\[
H_{\mathfrak{m}}^i \left(\lim_{p\in P}\K [x_1,\ldots
,x_d]/I_p\right)\cong\bigoplus_{p\in P} H_{\mathfrak{m}}^{d_p}
\left(\K [x_1,\ldots ,x_d]/I_p\right)^{\oplus M_{i,p}},
\]
where $d_p=\dim\left(\K [x_1,\ldots ,x_d]/I_p\right)$ and
$M_{i,p}=\dim_{\K} \widetilde{H}^{i-d_p-1} ((p,1_{\widehat{P}});
\K).$
\textbf{If, in addition, there is an isomorphism}
\[
\lim_{p\in P}\K [x_1,\ldots,x_d]/I_p\cong\K [x_1,\ldots,x_d]/I,
\]
(\textbf{see Remarks \ref{it doesn't work for arrangements} and \ref{it doesn't work for arrangements 2}}),
then
\[
H_{\mathfrak{m}}^i \left(\K [x_1,\ldots
,x_d]/I\right)\cong\bigoplus_{p\in P} H_{\mathfrak{m}}^{d_p}
\left(\K [x_1,\ldots ,x_d]/I_p\right)^{\oplus M_{i,p}}.
\]
\end{teo}

In case of central arrangements, we can deduce immediately from Theorem \ref{Hochster decomposition for subspace arrangements}
the corresponding $\Z$-graded Hilbert series; namely:

\begin{teo}[\textbf{Hilbert series of local cohomology for some central arrangements}]\label{Hochster decomposition for subspace arrangements 2}
Preserving the assumptions and notations of Theorem \ref{Hochster decomposition for subspace arrangements},
one has that
\[
H\left(H_{\mathfrak{m}}^i \left(\lim_{p\in P}\K [x_1,\ldots
,x_d]/I_p\right);t\right)=\sum_{p\in P}
\frac{\dim_{\K} \widetilde{H}^{i-d_p-1} ((p,1_{\widehat{P}});
\K)}{(t-1)^{d_p}},
\]
where $H(-;t)$ denotes the $\Z$-graded Hilbert series.
\textbf{If, in addition, there is an isomorphism}
\[
\lim_{p\in P}\K [x_1,\ldots,x_d]/I_p\cong\K [x_1,\ldots,x_d]/I,
\]
(\textbf{see Remarks \ref{it doesn't work for arrangements} and \ref{it doesn't work for arrangements 2}}),
then
\[
H\left(H_{\mathfrak{m}}^i \left(\K [x_1,\ldots
,x_d]/I\right);t\right)=\sum_{p\in P}
\frac{\dim_{\K} \widetilde{H}^{i-d_p-1} ((p,1_{\widehat{P}});
\K)}{(t-1)^{d_p}}.
\]
\end{teo}

\begin{proof}
The additivity of Hilbert series on short exact sequences, combined with Theorem \ref{Hochster decomposition for subspace arrangements},
implies that
\[
H\left(H_{\mathfrak{m}}^i \left(\lim_{p\in P}\K [x_1,\ldots
,x_d]/I_p\right);t\right)=\sum_{p\in P}\dim_{\K} \widetilde{H}^{i-d_p-1} ((p,1_{\widehat{P}});
\K)H\left(H_{\mathfrak{m}}^{d_p}
\left(\K [x_1,\ldots ,x_d]/I_p\right)\right),
\]
so it is enough to calculate $H\left(H_{\mathfrak{m}}^{d_p}
\left(\K [x_1,\ldots ,x_d]/I_p\right)\right)$; indeed, fix $p\in P$. Since $I_p$ is a prime ideal generated by
$\K$-linearly independent linear forms, it is straightforward to check that
\[
H\left(H_{\mathfrak{m}}^{d_p}
\left(\K [x_1,\ldots ,x_d]/I_p\right)\right)=\left(
\frac{1/t}{1-1/t}\right)^{d_p}=
\frac{1}{(t-1)^{d_p}},
\]
just what we finally wanted to prove.
\end{proof}

Finally, we can also write down our Musta{\c{t}}{\v{a}}--Terai type formula for certain toric face rings;
more precisely:

\begin{teo}[\textbf{Musta{\c{t}}{\v{a}}--Terai decomposition for toric face rings}]\label{Terai graduado}
Let $\Sigma\subseteq\R^d$ be a rational pointed fan, let
$\mathcal{M}_{\Sigma}$ be a Cohen-Macaulay monoidal complex
supported on $\Sigma$, let $\K$ be a field \textbf{of prime characteristic}, and let $\K
[\mathcal{M}_{\Sigma}]$ be the corresponding toric face ring with
$I=I_{\Sigma}$ as its defining ideal inside the polynomial ring $A=\K[x_1,\ldots,x_d].$
Then, there is an isomorphism of $\Z^d$-graded $\K$-vector spaces
\[
H_{I}^i (A)\cong\bigoplus_{C\in\Sigma}
H_{I_C}^{h_C} (A)^{\oplus m_{i,C}},
\]
where $h_C=d-d_C$, $d_C:=\dim (\K [M_C])$ and $m_{i,C}=\dim_{\K} \widetilde{H}_{h_C-i-1} ((p_C,1_{\widehat{P}});
\K).$
\end{teo}

\subsubsection{Application to the regularity of some central arrangements}
The goal of this part is to use the decomposition obtained in Theorem \ref{Hochster decomposition for subspace arrangements}
to calculate the (Castelnuovo--Mumford) regularity of \textbf{certain} arrangement of linear varieties
(see Theorem \ref{thank you, Jack}); this leads to an alternative proof, \textbf{specific for this
kind of arrangements}, of the so--called \textit{Subspace Arrangements Theorem} (see Theorem \ref{subspace arrangements theorem}), originally
obtained by Derksen and Sidman in \cite[Theorem 2.1]{DerksenSidman2002}.

Before doing so, we need to prove some auxiliary results; next lemma is an
immediate consequence of Remark \ref{rank and dimension remark} and its proof is left
to the interested reader.

\begin{lm}\label{anulacion de grupos de cohomologia}
Preserving the assumptions and notations of Theorem \ref{Hochster decomposition for subspace arrangements},
let $n$ be the number of irreducible components of the arrangement. Then, the following assertions
hold.

\begin{enumerate}[(i)]

\item For any $p\in P$, $\rank ((p,1_{\widehat{P}}))\leq n-1.$

\item Given $p\in P$ and $i\geq 0$ such that $i-d_p-1\geq n,$ $\widetilde{H}^{i-d_p-1} ((p,1_{\widehat{P}});
\K)=0.$

\end{enumerate}

\end{lm}

We also need to review the following notion of $\Z$--graded vector spaces.

\begin{df}\label{the end of a integer graded vector space}
Let $k$ be any field, and let $E$ be a $\Z$--graded $k$--vector space; set
\[
\findeg (E):=\max\{j\in\Z :\ E_j\neq 0\}.
\]
\end{df}
The last preliminary result we need is the behavior of the $\findeg$ of a $\Z$--graded
vector space with respect to finite direct sums; the proof is left to the interested reader.

\begin{lm}\label{end and direct sums}
Let $I$ be a finite index set, let $k$ be any field, let $\{E_i\}_{i\in I}$ be a family of
$\Z$--graded $k$--vector spaces, and set
\[
E:=\bigoplus_{i\in I}E_i.
\]
Then, $\findeg (E)=\sup_{i\in I}\{\findeg (E_i)\}.$
\end{lm}

Now, we are in position to state our main result about the regularity of arrangements of
linear varieties, which is the below:

\begin{teo}[\textbf{Regularity of an arrangement of linear varieties}]\label{thank you, Jack}
Let $\K$ be a field, and let $I\subseteq\K [x_1,\ldots ,x_d]$ be the
vanishing ideal of a central arrangement of linear varieties in
$\K^d$. Moreover, let $P$ be the poset given by all the possible
non-empty intersections of subarrangements ordered by
inclusion; for each $p\in P$, denote by $I_p$ the corresponding
vanishing ideal, and suppose that the inverse system
$(\K [x_1,\ldots
,x_d]/I_p)_{p\in P}$ is acyclic with respect to the limit. Then,
\[
\reg \left(\lim_{p\in P}\K [x_1,\ldots
,x_d]/I_p\right)=\max_{i\geq 0,\ p\in P} \{i-d_p:\ M_{i,p}\neq 0\},
\]
where $d_p=\dim\left(\K [x_1,\ldots ,x_d]/I_p\right)$ and
$M_{i,p}=\dim_{\K} \widetilde{H}^{i-d_p-1} ((p,1_{\widehat{P}});
\K).$
\end{teo}

\begin{proof}
It is known \cite[pages 58--59]{Eisenbud2005} that
\[
\reg \left(\lim_{p\in P}\K [x_1,\ldots
,x_d]/I_p\right)=\max_{i\geq 0} \{\findeg\left(H_{\mathfrak{m}}^i \left(\lim_{p\in P}\K [x_1,\ldots
,x_d]/I_p\right)\right)+i\}.
\]
In this way, combining Theorem \ref{Hochster decomposition for subspace arrangements} jointly with
Lemma \ref{end and direct sums} it follows that
\[
\reg \left(\lim_{p\in P}\K [x_1,\ldots
,x_d]/I_p\right)=\max_{i\geq 0,\ p\in P} \{\findeg\left( H_{\mathfrak{m}}^{d_p}
\left(\K [x_1,\ldots ,x_d]/I_p\right)\right)+i:\ M_{i,p}\neq 0\}.
\]
On the other hand, given any $p\in P$, \cite[Remark 3.1.6]{GotoWatanabe1978I} implies that
\[
\findeg\left(H_{\mathfrak{m}}^{d_p}
\left(\K [x_1,\ldots ,x_d]/I_p\right)\right)=-d_p.
\]
Summing up, combining all the foregoing equalities one finally obtains that
\[
\reg \left(\lim_{p\in P}\K [x_1,\ldots
,x_d]/I_p\right)=\max_{i\geq 0,\ p\in P} \{i-d_p:\ M_{i,p}\neq 0\},
\]
and the proof is therefore completed.
\end{proof}

\begin{rk}\label{subspace arrangements theorem}
As immediate consequence of Theorem \ref{thank you, Jack} and Lemma \ref{anulacion de grupos de cohomologia},
we obtain a new proof, \textbf{specific for arrangements satisfying $A/I\cong
\lim_{p\in P} A/I_p$ and that $A/[*]$ is acyclic with respect to the
limit}, of the so--called \textbf{subspace arrangements theorem}, originally proved
by Derksen and Sidman in \cite[Theorem 2.1]{DerksenSidman2002} (see also
\cite[Theorem 4.19]{Eisenbud2005}) for any central arrangement. This reflects the fact that, in general, the
regularity of an algebraic variety is not determined by its
intersection lattice.
\end{rk}


\subsubsection{Additional Frobenius structure}

Finally,  we are ready to provide the announced generalization of
\cite[5.1]{EnescuHochster2008}; namely:

\begin{teo}\label{formula de Enescu-Hochster}
Under the assumptions of \textbf{Theorem} \ref{BBR graduado}, if $\K$ is a
field of prime characteristic $p$, there there is a $\K
\left[\Theta; F\right]$-isomorphism
\[
H_{\mathfrak{m}}^i \left(\K
[\mathcal{M}_{\Sigma}]\right)\cong\bigoplus_{C\in\Sigma}
H_{\mathfrak{m}}^{d_C} \left(\K [M_C]\right)^{\oplus M_{i,C}}.
\]
Moreover, if $\mathcal{M}_{\Sigma}$ is $F$-rational, then every
$H_{\mathfrak{m}}^i \left(\K [\mathcal{M}_{\Sigma}]\right)$ is a
finite direct sum of simple $A [\Theta ;F]$--modules (where $A=\K
[\mathcal{M}_{\Sigma}]$) on which $F$ acts injectively. In particular (cf. \cite[Theorem 5.1]{EnescuHochster2008}),
if $(A_1,\mathfrak{m}_1)$ is either $\K [\mathcal{M}_{\Sigma}]_{\mathfrak{m}}$
or its completion, then $H_{\mathfrak{m}}^i \left(\K[\mathcal{M}_{\Sigma}]\right)$
can be identified with $H_{\mathfrak{m}_1}^i (A_1),$ and
$H_{\mathfrak{m}_1}^i (A_1)$ is a finite direct sum of simple
$A_1[\Theta ;F]$--modules on which $F$ acts injectively, and therefore
it has only a finite number of $F$--compatible submodules because
of \cite[Theorem 4.12]{EnescuHochster2008}.
\end{teo}

\begin{proof}
By \textbf{Example} \ref{accion del Frobenius a la formula de Hochster}, the
spectral sequence (respectively, the corresponding filtration
produced by its degeneration) can be regarded as spectral sequence
(respectively, filtration) in the category of left $\mathbb{F}_p \left[\Theta;
F\right]$-modules; moreover, it is also clear (by means of \textbf{Theorem}
\ref{formula de Hochster: no me digas}) that the claimed
decomposition holds in the category of $\K$-vector spaces. In this
way, it only remains to check that the Frobenius map preserves the
decomposition; in other words, that at both sides of such
decomposition the Frobenius acts in exactly the same way; the reader
will easily note that, hereafter, we follow so closely the argument
pointed out during the proof of \cite[5.1]{EnescuHochster2008}.

The first thing one has to ensure is that the action of $F$ is
$\K$-linear (otherwise, compatibility with the $\K$-vector space
structure would be impossible); with this purpose in mind, we have
to restrict our ground field of coefficients to $\mathbb{F}_p$
(here, we are using Fermat's Little Theorem). Notice that we can do
so without loss of generality; indeed, firstly, the multiplicities
$m_{i,C}$'s appearing in the decomposition are not affected by this
restriction of coefficients. Secondly, the action of $F$ on
$H_{\mathfrak{m}}^i \left(\K [\mathcal{M}_{\Sigma}]\right)$ can be
canonically identified with the one induced on the cohomology of the
complex $L^{\bullet} (\mathcal{M}_{\Sigma};\K)$ (labeled
$L^{\bullet} (\mathcal{M}_{\Sigma})$ in \cite[page
256]{IchimRomer2007}, see also \cite[Theorem 3.2]{IchimRomer2007}),
and this action on $L^{\bullet} (\mathcal{M}_{\Sigma};\K)$ is
obtained from the action on $L^{\bullet}
(\mathcal{M}_{\Sigma};\mathbb{F}_p)$ by applying
$\K\otimes_{\mathbb{F}_p} (-).$ Finally, for each cone $C\in\Sigma$,
the action of $F$ on $H_{\mathfrak{m}}^i \left(\K [M_C]\right)$ can
be canonically identified with the one induced on the cohomology of
the complex $L^{\bullet} (\K)$ (labeled $L^{\bullet}$ in \cite[page
267]{BrunsHerzog1993}, see also \cite[Theorem
6.2.5]{BrunsHerzog1993}), and this action on $L^{\bullet} (\K)$ is
obtained from the action on $L^{\bullet} (\mathbb{F}_p)$ by applying
$\K\otimes_{\mathbb{F}_p} (-).$

Summing up, we can suppose that $\K=\mathbb{F}_p;$ under this
assumption, the value of the action of $F$ on both sides of our
decomposition is the one which acts on a coset of the form $a\eta$
($a\in\K$) by $F(a\eta)=a\eta^p,$ which is what we want to check.

Finally, if $\mathcal{M}_{\Sigma}$ is $F$-rational, then for each
cone $C\in\Sigma$, $H_{\mathfrak{m}}^{\dim (C)} \left(\K
[M_C]\right)$ is a simple $A[\Theta ;F]$-module
\cite[2.6]{Smith1997}; the proof is therefore completed.
\end{proof}


\begin{rk}
Notice that the assumption of $F$--rationality on $\mathcal{M}_{\Sigma}$
required in Theorem \ref{formula de Enescu-Hochster} holds when, for instance,
$\mathcal{M}_{\Sigma}$ is a normal monoidal complex; indeed, under
this assumption, since for each cone $C\in\Sigma,$ $\K [M_C]$
is a normal affine monoid ring, one has that $\K [M_C]$ is a direct
summand of a Laurent polynomial ring \cite[Exercise 6.1.10]{BrunsHerzog1993} and
therefore it is $F$--regular by \cite[Proposition 4.12]{HoHun1990}. This
shows, in particular, that Theorem \ref{formula de Enescu-Hochster} extends
\cite[5.1]{EnescuHochster2008} to Stanley toric face rings.
\end{rk}

\section{Extension problems for cohomological spectral sequences: A Gr{\"a}be's type formula}
\label{section of Grabe formulae}

In the spirit of \cite[Section 3]{AlvarezGarciaZarzuela2003}, the
aim of this section is to focus on the study of the extension
problems attached to the corresponding filtrations produced by the
degeneration of our previously introduced homological spectral
sequences.

\subsection{Extension problems for homological spectral sequences associated to modules}\label{problemas de extension homologicos}


Consider the homological spectral sequence
$$E_2^{-i,j}=\Lf_i \colim_{p\in P} \R^j T_{[*]} (M)\xymatrix@1{\hbox{}\ar@{=>}[r]_-{i}& \R^{j-i} T (M)}$$
constructed in Section \ref{construccion de libros}. In \textbf{Theorem} \ref{extension del trabajo de master} we provided
necessary conditions for the spectral sequence to degenerate at the $E_2$-page.
Therefore, by Corollary \ref{habemus filtracion}, we obtain
the following collection of short exact
sequences (where $b\in\N$):
\begin{align*}
& \xymatrix{0\ar[r]& G_0^r \ar[r]& G_1^r \ar[r]& G_1^r /G_0^r \ar[r]& 0}\\
& \xymatrix{0\ar[r]& G_1^r \ar[r]& G_2^r \ar[r]& G_2^r /G_1^r \ar[r]& 0}\\
& \xymatrix{\vdots& \vdots& &\vdots& &\vdots& & }\\
& \xymatrix{0\ar[r]& G_{b-1}^r \ar[r]& G_b^r \ar[r]& G_b^r /G_{b-1}^r \ar[r]& 0.}\\
\end{align*}
where
\[
G_k^r/G_{k-1}^r\cong\bigoplus_{\{q\in P\ \mid\ k+r=h_q\}}
\R^{h_q} T_q (M)^{\oplus m_{k,q}},
\]
where $m_{k,q}:=\dim_{\K} (\widetilde{H}_{k-1} ((q,1_{\widehat{P}}); \K)).$ Hereafter, we omit the superscript $r;$ moreover, we have to point
out that, for any $k$,
\[
(s_k):\quad\xymatrix{0\ar[r]& G_{k-1}\ar[r]& G_k\ar[r]&
G_k/G_{k-1}\ar[r]& 0}
\]
may be regarded as an element of $\Ext_{\cA}^1 (G_k/G_{k-1},
G_{k-1})$. In order to study the corresponding extension problems
we may use the following mild generalization of \cite[Lemma of page
47]{AlvarezGarciaZarzuela2003}.


\begin{lm}\label{ya lo provaron Josep, Ricardo y Santi}
We assume, in addition, that $\Ext_{\cA}^1 (\R^{h_p} T_p (M),
\R^{h_q} T_q (M))=0$ provided $h_p\geq h_q+2$. Then, the natural
maps $\xymatrix@1{\Ext_{\cA}^1 (G_k/G_{k-1}, G_{k-1})\ar[r]&
\Ext_{\cA}^1 (G_k/G_{k-1}, G_{k-1}/G_{k-2})}$ are injective for all
$k\geq 2$.
\end{lm}

\begin{proof}
Consider the short exact sequence
\[
(s_{k-1}):\quad\xymatrix@1{0\ar[r]& G_{k-2}\ar[r]& G_{k-1}\ar[r]&
G_{k-1}/G_{k-2}\ar[r]& 0.}
\]
In this way, applying the functor $\Hom_{\cA}
(G_k/G_{k-1},-)$  to $(s_{k-1})$, one obtains the following exact sequence:
$\xymatrix@1{\Ext_{\cA}^1 (G_k/G_{k-1},G_{k-2})\ar[r]& \Ext_{\cA}^1
(G_k/G_{k-1}, G_{k-1})\ar[r]& \Ext_{\cA}^1
(G_k/G_{k-1},G_{k-1}/G_{k-2}).}$ So, applying once again $\Hom_{\cA}
(G_k/G_{k-1},-)$ to the short exact sequence $(s_l)$ for $l\leq k-2$
and descending induction, it turns out that we only need to check
that $\Ext_{\cA}^1 (G_k/G_{k-1},G_l/G_{l-1})=0$ for any $l\leq k-2$.
However, applying Corollary \ref{habemus filtracion} it is enough to
show that the group $\Ext_{\cA}^1 (\R^{h_p} T_p (M), \R^{h_q} T_q
(M))$ vanishes, where $h_q\leq j-2$ and $h_p =j$. But this vanishing
holds by assumption.
\end{proof}

\subsubsection{Mayer-Vietoris spectral sequence of local cohomology modules}
These extension problems were studied  in \cite{AlvarezGarciaZarzuela2003}
for the spectral sequence
$$E_2^{-i,j}=\Lf_i
\colim_{p\in P} H_{I_p}^j (A) \xymatrix@1{\hbox{}\ar@{=>}[r]_-{i}&
}H_I^{j-i}(A), $$ where $A=\K[x_1,\cdots,x_d]$ is the polynomial ring over a field
$\K$ and $I$ is a squarefree monomial ideal.
Namely, in the subcategory of $\mathbb{Z}^d$-graded modules introduced by Yanagawa \cite{Yanagawa2001}
under the notion of {\it straight} modules, these extension problems are non-trivial and
are described by the multiplication by the variables $x_i$. In particular this result
recovers the description given by Musta\c{t}\v{a} in \cite{Mustata2000symbolic}.

\vskip 2mm

We point out that, whenever our field $\K$ is of characteristic zero, the category of straight modules is
equivalent to the category of regular holonomic $D$-modules with variation zero
\cite[Section 4]{AlvarezGarciaZarzuela2003}, so these extension
problems are also non-trivial in this subcategory of $D$-modules.  In
this way, it is natural to ask whether these extension problems are
trivial in the category of $D$-modules; however, this is not the
case, as shown using  {\it Lyubeznik numbers} introduced in \cite{Lyubeznik1993Dmod}.

\begin{ex}\label{una de numeros de Lyubeznik}
Let $I:=(x,yz)\subset A:=\K [x,y,z]$, where $\K$ is any field. In
this case, if
\[
\xymatrix{0\ar[r]& H_{(x,y)}^2 (A)\oplus H_{(x,z)}^2 (A)\ar[r]&
H_I^2 (A)\ar[r]& H_{(x,y,z)}^3 (A)\ar[r]& 0,}
\]
was split, then one would obtain the following equality of Lyubeznik
numbers:
\[
\lambda_{1,1} (A/I)=1\neq 2=\lambda_{1,1} \left(\frac{\K
[x,y,z]}{(x,y)}\right)+\lambda_{1,1} \left(\frac{\K
[x,y,z]}{(x,z)}\right)+\lambda_{1,0} \left(\frac{\K
[x,y,z]}{(x,y,z)}\right),
\]
which is clearly false \cite{JosepAlireza2014}.
\end{ex}

\begin{rk}
The above example also shows that, when $\K$ is a field of prime
characteristic, these extension problems are also non-trivial in the
category of $F$-modules, because any $F$-module is, in particular, a
$D$-module (see \cite[Section 5]{Lyu1997} for details).
\end{rk}

\subsection{Extension problems for cohomological spectral sequences associated to inverse systems}\label{ristra de sucesiones exactas cortas}


Our next goal is to carry out a similar
business with the filtration produced in \textbf{Theorem} \ref{sucesion
espectral y colapso todo en uno}; indeed, under its assumptions
\textbf{Theorem} \ref{sucesion espectral y colapso todo en uno} provides
a spectral sequence
\[
E_2^{i,j}=\bigoplus_{j=d_q} \R^{d_q} T \left(V_q\right)^{\oplus
M_{i,q}}\xymatrix{ \ar@{=>}[r]_-{i}& }\R^{i+j}T \left(\lim_{p\in
P}V_p\right),
\]
where $M_{i,q}=\dim_{\K} \widetilde{H}^{i-d_q-1} ((q,1_{\widehat{P}});
\K).$ Moreover, we have a collection of short exact sequences (for some $b\in\N$)
\begin{align*}
& \xymatrix{0\ar[r]& H_0^r\ar[r]& H_1^r\ar[r]& H_1^r/H_0^r\ar[r]& 0}\\
& \xymatrix{0\ar[r]& H_1^r\ar[r]& H_2^r\ar[r]& H_2^r/H_1^r\ar[r]& 0}\\
& \xymatrix{\vdots& \vdots& &\vdots& &\vdots& & }\\
& \xymatrix{0\ar[r]& H_{b-1}^r\ar[r]& H_b^r\ar[r]&
H_b^r/H_{b-1}^r\ar[r]& 0.}
\end{align*}
where the quotients
$H_k^r/H_{k-1}^r$ can be decomposed in the following manner:
\[
H_k^r/H_{k-1}^r\cong\bigoplus_{\{q\in P\ \mid\ r-k=d_q\}}
\R^{d_q}T (V_q)^{\oplus M_{k,q}},
\]
where $M_{k,q}:=\dim_{\K} \left(\widetilde{H}^{k-d_q-1}
\left(\left(q,1_{\widehat{P}}\right);\K\right)\right).$
From now on, we omit the superscript $r$; in addition, the reader
should point out that, for each $k$,
\[
(h_k):\ \xymatrix{0\ar[r]& H_{k-1}\ar[r]& H_k\ar[r]&
H_k/H_{k-1}\ar[r]& 0}
\]
can be considered as a member of $\Ext_{\cA}^1
\left(H_k/H_{k-1},H_{k-1}\right)$.
Next result is just a reformulation of \textbf{Lemma} \ref{ya lo provaron
Josep, Ricardo y Santi} in this setup; since its proof is exactly
the same as the one of \textbf{Lemma} \ref{ya lo provaron Josep, Ricardo y
Santi}, we omit the details.

\begin{lm}\label{ya lo provaron Josep, Ricardo y Santi, o no}
We assume, in addition, that $\Ext_{\cA}^1 \left(\R^{d_p} T
\left(V_p\right), \R^{d_q} T \left(V_q\right)\right)=0$ provided
$d_p\geq d_q+2$. Then, the natural maps $\xymatrix@1{\Ext_{\cA}^1
(H_k/H_{k-1}, H_{k-1})\ar[r]& \Ext_{\cA}^1 (H_k/H_{k-1},
H_{k-1}/H_{k-2})}$ are injective for all $k\geq 2$.
\end{lm}

\subsubsection{Another spectral sequence of local cohomology modules}
From now on we will consider the spectral sequence
\[
E_2^{i,j}=\bigoplus_{j=d_q} H_{\mathfrak{m}}^{\dim (A/I_q)}
\left(A/I_q\right)^{\oplus M_{i,q}}\xymatrix{ \ar@{=>}[r]_-{i}&
}H_{\mathfrak{m}}^{i+j} \left(\lim_{p\in P} A/I_p\right).
\]
considered in \textbf{Theorem} \ref{sucesion espectral y colapso todo en uno de
cohomologia local}, degenerating at its $E_2$-sheet. In this situation,
the Ext  group considered in \textbf{Lemma} \ref{ya lo provaron Josep, Ricardo y Santi, o no} can
be decomposed into the following way:
\[
\Ext_{\cA}^1
(H_k/H_{k-1},H_{k-1}/H_{k-2})\cong\bigoplus_{\substack{r-k=\dim
(A/I_p)\\ r-k-1=\dim (A/I_q)}}\Ext_{\cA}^1
\left(H_{\mathfrak{m}}^{\dim (A/I_p)}
\left(A/I_p\right),H_{\mathfrak{m}}^{\dim (A/I_q)}
\left(A/I_q\right)\right).
\]

Keeping in mind this decomposition, it seems to us that it can be of
some interest to calculate Ext groups of the form $\Ext_A^1
(H_{\mathfrak{m}}^{d-t} (A/I_t), H_{\mathfrak{m}}^{d-t-1}
(A/I_{t+1})),$ where the ideal $I_t$ ($1\leq t\leq d-1$) is
generated by an $A$-regular sequence; this will be our next goal.
Indeed, we shall consider two different cases; when $A=\K
[\![x_1,\ldots, x_d]\!]$ (we refer to this situation as the
\emph{local case}), and the case where $A=\K [x_1,\ldots ,x_d]$
graded in a certain way which we specify more precisely later on (we
refer to this situation as the \emph{graded case}).

\vskip 2mm

Before doing so, we show that, in general, the extension
problems associated to the filtration produced by \textbf{Theorem}
\ref{sucesion espectral y colapso todo en uno de cohomologia local}
are non-trivial in the category of $A$-modules; we do so by means of
the below:

\begin{ex}\label{ejemplo de problemas de extension no triviales}
Let $\K$ be any field, and let
$I:=(x,yz)\subset\K[x,y,z]_{(x,y,z)}=:A;$ in this case, the
filtration produced by \textbf{Theorem} \ref{sucesion espectral y colapso
todo en uno de cohomologia local} boils down to the following short
exact sequence, where $\mathfrak{m}$ denotes the maximal ideal of
$A$:
\[
\xymatrix{0\ar[r]& H_{\mathfrak{m}}^0 (A/\mathfrak{m})\ar[r]&
H_{\mathfrak{m}}^1 (A/I)\ar[r]& H_{\mathfrak{m}}^1 (A/(x,y))\oplus
H_{\mathfrak{m}}^1 (A/(x,z))\ar[r]& 0.}
\]
So, if the short exact sequence was split, then $\Att (H_{\mathfrak{m}}^0
(A/\mathfrak{m}))\subseteq\Att (H_{\mathfrak{m}}^1 (A/I))$ by \cite[7.2.6]{BroSha}; however,
this inclusion is false, because \cite[7.3.2]{BroSha} implies that
$\Att (H_{\mathfrak{m}}^0 (A/\mathfrak{m}))=\{\mathfrak{m}\}$ and
$\Att (H_{\mathfrak{m}}^1 (A/I))=\{(x,y),(x,z)\}.$
\end{ex}

\vskip 2mm

\subsubsection{The local case} The main result of this part is the
following:
\begin{prop}\label{problemas de extension no triviales si no hay graduacion}
Let $\K$ be any field, set $A:=\K[\![x_1,\ldots ,x_d]\!]$ and let
$y_1,\ldots ,y_n$ be an $A$-regular sequence; moreover, for each
$1\leq t\leq n$, set $I_t :=\langle y_1,\ldots ,y_t\rangle$ and
$Q_t:=H_{\mathfrak{m}}^{d-t} (A/I_t)$. Finally, $(-)^{\vee}$ denotes
the Matlis duality functor $\Hom_A (-,E),$ where $E$ denotes a
choice of injective hull of $\K$ over $A$. Then, the following
statements hold.

\begin{enumerate}[(i)]

\item $\Hom_A (A/I_{t+1},A/I_t)=0$.

\item $\Hom_A (A/I_t,A/I_t)=A/I_t$.

\item $\Ext_A^1 (A/I_{t+1},A/I_t)=A/I_{t+1};$ more precisely, $\Ext_A^1 (A/I_{t+1},A/I_t)$
is a free $(A/I_{t+1})$--module of rank one that admits as generator
the class of the short exact sequence
\[
\xymatrix{0\ar[r]& A/I_t\ar[r]^-{\cdot y_{t+1}}& A/I_t\ar[r]&
A/I_{t+1}\ar[r]& 0.}
\]

\item $\Ext_A^1 \left(\left(0:_E I_t\right),\left(0:_E I_{t+1}\right)\right)= A/I_{t+1}$.

\item $\Ext_A^1 (Q_t, Q_{t+1})= A/I_{t+1}$.

\end{enumerate}
\end{prop}

\begin{proof}
Since $y_1,\ldots, y_n$ form an $A$--regular sequence, for
each $1\leq t\leq n,$ one has $(I_t:_A I_{t+1})=I_t,$ whence
\[
\Hom_A (A/I_{t+1},A/I_t)\cong\frac{(I_t:_A I_{t+1})}{I_t}=0,
\]
just what we firstly wanted to check.

On the other hand, consider the short exact sequence
\[
\xymatrix{0\ar[r]& A/I_t\ar[r]^-{\cdot y_{t+1}}& A/I_t\ar[r]&
A/I_{t+1}\ar[r]& 0.}
\]
In this way, applying to this short exact sequence the functor
$\Hom_A (-,A/I_t)$ one obtains the following exact one:
\[
\xymatrix{0\ar[r]& \Hom_A (A/I_{t+1},A/I_t)\ar[r]& \Hom_A
(A/I_t,A/I_t)\ar[r]^-{\cdot y_{t+1}}& \Hom_A (A/I_t,A/I_t)\\
\hbox{}\ar[r]& \Ext_A^1 (A/I_{t+1},A/I_t)\ar[r]& \Ext_A^1
(A/I_t,A/I_t)\ar[r]^-{\cdot y_{t+1}}&\Ext_A^1 (A/I_t,A/I_t).}
\]
Thus, since
\[
\Hom_A (A/I_t,A/I_t)\cong\frac{(I_t:_A I_t)}{I_t}=A/I_t
\]
and $\Hom_A (A/I_{t+1},A/I_t)=0$ one can rewrite the previous exact
sequence in the following way:
\[
0\longrightarrow A/I_t\stackrel{\cdot y_{t+1}}{\longrightarrow}
A/I_t\longrightarrow\Ext_A^1
(A/I_{t+1},A/I_t)\longrightarrow\Ext_A^1
(A/I_t,A/I_t)\stackrel{\cdot y_{t+1}}{\longrightarrow}\Ext_A^1
(A/I_t,A/I_t).
\]
Therefore, since the cokernel of $\xymatrix@1{A/I_t\ar[r]^-{\cdot
y_{t+1}}& A/I_t}$ is $A/I_{t+1}$ one obtains:
\begin{equation}\label{esta va a ser util}
0\longrightarrow A/I_{t+1}\longrightarrow\Ext_A^1
(A/I_{t+1},A/I_t)\longrightarrow\Ext_A^1
(A/I_t,A/I_t)\stackrel{\cdot y_{t+1}}{\longrightarrow}\Ext_A^1
(A/I_t,A/I_t).
\end{equation}
Now, we claim that $\Ext_A^1 (A/I_t,A/I_t)=(A/I_t)^{\oplus t}$;
indeed, $\Ext_A^1 (A/I_t,A/I_t)$ can be computed as $H^1 (\Hom_A
(K_{\bullet} (y_1,\ldots,y_t),A/I_t))$; that is, the first
cohomology group of the cochain complex obtained by applying the
functor $\Hom_A (-,A/I_t)$ to the Koszul resolution
$K_{\bullet} (y_1,\ldots,y_t)$ of $A/I_t$. Regardless, taking into
account the very definition of the Koszul complex, we know that all
the matrices which represent the differentials in $K_{\bullet}
(y_1,\ldots,y_t)$ have all their entries in $I_t$; thus, this single
fact implies that all the differentials of the cochain complex
$\Hom_A (K_{\bullet} (y_1,\ldots,y_t),A/I_t)$ vanish and therefore
one obtains that
\[
\Ext_A^1 (A/I_t,A/I_t)=H^1 (\Hom_A (K_{\bullet}
(y_1,\ldots,y_t),A/I_t))=\left(A/I_t\right)^{\oplus t}.
\]
In this way, bearing in mind this fact we can arrange the exact
sequence \eqref{esta va a ser util} in the following way:
\[
\xymatrix{0\ar[r]& A/I_{t+1}\ar[r]& \Ext_A^1
(A/I_{t+1},A/I_t)\ar[r]& \left(A/I_t\right)^{\oplus t}\ar[r]^-{\cdot
y_{t+1}}&\left(A/I_t\right)^{\oplus t}.}
\]
But the endomorphism on $\left(A/I_t\right)^{\oplus t}$ given by
multiplication by $y_{t+1}$ is injective; whence one finally obtains
that
\[
A/I_{t+1}\cong\Ext_A^1 (A/I_{t+1},A/I_t).
\]
In particular, part (iii) holds.

In addition, we have to point out that part (iv) follows combining
part (iii) together with \cite[3.4.14]{Strooker1990} and Matlis duality
in the following way: indeed, we have to notice that
\[
(A/I_{t+1})^{\vee}=\Ext_A^1 (A/I_{t+1},A/I_t)^{\vee}\cong\Tor_1^A
\left(A/I_{t+1},\left(A/I_t\right)^{\vee}\right)\cong\Tor_1^A
\left(\left(A/I_t\right)^{\vee},A/I_{t+1}\right)
\]
and therefore
$A/I_{t+1}\cong\left(A/I_{t+1}\right)^{\vee\vee}\cong\Tor_1^A
\left(\left(A/I_t\right)^{\vee},A/I_{t+1}\right)^{\vee}\cong\Ext_A^1
\left(\left(0:_E I_t\right),\left(0:_E I_{t+1}\right)\right),$
whence part (iv) also holds.

Finally, since $A/I_t$ is a complete intersection ring for any $t$,
it is, in particular, quasi Gorenstein. In this way, combining this
fact joint with part (iv) one has that
\[
\Ext_A^1 (Q_t, Q_{t+1})\cong\Ext_A^1 \left(\left(0:_E
I_t\right),\left(0:_E I_{t+1}\right)\right)= A/I_{t+1},
\]
just what we finally wanted to show.
\end{proof}

\vskip 2mm

\subsubsection{The graded case } Firstly, we want to review the
following notions \cite[Definition 4.1.6 and Definition
4.1.17]{KreuzerRobbiano2005}.

\begin{df}
Let $\K$ be a field, let $A$ be the polynomial ring $\K [x_1,\ldots
,x_d]$ and let $m\geq 1$ be an integer.

\begin{enumerate}[(i)]

\item Given a matrix $W\in\mathcal{M}_{m\times d}(\Z)$, we can consider the $\Z^m$-grading on $S$ for which $\K\subseteq A_0$ and the indeterminates are homogeneous elements whose degrees are given by the columns of $W$. In this case, it is said that \emph{$A$ is graded by $W$}. Moreover, we refer to the rows of $W$ as the \emph{weight vectors} of the indeterminates $x_1,\ldots ,x_d$.

\item Now, suppose that $A$ is graded by a matrix $W\in\mathcal{M}_{m\times d}(\Z)$ of rank $m$ and let $w_1,\ldots ,w_m$ be the weight vectors. It is said that the grading on $A$ given by $W$ is \emph{of positive type} provided there exist $a_1,\ldots ,a_m\in\Z$ such that all the entries of $a_1w_1+\ldots+a_m w_m$ are positive. In this case, it is also said that $W$ is a matrix \emph{of positive type}.

\end{enumerate}

\end{df}

\begin{ex}
We exhibit some examples of positive type matrices.

\begin{enumerate}[(a)]

\item The standard grading on $\Z$ (that is, $\deg (x_i)=1$ for all $i$) is given by
matrix $\begin{pmatrix} 1& \ldots & 1\end{pmatrix}$, which is clearly of positive type.

\item The standard $\Z^d$-grading on $S$ (that is, $\deg (x_i)=\mathbf{e}_i$, where $\mathbf{e}_i$ denotes
the element of $\Z^d$ which has all its components $0$ up to a $1$ in the $i$th position) is
given by $W=$ the identity matrix of size $d$. It is also clear in this case that $W$ is of positive type;
indeed, just take $a_i=1$ for all $i$ in the definition.

\end{enumerate}
\end{ex}

The reason for which we consider matrices of positive type is the
following result, which says that polynomial rings with gradings of
positive type and finitely generated graded modules over them have
finite dimensional homogeneous components. We omit its proof and
refer to \cite[Proposition 4.1.19]{KreuzerRobbiano2005} for details.

\begin{prop}
Let $\K$ be a field, let $A$ be the polynomial ring $\K [x_1,\ldots
,x_d]$ graded by a matrix $W\in\mathcal{M}_{m\times d}(\Z)$ of
positive type, and let $M$ be a finitely generated $W$-graded
$A$-module. Then, the following statements hold.

\begin{enumerate}[(a)]

\item We have $A_0=\K$.

\item For all $a\in\Z^m$, we have $\dim_{\K} (M_a)<+\infty$.

\end{enumerate}

\end{prop}

\begin{rk}
It is worth mentioning that the conclusion of the previous
proposition also works in greater generality; the interested reader
may like to consult \cite[Theorem 8.6]{MillerSturmfels2005} for
additional details.
\end{rk}

In this way, hereafter $\K$ will denote a field and $A$ will stand
for the polynomial ring $\K[x_1,\ldots ,x_d]$ graded by a positive
type matrix $W\in\mathcal{M}_{m\times d}(\Z)$. Moreover, let
$y_1,\ldots ,y_n$ be homogeneous elements of $A$ (with $\deg
(y_j)=D_j\in\Z^m$) which form an $A$-regular sequence and, for any
$1\leq t\leq n$, set $I_t:=\langle y_1,\ldots ,y_t\rangle$. On the
other hand, borrowing notation from \cite[1.5]{BrunsHerzog1993} (see
also \cite[13.1.8]{BroSha}) $\gHom_A (-,-)$ will stand for the
internal $\Hom$ in the category of $W$-graded $A$-modules, and set
$(-)^{\vee}:=\gHom_{\K} (-,\K)$.

Next result gives the $W$-graded analogue of
\cite[3.4.14]{Strooker1990}, which was already used
along the proof of Proposition \ref{problemas de extension no triviales si no hay graduacion}; albeit its proof is the adaptation in
this graded context of \cite[Proof of 3.4.14]{Strooker1990}, we
provide it for the convenience of the reader.

\begin{prop}\label{dualidad graduada entre ext y tor}
Let $j\in\Z$ and let $X_0,X_1$ be $W$-graded $A$-modules. Then, the
following statements hold.

\begin{enumerate}[(a)]

\item $\Tor_j^A (X_0,X_1)^{\vee}\cong\gExt_A^j (X_0,X_1^{\vee})$.

\item If, in addition, $X_0$ is finitely generated, then one has that
\[
\Tor_j^A (X_0,X_1^{\vee})\cong\gExt_A^j (X_0,X_1)^{\vee}.
\]

\end{enumerate}
In any case, both isomorphisms are canonical.
\end{prop}

\begin{proof}
Firstly, we prove part (a). Indeed, set $T^j$ and $U^j$ to be the
functors $\Tor_j^A (-,X_1)^{\vee}$ and $\gExt_A^j (-,X_1^{\vee})$.
Since $(-)^{\vee}$ is exact and contravariant, it follows that both
$(T^j)_{j\in\N}$ and $(U^j)_{j\in\N}$ form a positive strongly
connected sequences of contravariant functors. Moreover, it is well
known that
\[
\gHom_{\K} (X_0\otimes_A X_1,\K)\cong\gHom_A (X_0,\gHom_{\K}
(X_1,\K))
\]
for any $W$-graded $A$-module $X_0$; on the other hand, it is also
clear that $T^j P=0=U^j P$ for $j\geq 1$ and any
$\hbox{}^*$projective module $P$. Therefore, applying the
appropriate dual of \cite[13.3.5]{BroSha} one has that there exist
uniquely determined natural equivalences of functors
$\xymatrix@1{T^j\ar@{=>}[r]& U^j}$; whence part (a) follows directly
from this fact.

Finally, we prove part (b). In this case, we set $T^j$ and $U^j$ as
the functors $\Tor_j^A (-,X_1^{\vee})$ and $\gExt_A (-,X_1)^{\vee}$.
In this case, $(T^j)_{j\in\N}$ and $(U^j)_{j\in\N}$ both form a
positive strongly connected sequence of covariant functors. In
addition, for a finitely generated $W$-graded $A$-module $X_0$ one
has a canonical isomorphism
\[
X_0\otimes_A \gHom_{\K} (X_1,\K)\cong\gHom_{\K} (\gHom_A
(X_0,X_1),\K).
\]
Again, $T^j P=0=U^j P$ for $j\geq 1$ and any $\hbox{}^*$projective
module $P$. Therefore, applying the appropriate dual of
\cite[13.3.5]{BroSha} one has that there exist uniquely determined
natural equivalences of functors $\xymatrix@1{T^j\ar@{=>}[r]& U^j}$;
whence one has that part (b) also holds.
\end{proof}

The following statement can be regarded as the $W$-graded analogue
of \textbf{Proposition} \ref{problemas de extension no triviales si no hay
graduacion}.

\begin{prop}\label{problemas de extension no triviales si hay graduacion estandard}
Preserving the foregoing assumptions and notations, the following
statements hold.

\begin{enumerate}[(i)]

\item $\gHom_A (A/I_{t+1},A/I_t)=0$.

\item $\gHom_A (A,A)\cong A$.

\item $\gHom_A (A/I_t,A/I_t)\cong A/I_t$.

\item $\gHom_A \left(A/I_t,(A/I_t)(-D_{t+1})\right)\cong (A/I_t)(-D_{t+1})$.

\item $\gExt_A^1 (A/I_{t+1},(A/I_t)(-D_{t+1}))\cong A/I_{t+1};$ more precisely, $\gExt_A^1 (A/I_{t+1},(A/I_t)(-D_{t+1}))$
is a graded free $(A/I_{t+1})$--module of rank one that admits as generator
the class of the short exact sequence
\[
\xymatrix{0\ar[r]& (A/I_t)(-D_{t+1})\ar[r]^-{\cdot y_{t+1}}& A/I_t\ar[r]&
A/I_{t+1}\ar[r]& 0.}
\]

\item $\gExt_A^1 (\left(A/I_{t+1}\right)^{\vee},\left(A/I_t\right)^{\vee})\cong (A/I_{t+1})(D_{t+1})$.

\item $\gExt_A^1 (Q_t, Q_{t+1})\cong (A/I_{t+1})(D_{t+1})$, where $Q_t$ (respectively, $Q_{t+1}$) stands for the local cohomology module $H_{\mathfrak{m}}^{d-t} (A/I_t)$ (respectively, $H_{\mathfrak{m}}^{d-t-1} (A/I_{t+1})$).

\end{enumerate}

\end{prop}

\begin{proof}
First of all, we have to point out that, since $A/I_t$ is finitely
generated, one has that $\gHom_A (A/I_{t+1},A/I_t)$ is nothing but
$\Hom_A (A/I_{t+1},A/I_t)$ in case the grading is forgotten.
Regardless, we have checked in part (i) of \textbf{Proposition}
\ref{problemas de extension no triviales si no hay graduacion} that
$\Hom_A (A/I_{t+1},A/I_t)=0$; whence part (i) follows directly from
this fact.

Second, as $\gHom_A (A,A)$ (respectively, $\gHom_A (A/I_t,A/I_t)$)
are nothing but $\Hom_A (A,A)$ (respectively, $\Hom_A
(A/I_t,A/I_t)$) when the grading is forgotten, we obtain that both
part (ii) and (iii) hold. Moreover, we can also get part (iv) in the
below way:
\[
\gHom_A \left(A/I_t,(A/I_t)(-D_{t+1})\right)=\gHom_A
\left(A/I_t,A/I_t\right)(-D_{t+1})=(A/I_t)(-D_{t+1}).
\]
Now, consider the next short exact sequence of $W$-graded
$A$-modules and homogeneous homomorphisms:
\[
\xymatrix{0\ar[r]& (A/I_t) (-D_{t+1})\ar[r]^-{\cdot y_{t+1}}&
A/I_t\ar[r]& A/I_{t+1}\ar[r]& 0.}
\]
Applying to such short exact sequence the functor $\gHom_A
(-,(A/I_t) (-D_{t+1}))$ one obtains the following exact sequence of
$W$-graded $A$-modules and homogeneous homomorphisms:
\begin{align*}
& 0\longrightarrow\gHom_A (A/I_{t+1},(A/I_t)(-D_{t+1}))\longrightarrow\gHom_A (A/I_t,(A/I_t)(-D_{t+1}))\\
& \stackrel{\cdot y_{t+1}}{\longrightarrow} \gHom_A
((A/I_t)(-D_{t+1}),(A/I_t)(-D_{t+1}))\longrightarrow\gExt_A^1
(A/I_{t+1},(A/I_t)(-D_{t+1}))\\ & \longrightarrow\gExt_A^1
(A/I_t,(A/I_t)(-D_{t+1}))\stackrel{\cdot
y_{t+1}}{\longrightarrow}\gExt_A^1
((A/I_t)(-D_{t+1}),(A/I_t)(-D_{t+1})).
\end{align*}
We have that $\gHom_A
(A/I_{t+1},(A/I_t)(-D_{t+1}))=0$, $\gHom_A
(A/I_t,(A/I_t)(-D_{t+1}))=(A/I_t)(-D_{t+1})$  and  $\gHom_A ((A/I_t)(-D_{t+1}),(A/I_t)(-D_{t+1}))=A/I_t$.
In this way, we can rewrite the previous exact sequence in the next
way:
\begin{align*}
& 0\longrightarrow (A/I_t)(-D_{t+1})\stackrel{\cdot
y_{t+1}}{\longrightarrow} A/I_t\longrightarrow\gExt_A^1
(A/I_{t+1},(A/I_t)(-D_{t+1}))\\ & \longrightarrow\gExt_A^1
(A/I_t,(A/I_t)(-D_{t+1}))\stackrel{\cdot y_{t+1}}{\longrightarrow}
\gExt_A^1 ((A/I_t)(-D_{t+1}),(A/I_t)(-D_{t+1})).
\end{align*}
Moreover, since the cokernel of
$\xymatrix@1{(A/I_t)(-D_{t+1})\ar[r]^-{\cdot y_{t+1}}& A/I_t}$ is
$A/I_{t+1}$, it follows that we have the following exact sequence:
\begin{align}
& 0\longrightarrow A/I_{t+1}\longrightarrow\gExt_A^1
(A/I_{t+1},(A/I_t)(-D_{t+1}))\longrightarrow\gExt_A^1
(A/I_t,(A/I_t)(-D_{t+1}))\nonumber\\ & \stackrel{\cdot
y_{t+1}}{\longrightarrow}\gExt_A^1
((A/I_t)(-D_{t+1}),(A/I_t)(-D_{t+1}))\label{la llamaban hola3}.
\end{align}
Now, we claim that
\[
\gExt_A^1 (A/I_t, (A/I_t)(-D_{t+1}))\cong\bigoplus_{j=1}^t (A/I_t)
(D_j-D_{t+1}).
\]
Indeed, it is well known that $\gExt_A^1 (A/I_t, (A/I_t)(-D_{t+1}))$
is the first cohomology group of the complex $\gHom_A (K_{\bullet}
(y_1,\ldots ,y_t),(A/I_t) (-D_{t+1}))$, where $K_{\bullet}
(y_1,\ldots ,y_t)$ denotes the homological Koszul resolution of
$A/I_t$. However, since all the spots in
$\gHom_A (K_{\bullet} (y_1,\ldots ,y_t),(A/I_t) (-D_{t+1}))$ are
$A/I_t$-modules and all the differentials $\partial^i$ in such
cochain complex are represented by matrices with entries in $I_t$,
it follows that all these differentials are zero and therefore
\[
\gExt_A^1 (A/I_t, (A/I_t)(-D_{t+1}))=\ker
(\partial^1)=\bigoplus_{j=1}^t (A/I_t) (D_j-D_{t+1}).
\]
In addition, by similar reasons one also has that
\[
\gExt_A^1 ((A/I_t)(-D_{t+1}),
(A/I_t)(-D_{t+1}))\cong\bigoplus_{j=1}^t (A/I_t) (D_j).
\]
In this way, the map
\[
\xymatrix{\gExt_A^1 (A/I_t,(A/I_t)(-D_{t+1}))\ar[r]^-{\cdot
y_{t+1}}& \gExt_A^1 ((A/I_t)(-D_{t+1}),(A/I_t)(-D_{t+1}))}
\]
can be rewritten in the following way:
\[
\bigoplus_{j=1}^t (A/I_t) (D_j-D_{t+1})\xymatrix{ \ar[r]^-{\cdot
y_{t+1}}& }\bigoplus_{j=1}^t (A/I_t) (D_j).
\]
But this homomorphism is clearly injective. In this way, combining
this fact joint with \eqref{la llamaban hola3} one finally obtains
that
\[
A/I_{t+1}\cong\gExt_A^1 (A/I_{t+1},(A/I_t)(-D_{t+1}))=\gExt_A^1
(A/I_{t+1},A/I_t)(-D_{t+1}),
\]
whence part (v) holds too. The reader will easily note that the
righmost equality is well known \cite[14.1.10]{BroSha}.

Now, we can deduce part (vi) combining part (v) jointly with
\textbf{Proposition} \ref{dualidad graduada entre ext y tor} in the following
manner:
\begin{align*}
A/I_{t+1}&
\cong\left(\left(A/I_{t+1}\right)^{\vee}\right)^{\vee}\cong\left(\gExt_A^1
(A/I_{t+1},(A/I_t)(-D_{t+1}))^{\vee}\right)^{\vee}\\ & \cong\Tor_1^A
(A/I_{t+1},\left[(A/I_t)(-D_{t+1})\right]^{\vee})^{\vee}\cong\Tor_1^A
((A/I_t)^{\vee} (D_{t+1}),A/I_{t+1})^{\vee}\\ & \cong\gExt_A^1
(\left(A/I_t\right)^{\vee}(D_{t+1}),\left(A/I_{t+1}\right)^{\vee}).
\end{align*}
Finally, the graded local duality theorem
\cite[14.4.1]{BroSha} implies that $(A/I_t)^{\vee}(-c)\cong Q_t$
and $(A/I_{t+1})^{\vee}(-c)\cong Q_{t+1}$, where $c=c_1+\ldots +c_d$
and $c_1,\ldots ,c_d$ are the columns of matrix $W$. Whence
\[
\gExt_A^1 (Q_t, Q_{t+1})\cong\gExt_A^1
(\left(A/I_t\right)^{\vee}(-c),\left(A/I_{t+1}\right)^{\vee}(-c))\cong\gExt_A^1
(\left(A/I_t\right)^{\vee},\left(A/I_{t+1}\right)^{\vee}).
\]
But the leftmost term is isomorphic to $(A/I_{t+1})(D_{t+1})$; the
proof is therefore completed.
\end{proof}

\subsection{A Gr{\"a}be type description formula}
The results obtained in \textbf{Proposition} \ref{problemas de extension no triviales si
hay graduacion estandard} can be regarded as a Gr{\"a}be
type description formula (cf. \cite[Theorem 2]{Grabe1984}) because
they describe the $W$--graded structure of these local
cohomology modules by looking at these groups of extensions; more precisely, we have:


\begin{teo}\label{generalizamos la formula de Grabe}
Let $\K$ be a field, and let $A$ be the polynomial ring $\K
[x_1,\ldots ,x_d]$ graded by a positive type matrix
$W\in\mathcal{M}_{m\times d}\left(\Z\right)$. Moreover, for any
$1\leq t\leq d$ set $F_t$ as the face ideal of $A$ generated by
$x_1,\ldots ,x_t$ and $Q_t:=H_{\mathfrak{m}}^{d-t}
\left(A/F_t\right)$. Then, there is a canonical isomorphism
\[
\gExt_A^1\left(Q_t\left(c_{t+1}\right),Q_{t+1}\right)\cong\left(A/F_{t+1}\right),
\]
where $c_{t+1}$ denotes the $(t+1)$th column of matrix $W$; more
precisely, $\gExt_A^1\left(Q_t\left(c_{t+1}\right),Q_{t+1}\right)$
is a free, graded $(A/F_{t+1})$--module that admits as generator
the class of the short exact sequence
\[
\xymatrix{0\ar[r]& Q_{t+1}\ar[r]& Q_t\ar[r]^-{\cdot x_{t+1}}&
Q_t (c_{t+1})\ar[r]& 0.}
\]
\end{teo}

\begin{rk}
When $W$ is the identity matrix of size $d$, \textbf{Theorem}
\ref{generalizamos la formula de Grabe} describes the standard
$\Z^d$--graded structure of local cohomology modules of
Stanley--Reisner rings; in \cite[Theorem 2]{Grabe1984}, Gr{\"a}be
gave a nice topological interpretation of this structure by
identifying the multiplication by each variable on the graded
pieces of local cohomology modules in terms of certain
connecting maps between simplicial cohomology groups. In this way, by means
of Gr{\"a}be's formula one can also describe these groups of
extensions in terms of the corresponding simplicial complex. On the other hand, Miller
proved in his thesis \cite[Corollary 6.24]{Miller2000thesis}
that Gr{\"a}be's formula is equivalent to the one obtained by Musta{\c{t}}{\v{a}} in \cite[Theorem 2.1 and Corollary
2.2]{Mustata2000symbolic}; keeping in mind this fact, \textbf{Theorem}
\ref{generalizamos la formula de Grabe} can also be regarded as a
Musta{\c{t}}{\v{a}} type description formula.
\end{rk}

\section*{Acknowledgements}
The authors would like to thank Jack Jeffries, Wenliang Zhang and the anonymous referee for their comments on
an earlier draft of this manuscript. Part of this work was done when A.F.B. visited Northwestern University funded by the CASB fellowship program.

\bibliographystyle{alpha}
\bibliography{AFBoixReferences}

\end{document}